\documentclass{amsart}

\usepackage{graphics,graphicx}
\usepackage{tikz-cd}
\usetikzlibrary{calc}
\usetikzlibrary{arrows} 

\usepackage[utf8]{inputenc}
\usepackage[T1]{fontenc}

\usepackage{amsmath}
\usepackage{amsthm}
\usepackage{amssymb}

\usepackage{xspace}
\usepackage[english,french]{babel}

\usepackage[backref=page,colorlinks=true,allcolors=blue]{hyperref}

\newtheorem{thm}{Theorem}[section]
\newtheorem{cor}[thm]{Corollary}
\newtheorem{lem}[thm]{Lemma}
\newtheorem{prop}[thm]{Proposition}
\newtheorem*{question}{Question}
\newtheorem*{thma}{Theorem A}
\newtheorem*{thmb}{Theorem B}
\newtheorem*{thmc}{Theorem C}

\theoremstyle{definition}
\newtheorem{defn}[thm]{Definition}
\newtheorem{rem}[thm]{Remark}
\newtheorem{exa}[thm]{Example}

\begin{document}

\selectlanguage{english}

\title[K-Stability of Fano Spherical Varieties]{
K-Stability of Fano Spherical Varieties
\\
(K-Stabilité des variétés sphériques Fano)
}

\author{Thibaut Delcroix}

\address{D\'epartement de math\'ematiques et applications, 
\'Ecole Normale Sup\'erieure, CNRS, PSL Research University, 
75005 Paris, France}

\email{thibaut.delcroix@ens.fr}



\date{}

\maketitle

\begin{abstract}
We prove a combinatorial criterion for K-stability of a $\mathbb{Q}$-Fano 
spherical variety 
with respect to equivariant special test configurations, in terms of its moment 
polytope and some combinatorial data associated to the open orbit. Combined 
with the equivariant version of the Yau-Tian-Donaldson conjecture for Fano manifolds
proved by Datar and Székelyhidi, it yields a criterion for the existence of a 
Kähler-Einstein metric on a spherical Fano manifold. The results hold also for 
modified K-stability and existence of Kähler-Ricci solitons. 
\end{abstract}

\begin{otherlanguage}{french}
\begin{abstract}
Nous prouvons, pour une variété sphérique $\mathbb{Q}$-Fano, un critère 
combinatoire de K-stabilité par rapport aux configurations test spéciales 
équivariantes, exprimé en fonction de son polytope moment et d'une donnée 
combinatoire associée à l'orbite ouverte. En utilisant la version équivariante 
de la conjecture de Yau-Tian-Donaldson prouvée par Datar et Székelyhidi, 
cela devient un critère d'existence de métriques Kähler-Einstein sur les 
variétés sphériques Fano lisses. Les résultats s'appliquent également à 
la K-stabilité modifiée et à l'existence de solitons de Kähler-Ricci.
\end{abstract}
\end{otherlanguage}

\section*{Introduction}

Kähler-Einstein metrics on a Kähler manifold are the solutions (if they 
exist) of a highly non linear second order partial differential equation on 
the manifold. It is not clear at the moment under which conditions the 
equation admits solutions on a Fano manifold. 
In the recent years a major advance in this direction has been made through 
the resolution of the Yau-Tian-Donaldson conjecture in the Fano case, 
by Chen, Donaldson, and Sun \cite{CDS15a,CDS15b,CDS15c}. 
This conjecture states in its general form that the existence of some 
canonical metrics on a Kähler manifold should be related to the algebro-geometric 
condition of K-stability on the manifold. 

The K-stability condition is a condition involving the positivity of 
numerical invariants
associated to polarized one parameter degenerations of the manifold, equipped with an 
action of $\mathbb{C}^*$, called \emph{test configurations}. 
In the Fano case, it was proved by Li and Xu \cite{LX14} 
(and also Chen, Donaldson and Sun) 
that it is enough to consider test 
configurations with normal central fiber, which are called \emph{special} test 
configurations. 
Other proofs of the Yau-Tian-Donaldson conjecture were obtained by 
Tian \cite{Tia15}, Datar and Székelyhidi \cite{DS15}, Chen, Sun and Wang \cite{CSW}, 
Berman, Boucksom and Jonsson \cite{BBJ}.
The work of Datar and Székelyhidi is of special interest to us as it allows us to 
take into account automorphisms of the manifold, by considering only equivariant 
test configurations, and extends the result to 
Kähler-Ricci solitons. 
 
The necessity of the K-stability condition with respect to special test configurations 
was established earlier by Tian who 
provided an example of Fano manifold with vanishing Futaki invariant 
but no Kähler-Einstein metrics \cite{Tia97}. 
It was not clear at first if the K-stability condition could be used to prove 
the existence of Kähler-Einstein metrics on explicit examples of Fano manifolds. 
One aim of the present article is to provide an illustration of the power 
of the approach to the existence of Kähler-Einstein metrics \emph{via} K-stability, 
on highly symmetric manifolds. 

Namely, we obtain a criterion for the existence of Kähler-Einstein metrics 
on a Fano spherical manifold, involving only the moment polytope and the 
valuation cone of the spherical manifold. Both are classical and central 
objects in the theory of spherical varieties. 
The class of spherical varieties is a very large class of highly symmetric 
varieties, which contains toric varieties, generalized flag manifolds, 
homogeneous toric bundles, biequivariant compactifications of reductive groups.
For all of these subclasses for which a criterion was known, 
our result specializes to the same criterion (compare with \cite{WZ04,PS10,DelTh}).  
Examples of new varieties to which our criterion applies include colored 
horospherical varieties and symmetric varieties (for examples the varieties 
constructed in \cite{DCP83}).
Let us mention here the work of Ilten and Süss on Fano manifolds with an 
action of a torus with complexity one \cite{IS}, and also the discussion in  
\cite[Section 4]{DS15}, which were to the author's knowledge the first applications 
of K-stability as a sufficient criterion for explicit examples.

Another aim of this article is to provide a framework for understanding 
the K-stability condition in this class of spherical varieties, which should
lead to a better understanding of K-stability in general. 
The author obtained in \cite{DelTh,DelKE} an example of group compactification with 
no Kähler-Einstein metric but vanishing Futaki invariant, which is furthermore not 
K-semistable unlike the Mukai-Umemura type example from \cite{Tia97}.
This is evidence that non trivial K-stability phenomena appear in the class
of spherical Fano manifolds, which was not true for toric Fano manifolds.

Before stating the main result of the article, let us introduce some notations, 
the moment polytope and the valuation cone.

Let $G$ be a complex connected reductive algebraic group.
Let $B$ be a Borel subgroup of $G$ and $T$ a maximal torus of $B$.
Let $\mathfrak{X}(T)$ denote the group of algebraic characters of $T$.
Denote by $\Phi\subset \mathfrak{X}(T)$ the root system of $(G,T)$ and 
$\Phi^+$ the positive roots determined by $B$.

Let $X$ be a Fano manifold, spherical under the action of $G$, which means that 
$B$ acts on $X$ with an open and dense orbit.
The \emph{moment polytope} $\Delta^+\subset \mathfrak{X}(T)\otimes \mathbb{R}$ of $X$ 
with respect to $B$ is a polytope encoding the structure of representation of $G$ 
on the spaces of sections of tensor powers of the anticanonical line bundle.
Alternatively, from a symplectic point of view, it can be characterized as the 
Kirwan moment polytope of $(X,\omega)$ with respect to the action of a maximal 
compact subgroup $K$ of $G$, where $\omega$ is a $K$-invariant Kähler 
form in $c_1(X)$ (see \cite{Bri87}).  
The moment polytope $\Delta^+$ determines a sub-root system $\Phi_L$ of $\Phi$, 
composed of those roots that are orthogonal to the affine 
span of $\Delta^+$ with respect to the \emph{Killing form} $\kappa$. 
Let $\Phi_{Q^u}$ be the set $\Phi^+ \setminus \Phi_L$, and 
$2\rho_Q$ be the sum of the elements of $\Phi_{Q^u}$.

A spherical variety $X$ also has an open and dense orbit $O$ under the action of $G$.
The valuation cone of $X$ depends only on this open orbit $O$. 
Let $\mathcal{M}\subset \mathfrak{X}(T)$ be the set of characters of 
$B$-semi-invariant functions in the function field $\mathbb{C}(O)$ of $O$, and 
let $\mathcal{N}$ be its $\mathbb{Z}$-dual. 
The restriction of a $\mathbb{Q}$-valued valuation on $\mathbb{C}(O)$ to the 
$B$-semi-invariant functions defines an element of $\mathcal{N}\otimes \mathbb{Q}$.
The \emph{valuation cone} $\mathcal{V}$ with respect to $B$ is defined as 
the set of those elements of $\mathcal{N}\otimes \mathbb{Q}$ induced by 
$G$-invariant valuations on $\mathbb{C}(O)$.

Remark that the vector space $\mathcal{N}\otimes \mathbb{Q}$ is a quotient of 
the vector space $\mathfrak{Y}(T)\otimes \mathbb{Q}$, where $\mathfrak{Y}(T)$ is 
the group of algebraic one parameter subgroups of $T$. Denote by 
$\pi : \mathfrak{Y}(T)\otimes \mathbb{Q} \longrightarrow \mathcal{N} \otimes \mathbb{Q}$
the quotient map, so that $\pi^{-1}(\mathcal{V}) \subset \mathfrak{Y}(T)\otimes \mathbb{Q}$.
Let $\Xi \subset \mathfrak{X}(T)\otimes \mathbb{R}$ be the dual cone to the closure of 
the inverse image by $\pi$ of the opposite of the valuation cone  
$\pi^{-1}(-\mathcal{V})$ in $\mathfrak{Y}(T)\otimes \mathbb{R}$ (dual with respect to the 
extension of the natural 
pairing $\left<\cdot,\cdot\right>:\mathfrak{X}(T) \times \mathfrak{Y}(T) \longrightarrow \mathbb{Z}$).

The group of $G$-equivariant automorphisms $\mathrm{Aut}_G(X)$ of 
the spherical manifold $X$ is diagonalizable. 
The real vector space generated by the linear part of $\mathcal{V}$ 
is in fact isomorphic to $\mathfrak{Y}(\mathrm{Aut}_G(X))\otimes \mathbb{R}$. 

Given $\zeta$ in $\mathfrak{Y}(\mathrm{Aut}_G(X))\otimes \mathbb{R}$ identified with an element of 
$\mathcal{N}\otimes \mathbb{R}$, and a choice $\tilde{\zeta}\in \pi^{-1}(\zeta)$ of lift of $\zeta$, 
we denote by $\mathrm{bar}_{DH,\tilde{\zeta}}(\Delta^+)$ the barycenter of the polytope 
$\Delta^+$ with respect to the measure with density 
$p\mapsto e^{2\left<p-2\rho_Q,\tilde{\zeta}\right>}
\prod_{\alpha\in \Phi_{Q^u}} \kappa (\alpha,p)$ 
with respect to 
the Lebesgue measure $dp$ on $\mathfrak{X}(T) \otimes \mathbb{R}$.
Our main result is the following.

\begin{thma}
Let $X$ be a Fano spherical manifold. The following are equivalent.
\begin{enumerate}
\item There exists a Kähler-Ricci soliton on $X$ with associated holomorphic 
vector field $\zeta$.
\item The barycenter $\mathrm{bar}_{DH,\tilde{\zeta}}(\Delta^+)$ is in the 
relative interior of the cone $2\rho_Q + \Xi$.
\item The manifold $X$ is modified K-stable with respect to equivariant 
special test configurations.
\item The manifold $X$ is modified K-stable.
\end{enumerate} 
\end{thma}

The equivalence between (1) and (4) holds for any Fano manifold. In the setting of 
Kähler-Einstein metrics, it is the consequence on one hand of the work of Chen, 
Donaldson and Sun recalled earlier, and on the other hand of the work of Berman 
\cite{Ber16}. 
In the more general setting of Kähler-Ricci solitons, Berman and Witt-Nystrom 
\cite{BWN} showed that (4) is a necessary condition for (1). 
Datar and Székelyhidi \cite{DS15} showed that (3) implies (1) for any 
Fano manifold equipped with an action of a complex reductive group, and 
(4) clearly implies (3). 
What we prove in this article is the equivalence between (2) and (3) in the 
case of a spherical Fano manifold.
Furthermore we prove that the equivalence between (2) and (3) holds for 
singular $\mathbb{Q}$-Fano spherical varieties. 

The intuition for our main result came from our previous work on group compactifications, 
which did not involve K-stability.
The proof of a Kähler-Einstein criterion for smooth and Fano group compactifications
in \cite{DelTh,DelKE} can be adapted to provide another proof of the criterion for 
Kähler-Ricci solitons on the same manifolds. 
Similarly, Wang-Zhu type methods (as used in \cite{WZ04} and \cite{PS10}), together 
with some results proved for horospherical manifolds in the present paper, could 
be used to obtain the Kähler-Ricci soliton criterion for these manifolds. 
One advantage of this other approach is that the value of the greatest Ricci lower 
bound can be explicitly computed. 
Alternatively, for horospherical varieties at least, this quantity could be computed 
using \emph{twisted} modified K-stability (see \cite{DS15}). 

The computation of the K-stability of a manifold requires two ingredients.
The first one is a description of all test configurations (rather, using \cite{DS15}, 
all special equivariant test configurations), and the second one is a way to 
compute the Donaldson-Futaki invariant for all of these test configurations. 

The description of special equivariant test configurations
is obtained using the general theory of spherical varieties. 
Generalizing the fan of a toric variety, the \emph{colored fan} of a spherical variety
consists essentially of a fan subdividing the valuation cone together with 
additional data called \emph{colors}. 
The total space of a (normal) test configuration is still a spherical variety, 
and can thus be described by its colored fan. In the case of a special test 
configuration, the central fiber itself is a spherical variety and the 
action of $\mathbb{C}^*$ on the central fiber may be deduced from the 
colored fan of the test configuration. More precisely, we prove the following.

\begin{thmb}
Let $X$ be a spherical variety under the action of a reductive algebraic group $G$. 
Let $\mathcal{X}$ be an equivariant test configuration 
for $X$, with irreducible (scheme theoretic) central fiber $X_0$. Then 
\begin{itemize}
\item $X_0$ is spherical under the action of $G$,
\item to $\mathcal{X}$ is associated an element $\xi$ of the valuation cone such that 
for an appropriate choice of $H_0$ such that the open orbit of $G$ in $X$ is 
identified with $G/H_0$, 
the action of $e^{\tau}\in \mathbb{C}^*$ induced by the test configuration is given by:
\[
e^{\tau}\cdot gH_0=g\exp(-\tau \tilde{\xi})H_0
\]
for any lift $\tilde{\xi} \in \mathfrak{Y}(T)\otimes\mathbb{Q}$ considered as an element of 
the Lie algebra of $T$.
\end{itemize}
Furthermore, for any $\xi$ in the valuation cone, we may associate to $\xi$ an 
equivariant test configuration with irreducible central fiber,  
and there exists an integral multiple $m\xi$, $m\in \mathbb{N}$ such that the 
test configuration associated to $m\xi$ is \emph{special}, that is, $X_0$ is a normal 
(reduced) variety.
Finally, a special equivariant test configuration has central fiber 
$X_0$ isomorphic to $X$ if and only if the associated $\xi$ is in the 
linear part of the valuation cone.
\end{thmb}

Degenerations of spherical varieties, and moduli questions, have been studied 
by Alexeev and Brion \cite{AB06} (see also \cite{AB04SRVI,AB04SRVII} for the case of
reductive varieties). 
We adopt here a different approach to be better able to keep track of the actions 
of $\mathbb{C}^*$ on the degenerations. 
Remark that the central fiber of a normal and equivariant test configuration which is 
not special would be a stable spherical variety in the sense of \cite{AB06}. 
We do not study these test configurations here. 
Our description of the action of $\mathbb{C}^*$ on the central fiber of a special 
equivariant test configuration relies heavily on the work of Brion and Pauer \cite{BP87}
on elementary embeddings of spherical homogeneous spaces.  

The Donaldson-Futaki invariant of a test configuration depends only on the 
central fiber and the induced action of $\mathbb{C}^*$ on the central fiber.
Namely, it reduces to the Futaki invariant of the central fiber evaluated 
at the holomorphic vector field generating the action.
The basic idea behind our computation of these numerical invariants is that 
we may degenerate the central fiber even more in order to acquire more symmetries, 
then compute the Futaki invariant on the corresponding degeneration.

For spherical varieties, this idea leads to consider only the Futaki invariants
of horospherical varieties. Indeed, there always exist a test configuration with 
horospherical central fiber. The existence of a horospherical degeneration 
for spherical varieties is a classical result \cite{Bri86,Pop87}. 
In the notations of Theorem~B, 
the test configuration $\mathcal{X}$ has a horospherical central fiber if and 
only if $\xi$ is in the interior of the valuation cone. 
Horospherical varieties are the simplest among spherical varieties, but still 
form a large class containing toric varieties and homogeneous toric bundles. 
They should be considered as the "most symmetric" spherical varieties. 
In particular, even if it is possible to degenerate any spherical 
variety to a toric variety \cite{AB04Toric}, 
it is at the expense of some symmetries, as
it can be made equivariantly only with respect to a maximal torus.

Our computation of the modified Futaki invariant on $\mathbb{Q}$-Fano 
horospherical varieties gives the following statement, where we keep the 
notations introduced earlier.
\begin{thmc}
Let $X$ be a $\mathbb{Q}$-Fano horospherical variety, with moment polytope 
$\Delta^+$. Let $\zeta, \xi$ be two elements of 
$\mathfrak{Y}(\mathrm{Aut}_G(X))\otimes \mathbb{R}$ and let $\tilde{\zeta}$ 
and $\tilde{\xi}$ be choices of lifts in $\mathfrak{Y}(T)\otimes \mathbb{R}$.
Then the modified (with respect to $\zeta$) Futaki invariant of $X$ evaluated 
at $\xi$ is 
\[
\mathrm{Fut}_{X,\zeta}(\xi)=
C \left< 2\rho_Q-\mathrm{bar}_{DH,\tilde{\zeta}}(\Delta^+), \tilde{\xi} \right>
\]  
where $C$ is a positive constant independent of $\xi$.
\end{thmc}

This statement was first obtained by Mabuchi \cite{Mab87} for smooth toric manifolds. 
For smooth homogeneous toric bundles, this statement was obtained by Podesta and 
Spiro in \cite{PS10}. 
Let us also mention here the work of Alexeev and Katzarkov on K-stability of 
group compactifications \cite{AK05}. 
Our computation is based on an expression for the curvature form of a 
positive hermitian metric on a polarized horospherical homogeneous space which is 
invariant under the action of a maximal compact subgroup of $G$, in terms 
of a convex potential associated to the metric, and a description of the asymptotic 
behavior of this convex function.

We will use two definitions of the Futaki invariant. First, we will use 
an analytic definition to obtain its precise value. The original definition
of the Futaki invariant by Futaki \cite{Fut83} was generalized by Ding and 
Tian to normal $\mathbb{Q}$-Fano varieties \cite{DT92}.
The modified Futaki invariant, the analogue for Kähler-Ricci solitons, was 
introduced by Tian and Zhu \cite{TZ02}, 
and generalized to singular varieties by Berman and Witt Nystrom \cite{BWN}.
To show that we can compute the Futaki invariant on a degeneration though, 
we will use the algebraic definition of the Futaki invariant. 
This was first proposed by Donaldson \cite{Don02}.
He showed that the two definitions coincide when the variety is non singular. 
In fact, as remarked by Li and Xu \cite[Section 7, Remark 1]{LX14}, 
his proof extends to normal varieties. 
It is also Berman and Witt Nystrom who generalized this algebraic definition 
to the modified Futaki invariant (see also \cite{WZZ16} and \cite{Xio14} for related 
recent work on Futaki invariants). 

Let us end this introduction with a question related with K-stability and 
the method used here to compute Futaki invariants. 
Consider the partial order on $G$-spherical $\mathbb{Q}$-Fano varieties, 
given by $X_0\preceq X$ if $X_0$ is the central fiber of a special 
equivariant test configuration for $X$. 
Given a $G$-spherical $\mathbb{Q}$-Fano variety $X$, it is interesting to 
consider the set of all $G$-spherical $\mathbb{Q}$-Fano varieties
smaller than $X$ with respect to this order. 
It follows from \cite{AB04SRVII,AK05} that this poset is very well 
understood for reductive varieties. In 
particular there is a single minimum which is the central fiber of any special test 
configuration with horospherical central fiber, and the other elements are 
in bijection with the walls of the Weyl chamber, up to a finite number of isomorphisms
induced by exterior automorphisms of the root system. 
It would be interesting to extend this precise description to all spherical 
$\mathbb{Q}$-Fano varieties. The fact that the minima are horospherical 
varieties is stated in the present article. 
More generally, one may consider similar posets for more general 
$\mathbb{Q}$-Fano varieties by considering the partial order 
$X_0\preceq X$ if $X_0$ is the central fiber of a special 
test configuration for $X$, equivariant with respect to a Levi subgroup 
of the group of automorphisms of $X$.
Many questions about these posets could clarify  the interpretation of 
K-stability.  As a precise example, let us ask the following question:
\begin{question}
What are the minima of these posets?
In other words, if $X$ is a $\mathbb{Q}$-Fano variety, and $G$ is a 
Levi subgroup of $\mathrm{Aut}^0(X)$, under which conditions on $X$ 
are there no special $G$-equivariant test configurations for $X$ with central 
fiber not isomorphic to $X$?
\end{question}  
The only examples known to the author are horospherical varieties. 
One interest of knowing these minima is that, thanks to the argument 
we use in Section~\ref{sec_criterion}, they are the only varieties on which we 
really need to compute Futaki invariants.
Of course the poset alone is not enough to recover 
all the information about K-stability, as one should keep track of the different 
ways the elements of the poset appear as central fibers of test configurations, more 
precisely of the different holomorphic vector fields induced by these 
test configurations.

The ideas and the methods of the present article should have interesting applications 
to several related problems. The most obvious one is K-stability of spherical 
varieties for polarizations not given by the anticanonical line bundle. 
Another direction is in Sasakian 
geometry, where an analogue of the Yau-Tian-Donaldson conjecture was obtained
by Collins and Székelyhidi \cite{CS}. Finally, it would be interesting to combine the ideas 
from our work with the work of Ilten and Süss \cite{IS} to study varieties with a spherical 
action of complexity one. In particular, Langlois and Terpereau \cite{LT16,LT} have started a 
study of horospherical varieties of complexity one that should lead them to a 
criterion for being Fano, which would be a starting point.

The structure of the article is as follows. After introducing some notations 
and conventions in Section~\ref{sec_not}, we begin in Section~\ref{sec_horo_space}
by introducing horospherical homogeneous spaces and develop methods to deal 
with metrics on these spaces. Namely, given a metric on a linearized line bundle,
invariant under the action of a maximal compact subgroup, 
on a horospherical homogeneous space, we associate a real function defined 
on a vector space with it and obtain an expression of its curvature form in 
term of this function. 
In Section~\ref{sec_test_config} we recall the necessary notions from the 
theory of spherical varieties to determine special equivariant test 
configurations of Fano spherical varieties, then achieve this goal. 
We explain in particular why any $\mathbb{Q}$-Fano spherical variety 
has a special equivariant test configuration with a $\mathbb{Q}$-Fano 
horospherical central fiber. 
The valuation cone and the moment polytope are also introduced here.
Section~\ref{sec_fut_horo} is devoted to the computation of Futaki 
invariants of $\mathbb{Q}$-Fano horospherical varieties. It relies on 
Section~\ref{sec_horo_space} and the description of the asymptotic 
behavior of the function associated to the restriction of a positive 
hermitian metric to the open orbit. This description involves the moment polytope 
and generalizes the well-known case of polarized toric manifolds. 
Finally, we prove our main theorem in Section~\ref{sec_criterion}, 
after proving that the Futaki invariant for a given $\mathbb{C}^*$ 
action may be computed on a $\mathbb{C}^*$-equivariant 
degeneration of the variety.
The end of the section contains many examples of new situations 
in which the theorem applies, including new examples of Kähler-Einstein 
metrics, of Kähler-Ricci solitons, or of K-unstable manifolds. 
This illustrates also how the combinatorial criterion simplifies 
on specific subclasses of spherical varieties, such as symmetric 
varieties.

The reader interested in learning more about spherical varieties may consult 
\cite{Per14, Tim11, Bria}.  

\section{Notations}
\label{sec_not}

We introduce in this section some notions and notations on groups and Lie algebras
that will be used throughout the text.

Let $G$ be a complex linear algebraic group. 
We will denote its Lie algebra by the corresponding fraktur lower case letter $\mathfrak{g}$, 
and the exponential map by $\exp$.
An (algebraic) \emph{character} of $G$ is an algebraic group morphism 
$G \longrightarrow \mathbb{C}^*$.
We denote the group of algebraic characters of $G$ by $\mathfrak{X}(G)$. 
An (algebraic) \emph{one parameter subgroup} of $G$ is an algebraic group morphism 
$\mathbb{C}^* \longrightarrow G$. 
We denote the set of algebraic one parameter subgroups of $G$ by $\mathfrak{Y}(G)$.
We denote by $G^u$ the \emph{unipotent radical} of $G$,  
and by $[G,G]$ its derived subgroup.
If $H$ is a subgroup of $G$, let  
$N_G(H) = \{g\in G ; gHg^{-1}=H\}$ be the \emph{normalizer} of $H$ in $G$.

In this article, $G$ will denote a connected complex linear reductive group, 
$K$ will denote a maximal compact subgroup of $G$, and $\theta$ the \emph{Cartan 
involution} of $G$ such that $K = G^{\theta}$ is the fixed point set of $\theta$.
Let $T$ be a maximal torus of $G$, stable under $\theta$, let $B$ be a Borel 
subgroup of $G$ containing $T$, and let $B^-$ be the unique Borel subgroup of $G$ 
such that $B\cap B^-=T$, called the \emph{opposite Borel subgroup} with respect to $T$. 
Let $\Phi \subset \mathfrak{X}(T)$ be the root system of $(G,T)$ 
and let $\Phi^+\subset \Phi$ be the positive roots determined by $B$. 

A parabolic subgroup $P$ of $G$ containing $B$ admits a unique Levi subgroup $L$ 
containing $T$. The parabolic subgroup of $G$ \emph{opposite} to $P$ with respect to $L$ 
is defined as the unique parabolic subgroup $Q$ of $G$ such that $P\cap Q= L$ 
and $L$ is also a Levi subgroup
of $Q$. Let $\Phi_P$ denote the set of roots of $P$ with respect to $T$. 
We denote by $\Phi_{P^u}$ the set of roots of the unipotent radical $P^u$ of $P$. 
Alternatively, these are the roots of $G$ that are not roots of $Q$, or the roots of 
$P$ that are not roots of $L$.
Let  
\[
2\rho_P = \sum_{\alpha \in \Phi_{P^u}} \alpha.
\]

Consider the root decomposition of $\mathfrak{g}$:
\[
\mathfrak{g} = \mathfrak{t} \oplus \bigoplus_{\alpha \in \Phi} \mathfrak{g}_{\alpha}
\]
where 
$\mathfrak{g}_{\alpha} = 
\{x\in \mathfrak{g} ; \mathrm{ad}(t)(x)=\alpha(t)x ~ \forall t\in \mathfrak{t} \}$
is the root space for $\alpha$.
The Lie algebra of $P$ is 
\[
\mathfrak{p} = \mathfrak{t} \oplus \bigoplus_{\alpha \in \Phi_P} \mathfrak{g}_{\alpha}.
\]
The Cartan involution $\theta$ descends to an involution of $\mathfrak{g}$, 
still denoted by $\theta$. It sends $\mathfrak{g}_{\alpha}$ to $\mathfrak{g}_{-\alpha}$ 
for $\alpha \in \Phi$.

Assume now that $T$ is any algebraic complex torus. 
There is a natural duality pairing $\left< \lambda , \chi \right>$ 
between $\lambda \in \mathfrak{Y}(T)$
and $\chi \in \mathfrak{X}(T)$ defined as the integer such that 
$\chi\circ \lambda (z) = z^{\left< \lambda , \chi \right>}$ for $z\in \mathbb{C}^*$, 
which gives an isomorphism between $\mathfrak{Y}(T)$ and 
$\mathrm{Hom}(\mathfrak{X}(T),\mathbb{Z})$.
Given an algebraic complex subtorus $S\subset T$, there is a natural inclusion of 
$\mathfrak{Y}(S)$ in $\mathfrak{Y}(T)$ as a direct factor, 
and a natural inclusion of 
$\mathfrak{X}(T/S)$ in $\mathfrak{X}(T)$ as a direct factor, 
where a character of $T/S$ is 
identified with a character of $T$ which is trivial on $S$.
By duality, these inclusions imply that $\mathfrak{Y}(T/S)$ is a quotient 
of  $\mathfrak{Y}(T)$, and that $\mathfrak{X}(S)$ is a quotient of $\mathfrak{X}(T)$.

Recall that since $G$ is reductive, $\mathfrak{g}=\mathfrak{k}\oplus J \mathfrak{k}$,
where $J$ denotes the complex structure on $\mathfrak{g}$. 
Let $\mathfrak{a} = \mathfrak{t}\cap J\mathfrak{k}$. 
Given $\lambda : \mathbb{C}^* \longrightarrow T$ an algebraic one parameter subgroup, we 
consider its restriction to $\mathbb{R}^*_+$ and associate to $\lambda$ 
the derivative $a$ of the restriction at $1$, which lives in the tangent 
space $\mathfrak{t}=T_eT$. In fact, $a\in \mathfrak{a}$. Conversely any 
$a\in \mathfrak{a}$ defines a Lie group morphism 
$\tau \in \mathbb{C} \longmapsto \exp(\tau a)\in T$, 
which factorizes by $\mathbb{C} \longrightarrow 
\mathbb{C}^* = \mathbb{C}/2\pi\mathbb{Z}$ if and only if $a$ is obtained 
as the derivative of an algebraic one parameter subgroup as above.
This correspondence embeds $\mathfrak{Y}(T)$ as a lattice 
in $\mathfrak{a}$, and we identify $\mathfrak{Y}(T)\otimes \mathbb{R}$ 
with $\mathfrak{a}$.
The natural duality $\left<,\right>$ between $\mathfrak{X}(T)\otimes \mathbb{R}$ and 
$\mathfrak{Y}(T)\otimes \mathbb{R}$ translates, for $a\in \mathfrak{a}$ identified 
with an element of $\mathfrak{Y}(T)\otimes \mathbb{R}$, 
and $\chi \in \mathfrak{X}(T)$, as
$\left< \chi, a\right> = \ln \chi(\exp(a))$.

We will denote by $\kappa$ the \emph{Killing form} on $\mathfrak{g}$. It 
defines a scalar product on the semisimple part 
$\mathfrak{a} \cap [\mathfrak{g},\mathfrak{g}]$ of $\mathfrak{a}$.
We choose a scalar product $\left\{\cdot ,\cdot \right\}$ on $\mathfrak{a}$ which 
is invariant under the action of the Weyl group $W$, and whose 
restriction to the semisimple part is the Killing form $\kappa$. 
Note that the vector subspace on which $W$ acts trivially, which 
is the intersection of $\mathfrak{a}$ with the center of $\mathfrak{g}$, 
is then orthogonal to the semisimple part of $\mathfrak{a}$.
Let $e_{\alpha}$, for $\alpha \in \Phi$, be a generator of the root 
space $\mathfrak{g}_{\alpha}$, such that $e_{-\alpha}=-\theta(e_{\alpha})$,
and $[e_{\alpha},e_{-\alpha}]=t_{\alpha}$, where $t_{\alpha}$ is defined as 
the unique element of $\mathfrak{a}$ such that 
$\left\{t_{\alpha},a\right\}=\left< \alpha, a\right>$
for all $a\in \mathfrak{a}$.
More generally, given $\chi \in \mathfrak{X}(T)$, we denote by $t_{\chi}$ the  
 unique element of $\mathfrak{a}$ such that 
$\left\{t_{\chi},a\right\}=\left< \chi, a\right>$
for all $a\in \mathfrak{a}$.

\section{Curvature forms on horospherical homogeneous spaces}
\label{sec_horo_space}

\subsection{Horospherical homogeneous spaces}

We begin this section by introducing horospherical homogeneous spaces.
Our reference for these homogeneous spaces, and the horospherical 
$\mathbb{Q}$-Fano varieties that we will discuss in later sections, 
is \cite{Pas08}. 

\subsubsection{Definition and normalizer fibration}

\begin{defn}
A closed subgroup $H$ of a connected complex reductive group $G$ is called 
\emph{horospherical} if it contains the unipotent radical 
$U$ of a Borel subgroup $B$ of $G$. The homogeneous space $G/H$ is then called a \emph{horospherical
homogeneous space}.
\end{defn}

\begin{exa}
Consider the natural action of $\mathrm{SL}_2(\mathbb{C})$ on $\mathbb{C}^2$.
It has two orbits: the fixed point $0$ and its complement. 
The stabilizer of $(1,0)$ is easily seen to be the set of upper triangular matrices 
with diagonal coefficients equal to one. It is the unipotent radical $U$ 
of the Borel subgroup $B$ of $\mathrm{SL}_2(\mathbb{C})$ consisting of the upper 
triangular matrices, so $\mathbb{C}^2\setminus \{0\}$ is a horospherical homogeneous 
space. 
\end{exa}

A torus $(\mathbb{C}^*)^n$ is a horospherical homogeneous space under its action on 
itself. A generalized flag manifold $G/P$ where $G$ is a semisimple group and $P$
a parabolic of $G$ is a horospherical homogeneous space. Combining the two examples 
gives that products of tori with generalized flag manifolds are examples of 
horospherical homogeneous spaces. In fact all horospherical homogeneous spaces 
are torus fibrations over generalized flag manifolds (see for example \cite{Pas08}):

\begin{prop}
\label{prop_norm_fibr}
Assume that $H$ is horospherical, then the \emph{normalizer fibration} 
\[
N_G(H)/H\longrightarrow G/H \longrightarrow G/N_G(H)
\]
is a torus fibration over a generalized flag manifold.
More precisely, the normalizer $P = N_G(H)$ of $H$ in $G$ is a 
parabolic subgroup containing $B$, and the quotient $P/H=T/T\cap H$ is a torus. 
Conversely, if $H$ is such that $N_G(H)$ is a parabolic subgroup of $G$ 
and $P/H$ is a torus, then $H$ is horospherical. 
\end{prop}

\begin{exa}
The normalizer fibration associated to $\mathrm{SL}_2(\mathbb{C})/U$ is the one 
defining the complex projective line: 
$\mathbb{C}^* \longrightarrow \mathbb{C}^2\setminus \{0\} \longrightarrow \mathbb{P}^1.$
In this case, $P=B$ is the subgroup of upper triangular matrices.
\end{exa}

For any closed subgroup $H$, the normalizer $N_G(H)$ acts on $G/H$ 
by multiplication on the right by the inverse. The subgroup $H$ in $N_G(H)$
acts trivially. 
The action of $(g,pH)\in G\times N_G(H)/H$ on the coset $xH\in G/H$ is then 
given by 
$(g,pH)\cdot xH = gxp^{-1}H$.
The isotropy group of $eH$ under this action is the group 
$\{(p,pH), p\in N_G(H) \}$.
We will mainly use this for horospherical homogeneous spaces, 
and we denote the corresponding isotropy group by $\mathrm{diag}(P)$
in this case. It is isomorphic to $P$ \emph{via} the first projection.

\subsubsection{Polar decomposition on a horospherical homogeneous space}

We fix now $H$ a horospherical subgroup of $G$, and denote by $P$ its normalizer.
The inclusion of $\mathfrak{Y}(T\cap H)$ in $\mathfrak{Y}(T)$ gives rise to a 
subspace $\mathfrak{a}_0 \subset \mathfrak{a}$ under the identification of
$\mathfrak{a}$ with $\mathfrak{Y}(T)\otimes \mathbb{R}$.
Let $\mathfrak{a}_1$ denote the orthogonal complement of $\mathfrak{a}_0$
with respect to $\left\{\cdot,\cdot\right\}$.

\begin{prop}
\label{prop_polar_dec}
The image of $\mathfrak{a}_1$ in $G$ under the exponential is a fundamental 
domain for the action of $K\times H$ on $G$, where $K$ acts by multiplication 
on the left and $H$ by multiplication on the right by the inverse. 
As a consequence, the set $\{\exp(a)H; a\in \mathfrak{a}_1\}\subset G/H$ is a 
fundamental domain for the action of $K$ on $G/H$.
\end{prop}

\begin{proof}
The case when the horospherical subgroup is the unipotent radical $U$ of $B$
is a classical result known as the \emph{Iwasawa decomposition} 
(see for example \cite[Chapter IX, Theorem 1.3]{Hel78}).
It states more precisely that the map 
\[
\begin{tikzcd}
K\times \mathfrak{a} \times U \arrow[r]  & G, \quad (k,a,u) \arrow[r, maps to] & k\exp(a)u
\end{tikzcd} 
\]
is a diffeomorphism.

Now let $H$ be any horospherical subgroup containing $U$, and $\mathfrak{a}_1$ 
be as defined above. Given $g\in G$, use the Iwasawa decomposition to write
$g=k\exp(a)u$ where $a\in \mathfrak{a}$ is uniquely determined. Decompose 
$a$ as $a_0+a_1 \in \mathfrak{a}_0\oplus \mathfrak{a}_1=\mathfrak{a}$. 
Then $g=k\exp(a_1)\exp(a_0)u$ where $k\in K$, $\exp(a_0)u \in H$ and 
$a_1\in \mathfrak{a}_1$ uniquely determined. 
\end{proof}

\subsection{Hermitian metrics on linearized line bundles}

We will now associate to a $K$-invariant hermitian metric on a linearized bundle 
on a horospherical homogeneous space two functions.
One function will be associated to the pull back of the line bundle to $G$ under the 
quotient map, and the other will be associated to the restriction to $\exp(\mathfrak{a}_1)$.
Let us first describe the linearized line bundles obtained by this pull back and 
this restriction.

\subsubsection{Associated linearized line bundles}

\label{sec_sections}

\begin{defn}
Let $X$ be a $G$-variety. A \emph{$G$-linearized line bundle} $L$ on $X$ is a line bundle 
on $X$ equipped with an action of $G$ such that the bundle map $L\longrightarrow X$ 
is equivariant and the maps between the fibers induced by the action are linear.
\end{defn}

Let $G/H$ be a horospherical homogeneous space. 
Let $L$ be a $G\times P/H$-linearized line bundle on $G/H$.
The fiber $L_{eH}$ above $eH\in G/H$ defines a one dimensional representation of the 
isotropy group $\mathrm{diag}(P)$. Denote by $\chi$ the character of $P$
associated to this representation, so that 
$(p,pH)\cdot \xi = \chi(p) \xi$ for $\xi \in L_{eH}$.

Consider the quotient map $\pi : G\longrightarrow G/H, g\longmapsto gH$, and the 
pull back $\pi^*L$ of the line bundle $L$ to $G$. We still denote by $\pi$ the 
induced map $\pi^*L\longrightarrow L$. Since $\pi$ is equivariant 
under the action of $G$ by multiplication on the left, $\pi^*L$ admits a 
pulled back $G$-linearization. 
Let us choose a non zero element $s(e) \in (\pi^*L)_e$. Together with the linearization, 
it provides a global trivialization 
\[
s : g \in G \longmapsto g\cdot s(e) \in (\pi^*L)_g.
\]

Denote by $\iota : P/H\longrightarrow G/H$ the inclusion. It is equivariant 
under the action of $P\times P/H$.
The restriction $\iota^*L$ of the line bundle $L$ to $P/H$ provides a 
$P \times P/H$-linearized line bundle on the torus $P/H$. 
In general the linearizations of $\iota^*L$ for the action of $P\times \{eH\}\simeq P$
and the action of $P$ through $\{e\}\times P/H$ are different. 
They are related by:
\begin{align*}
(p,H)\cdot \xi & = (e,p^{-1}H)(p,pH) \cdot \xi \\
				& = \chi(p) (e,p^{-1}H)\cdot \xi
\end{align*}
for $p\in P$, $\xi\in \iota^*L$.

We still denote by $\iota$ the inclusion $\iota^*L\longrightarrow L$.
We define two trivializations $s_l$ and $s_r$ of $\iota^*L$: 
\begin{align*}
s_l(pH) & = (p,H)\cdot (\iota^{-1}\circ \pi)(s(e)), \\
s_r(pH) & = (e, p^{-1}H)\cdot (\iota^{-1}\circ\pi)(s(e)).
\end{align*}
By the previous calculation, they satisfy 
\[
s_l(pH)=\chi(p) s_r(pH).
\]

The following diagram summarizes the objects introduced in this subsection.
\[
\begin{tikzcd}
\pi^*L \arrow[d] \arrow[r] & 
	L \arrow[d] & 
		\iota^*L \arrow[l, hook'] \arrow[d] 
\\
G  \arrow[r,"\pi"] \arrow[u, bend left=50, "s"] & 
	G/H & 
		P/H \arrow[l, hook', "\iota"'] \arrow[u, bend left=50, "s_l"] \arrow[u, bend right=50, "s_r"'] 
\end{tikzcd}
\]

\subsubsection{Functions associated to hermitian metrics}

Let $L$ be a $G\times P/H$-linearized line bundle on $G/H$, and denote by $\chi$ 
the corresponding character of $\mathrm{diag}(P)$. Let $q$ be a smooth hermitian 
metric on $L$. 
Consider the pull-back $\pi^*q$ of the metric $q$ to $\pi^*L$. 
Denote by $\phi: G\longrightarrow \mathbb{R}$ the potential of $\pi^*q$ with respect to the section $s$, 
that is the function on $G$ defined by 
\[
\phi(g)=-2\ln|s(g)|_{\pi^*q}.
\]
We associate to $q$ another function, this time associated with its 
restriction to $L|_{P/H}$.
Denote by $u : \mathfrak{a}_1 \longrightarrow \mathbb{R}$ the function 
defined by 
\[
u(x)=-2\ln|s_r(\exp(x)H)|_{\iota^*q}.
\]

\begin{prop}
\label{prop_phi_u}
Assume that the metric $q$ is invariant under the action of the compact group $K$.
Then $q$ is uniquely determined by $u$. Furthermore, we have 
\[
\phi(k\exp(x)h) = u(x)-2\ln|\chi(\exp(x)h)|.
\]
for any $k\in K$, $x\in \mathfrak{a}_1$ and $h\in H$.
\end{prop}

\begin{proof}
It is clear that $\pi^*q$ is completely determined by its global potential $\phi$
on $G$, thus that $q$ is completely determined by $\phi$. We only need to prove 
the relation between $u$ and $\phi$, since any $g\in G$ can be written as 
$g =k\exp(x)h$ with $k\in K$, $x\in \mathfrak{a}_1$ and $h\in H$ by Proposition
\ref{prop_polar_dec}.

First remark that since $q$ is invariant under the action of $K$, and $\pi$ 
is equivariant under the action of $G$, $\pi^*q$ is also invariant under the action 
of $K$. Then the invariance of the section $s$ yields 
\[
\phi(k\exp(x)h)=-2\ln|s(k\exp(x)h)|_{\pi^*q}
=-2\ln|k\cdot s(\exp(x)h)|_{\pi^*q} = \phi(\exp(x)h).
\]

We then remark that 
\begin{align*}
\pi(s(\exp(x)h)) & = \pi((\exp(x)h)\cdot s(e)) \\
	& = (\exp(x)h,H)\cdot \pi(s(e)) \\
	& = (\exp(x),H) (h,hH) \cdot \pi(s(e)) \\
	& = \chi(h) (\exp(x),H)\cdot \pi(s(e)) 
\end{align*}
by definition of $\chi$.

We then write, since $\exp(x)\in P$,
\begin{align*}
\phi(\exp(x)h) & = -2\ln(|\chi(h)||(\exp(x),H)\cdot \pi(s(e))|_q) \\
	& = -2\ln|(\exp(x),H)\cdot \iota^{-1} \circ \pi(s(e))|_{\iota^*q} -2\ln|\chi(h)| \\
	& = -2\ln|s_l(\exp(x)H)|_{\iota^*q} -2\ln|\chi(h)| \\
\intertext{we then use the relation between sections described in Section \ref{sec_sections} to obtain}
\phi(\exp(x)h) & = -2\ln(|\chi(\exp(x))||s_r(\exp(x)H)|_{\iota^*q}) -2\ln|\chi(h)| \\
	& = -2\ln|s_r(\exp(x)H)|_{\iota^*q} -2\ln|\chi(\exp(x)h)|. 
\end{align*}
Recalling the definition of $u$, we obtain the statement.
\end{proof}

\subsection{Pointwise expression of a curvature form}

Let $L$ be a $G\times P/H$-linearized line bundle on $G/H$, with associated character
of $\mathrm{diag}(P)$ denoted by $\chi$.
Let $q$ be a $K$-invariant smooth hermitian metric on $L$, with associated functions 
$\phi : G \longrightarrow \mathbb{R}$ and
$u : \mathfrak{a}_1 \longrightarrow \mathbb{R}$.

Recall that the curvature form $\omega$ of $q$ is a global (1,1)-form defined locally as follows. 
If $s_0 : U\subset G/H \longrightarrow L$ is a local trivialization of $L$, 
and $\psi(z):=-2\ln|s_0(z)|_q$, then $\omega = i\partial \bar{\partial} \psi$ 
on $U$. 
If $q$ is invariant under the action of $K$ then $\omega$ is also invariant.
We want to compute the expression of $\omega$ in terms of $\chi$ and $u$.

In general we cannot find a global trivialization of $L$ on $G/H$, and  
cannot find a global potential for $\omega$. This is the case in particular
for generalized flag manifolds.
We will bypass this difficulty by computing $\omega$ through its pull back to 
$G$ under the quotient map.
This approach is similar to the use of \emph{quasipotentials} by Azad and Biswas 
in \cite{AB03} for generalized flag manifolds.

We can identify the tangent space at $eH$ to $G/H$ with 
\[
\mathfrak{g}/\mathfrak{h} \simeq 
\bigoplus_{\alpha \in \Phi_{P^u}} \mathbb{C} e_{-\alpha}
\oplus \mathfrak{a}_1 \oplus J\mathfrak{a}_1.
\]
Choose a basis $l_1,\ldots, l_r$ of the real vector space $\mathfrak{a}_1$.
Denote by $(x_1,\ldots,x_r)$ the corresponding coordinates of a point 
$x= \sum_i x_il_i \in \mathfrak{a}_1$.
A complex basis of the tangent space $T_{eH}G/H$ is then given by the union of the $l_j$ and 
the $e_{-\alpha}$ for $\alpha \in \Phi_{P^u}$.

On $P/H \subset G/H$, given a tangent vector $\xi$ at the coset $eH$, 
we can define a smooth real holomorphic vector field $R\xi$, 
invariant under the action of $P/H$ 
by multiplication on the right, simply by transporting the given tangent 
vector by the holomorphic action: 
\[ 
R\xi : pH \longmapsto (H,p^{-1}H)\cdot \xi \in T_{pH}P/H.
\]

Consider the complex basis of holomorphic (1,0)-vector fields 
composed of the 
$(Rl_j-iJRl_j)/2$ and 
$(Re_{-\alpha}-iJRe_{-\alpha})/2$, where $J$ denotes the complex 
structure on $G/H$, and $i$ is the complex structure coming from the 
complexification in $TG/H \otimes \mathbb{C}$.
We denote the dual basis of (1,0)-forms by $\{\gamma_j\}_j\cup \{\gamma_{\alpha}\}_{\alpha}$.
We will compute the curvature (1,1)-form $\omega$ pointwise in the basis of 
(1,1)-forms obtained from these. 

\begin{thm}
\label{thm_curv_form}
Let $\omega$ be the $K$-invariant curvature form of a $K$-invariant metric $q$.
Then the form $\omega$ is determined by its restriction to $P/H$, given 
for $x\in \mathfrak{a}_1$ by  
\[
\omega_{\exp(x)H} =  \! \sum_{1\leq j_1,j_2 \leq r} \!
\frac{1}{4} \frac{\partial^2 u}{\partial x_{j_1}\partial x_{j_2}}(x)
i\gamma_{j_1}\wedge \bar{\gamma}_{j_2}   
 + 
\! \sum_{\alpha \in \Phi_{P^u}} \! \left< \alpha, \nabla u(x)/2-t_{\chi} \right>
i\gamma_{\alpha} \wedge \bar{\gamma}_{\alpha}
\]
where $\nabla u$ is the gradient of $u$ with respect to the scalar product 
$\left\{,\right\}$.
\end{thm}

In order to prove the theorem we need to obtain an infinitesimal decomposition, 
adapted to the polar decomposition $G=K\exp(\mathfrak{a}_1)H$, for elements of 
$\mathfrak{g}/\mathfrak{h}$. 
Recall that $U$ is the unipotent radical of $B$ and is a subgroup of $H$. 
It is also the image under the exponential of 
$\oplus_{\alpha\in \Phi^+}\mathbb{C}e_{\alpha}$. 
It will be enough to obtain an infinitesimal decomposition adapted to 
the Iwasawa decomposition $G=K\exp(\mathfrak{a})U$, which is achieved 
by the following lemma.

\begin{lem}
\label{lem_inf_dec}
Let $\{z_j\}_j$ and $\{z_{\alpha}\}_{\alpha}$ denote complex numbers, and let 
\[ f= \sum_{1\leq j \leq r} z_jl_j 
+ \sum_{\alpha \in \Phi_P^+} z_{\alpha}e_{-\alpha}
\in \mathfrak{g}/\mathfrak{h}
\]
Then $\exp(f)=k\exp(y+O)u$ where $k\in K$, $y\in \mathfrak{a}$, $O\in \mathfrak{g}$ 
is of order strictly higher than two in the $z_j$ and $z_{\alpha}$, and 
$u\in U$. Furthermore, if $z_j=x_j+iy_j$,
\[
y=\sum_{1\leq j \leq r} x_j l_j + \sum_{\alpha \in \Phi_{P^u}} 
z_{\alpha}\bar{z}_{\alpha} t_{\alpha}/2.
\]
\end{lem}

\begin{proof}
Recall that $\theta$ denotes the Cartan involution on $\mathfrak{g}$, with fixed 
point set $\mathfrak{k}$.
We will denote by $O$ a term in $\mathfrak{g}$ of order strictly higher than two
in the $z_j$ and $z_{\alpha}$, which may change from line to line.

Write $f=A_1+A_2-A_3$, with 
\begin{align*}
A_1 & =\sum_{1\leq j \leq r} y_j J l_j 
+ \sum_{\alpha \in \Phi_{P^u}} (z_{\alpha}e_{-\alpha}+\theta(z_{\alpha}e_{-\alpha})) \in \mathfrak{k} \\
A_2 & =\sum_{1\leq j \leq r} x_j l_j  \in \mathfrak{a}_1 \\
A_3 & =\sum_{\alpha \in \Phi_{P^u}} \theta(z_{\alpha}e_{-\alpha}) 
\in \sum_{\alpha\in \Phi_{P^u}}\mathbb{C}e_{\alpha}.
\end{align*}

Using the Baker-Campbell-Hausdorff formula yields 
\[
\exp(-A_1)\exp(f)\exp(A_3) = \exp(A_2 + 1/2([A_2,A_1]+[A_1,A_3]+[A_2,A_3])+O)
\]

We now decompose 
$A_2 + 1/2([A_2,A_1]+[A_1,A_3]+[A_2,A_3]) = B_1+B_2+B_3$
with $B_1\in \mathfrak{k}$, $B_2\in \mathfrak{a}$ and 
$B_3\in \mathfrak{u}$. 
It is easy to see that $B_1$ and $B_3$ are of order two in the $z_j$, $z_{\alpha}$, 
so that, by the Baker-Campbell-Hausdorff formula again, 
\[
\exp(-B_1)\exp(-A_1)\exp(f)\exp(A_3)\exp(-B_3) = 
\exp(B_2 + O).
\]

We have $\exp(-B_1)\exp(-A_1) \in K$, and $\exp(A_3)\exp(-B_3)\in U$, so to conclude
the proof it remains to compute $y=B_2$.
There are two contributions to this term, the first being $A_2$ and the second, coming 
from $\frac{1}{2}[A_1,A_3]$, which is 
\begin{align*}
\frac{1}{2} \sum_{\alpha \in \Phi_{P^u}} 
[z_{\alpha}e_{-\alpha},\theta(z_{\alpha}e_{-\alpha})]
& = \frac{1}{2} \sum_{\alpha \in \Phi_{P^u}} 
z_{\alpha}\bar{z}_{\alpha} [e_{-\alpha},\theta(e_{-\alpha})] \\
& = \frac{1}{2} \sum_{\alpha \in \Phi_{P^u}} 
z_{\alpha}\bar{z}_{\alpha} t_{\alpha}  
\end{align*}
\end{proof}

We now proceed to prove the theorem.

\begin{proof}[Proof of Theorem \ref{thm_curv_form}]
Let $\pi$ denote again the quotient map $G\longrightarrow G/H$. 
Consider $\pi^*\omega$. 
This is the curvature form of the pulled back metric $\pi^*q$ on 
$\pi^*L$. 
Let $\phi$ be the global potential of $\pi^*q$ on $G$. 
Then $\phi$ is a global $i\partial \bar{\partial}$ potential for $\pi^*\omega$, 
which means that $\pi^*\omega=i\partial \bar{\partial} \phi$.
Recall from Proposition \ref{prop_phi_u} that 
$\phi(k\exp(x)h)=u(x)-2\ln(\chi(\exp(x)h))$.

Consider on $P\subset G$ the right invariant vector fields
$\tilde{R}l_j$, respectively $\tilde{R}e_{-\alpha}$ obtained 
by transporting the elements $l_j$ respectively $e_{-\alpha}$ 
from $\mathfrak{g}\simeq T_eG$ by the action of $P$ on $G$ by multiplication 
on the right by the inverse.
These vector fields are sent to $Rl_j$ respectively 
$Re_{-\alpha}$ by $\pi_*$.
Since for any elements $f_1,f_2$ of $T_gG\otimes \mathbb{C}$, we have 
\[
\pi^*\omega_g(f_1,f_2)=\omega_{\pi(g)}(\pi_*(f_1), \pi_*(f_2))
\]
it will be enough to compute $\pi^*\omega$ on pairs of holomorphic (1,0) 
vector fields $Z_j$ or $Z_{\alpha}$ corresponding to the 
real holomorphic vector fields just defined.

Let $f_1$ and $f_2$ be two elements of $\mathfrak{g}$, and let
$Z_1$, $Z_2$ be
the corresponding right-$P$-invariant holomorphic (1,0) vector fields 
on $P\subset G$ as defined above.
Since $\pi^*\omega = i\partial \bar{\partial} \phi$, we have, 
at $p\in P$,  
\[
(\pi^*\omega)_p(Z_1,\bar{Z}_2) = 
\left. \frac{\partial^2}{\partial z_1 \partial \bar{z}_2}\right|_0 
\phi(\exp(z_1f_1+z_2f_2)p).
\]
We will carry out this computation at $p=\exp(x)$ for $f_1$ and $f_2$ two  
elements in $\mathfrak{g}/\mathfrak{h}$.

Consider now $f$ as in Lemma~\ref{lem_inf_dec} and $x\in \mathfrak{a}_1$. Then 
\begin{align*}
\exp(f)\exp(x) & = k\exp(y+O)u\exp(x) \\
	& = k\exp(y+O)\exp(x)u' \\
\intertext{where $u'\in U$, because $T$ normalizes $U$. Then by the Baker-Campbell-Hausdorff 
formula and since $O$ is of order strictly higher than two in $z_j$, $z_{\alpha}$, there 
exists an $O'\in \mathfrak{g}$, still of order strictly higher than two in $z_j$, $z_{\alpha}$, 
such that}
\exp(f)\exp(x) & = k\exp(x+y+O')u. 
\end{align*}

We deduce that 
$\phi(\exp(f)\exp(x))=\phi(\exp(x+y+O'))$, 
because $u\in U\subset H$, and any character of $P$ vanishes on $U$.
Then we apply this to obtain, given $f_1$, $f_2$ in $\mathfrak{g}/\mathfrak{h}$,
\begin{align*}
(\pi^*\omega)_{\exp(x)}(Z_1,\bar{Z}_2) & =
\left. \frac{\partial^2}{\partial z_1 \partial \bar{z}_2}\right|_0 
\phi(\exp(z_1f_1+z_2f_2)\exp(x)) \\
	& = \left. \frac{\partial^2}{\partial z_1 \partial \bar{z}_2}\right|_0 
\phi(\exp(x+y+O')) \\
\intertext{where $y$ is given by Lemma~\ref{lem_inf_dec} for $f=z_1f_1 + z_2f_2$,}
	& = \left.  \frac{\partial^2}{\partial z_1 \partial \bar{z}_2}\right|_0 
\phi(\exp(x+y)) \\
	& = \left.  \frac{\partial^2}{\partial z_1 \partial \bar{z}_2}\right|_0 
\phi(\exp(x+y^1)\exp(y-y^1)) \\
\intertext{where $y^1$ is the projection in $\mathfrak{a}_1$ of 
$y\in \mathfrak{a}=\mathfrak{a}_0\oplus \mathfrak{a}_1$,}
	& = \left. \frac{\partial^2}{\partial z_1 \partial \bar{z}_2}\right|_0 
\left( u(x+y^1)-2\ln(\chi(\exp(x+y))) \right)\\
	& = \left. \frac{\partial^2}{\partial z_1 \partial \bar{z}_2}\right|_0 
\left( u(x+y^1)-2\left<\chi, x+y\right> \right). 
\end{align*}

Together with the precise value of $y$ given by 
Lemma~\ref{lem_inf_dec}, this allows us to compute the values of 
$(\pi^*\omega)_{\exp(x)}(Z_1,\bar{Z}_2)$
for all choices of $Z_1$, $Z_2$ in the set of right-$P$-invariant 
holomorphic (1,0)-vector fields obtained from elements $f_1$, $f_2$ of 
$\mathfrak{g}/\mathfrak{h}$.

(i) Let us first apply this to $f_1=l_{j_1}$, $f_2=l_{j_2}$. We obtain 
\begin{align*}
(\pi^*\omega)_{\exp(x)}(Z_1,\bar{Z}_2) &
= \left.\frac{\partial^2}{\partial z_1 \partial \bar{z}_2}\right|_0 
\left( u(x+x_1l_{j_1}+x_2l_{j_2})-2\left<\chi, x\right> \right)\\
&
= \left.\frac{1}{4}\frac{\partial^2}{\partial x_1 \partial x_2}\right|_0 
u(x+x_1l_{j_1}+x_2l_{j_2}) \\
& = \frac{1}{4}\frac{\partial^2 u}{\partial x_{j_1} \partial x_{j_2}}(x). 
\end{align*}

(ii) If $f_1=l_j$ and $f_2=e_{-\alpha}$ then Lemma~\ref{lem_inf_dec} gives 
\[
y=x_1l_j+\frac{1}{2}z_2\bar{z}_2t_{\alpha}
\]
and it is easy to see that the double derivative vanishes.

(iii) Similarly, if $f_1=e_{-\alpha_1}$ and $f_2=e_{-\alpha_2}$, with $\alpha_1\neq \alpha_2$, 
then 
\[
y=\frac{1}{2}( z_1\bar{z}_1t_{\alpha_1} + z_2\bar{z}_2t_{\alpha_2})
\]
and the double derivative vanishes again. 

(iv) The remaining case is when $f_1=f_2=e_{-\alpha}$. 
In that case, 
\[
y=\frac{1}{2}|z_1+z_2|^2t_{\alpha},
\] 
so that 
\begin{align*}
(\pi^*\omega)_{\exp(x)}(Z_1,Z_2) &
= \left.\frac{\partial^2}{\partial z_1 \partial \bar{z}_2}\right|_0 
\left( u(x+|z_1+z_2|^2t_{\alpha}^1/2)
-2\left<\chi, x+|z_1+z_2|^2t_{\alpha}/2\right> \right)\\
	& =\left\{\nabla u(x), t_{\alpha}^1/2\right\} - \left< \chi, t_{\alpha} \right> \\
	& =\left\{\nabla u(x)/2 - t_{\chi}, t_{\alpha} \right\}.
\end{align*}
We can indeed replace $t_{\alpha}^1$ with $t_{\alpha}$ because 
$\nabla u(x) \in \mathfrak{a}_1$, and, by definition,
$\mathfrak{a}_1$ is orthogonal to $\mathfrak{a}_0$.
\end{proof}

\section{Test configurations of spherical varieties}
\label{sec_test_config}

\subsection{Colored fans and spherical varieties}
\label{sec_intro_spher}

We first review general results about spherical varieties. 
We will use \cite{Kno91} as main reference for this section. 
The theory was initially developed by Luna and Vust \cite{LV83}.

\begin{defn}
A normal variety $X$ equipped with an action of $G$ is called \emph{spherical} 
if a Borel subgroup $B$ of $G$ acts on $X$ with an open and dense orbit. 
\end{defn}

A homogeneous space $G/H$ which is a spherical variety under the action of $G$
is a \emph{spherical homogeneous space}. A \emph{spherical subgroup} is a 
closed subgroup $H$ such that $G/H$ is a spherical homogeneous space.

Let $X$ be a spherical variety and $x$ a point in the open orbit of $B$.
Denote by $H$ the isotropy group of $x$ in $G$. The pair $(X,x)$ is called 
a \emph{spherical embedding} of the spherical homogeneous space $G/H$, 
and is equipped with a natural inclusion of $G/H$ in $X$ through the 
$G$-equivariant map $gH \longmapsto g\cdot x$.

\begin{exa}
A horospherical homogeneous space is spherical: if $H$ contains the unipotent 
radical $U$ of $B$, and $B^-$ is a Borel subgroup opposite to $B$, then 
$B^-H$ is open and dense in $G$, or equivalently, $B^-H/H$ is open and dense 
in $G/H$.
An embedding of a horospherical space is called a \emph{horospherical embedding} of $G/H$,
or a \emph{horospherical variety}.
\end{exa}

\begin{exa}
The group $G$ itself is a spherical  homogeneous space under the action of 
$G\times G$ defined by $(g_1,g_2)\cdot g=g_1gg_2^{-1}$. Indeed, if $B$ and 
$B^-$ are opposite Borel subgroups of $G$, the \emph{Bruhat 
decomposition} shows that the $B\times B^-$-orbit $BB^-$ is open and dense in $G$. 
\end{exa} 

\subsubsection{Valuation cone and colors}

Let $O$ be a spherical homogeneous space under the action of $G$.
Let $k=\mathbb{C}(O)$ be the function field of $O$. The action of $G$ 
on $k$ is defined by $(g\cdot f)(x)=f(g^{-1}\cdot x)$ for $g\in G$, $f\in k$, 
$x\in O$.

\begin{defn}
A \emph{valuation} of $k$ is a map 
$\nu: k^* = k\setminus\{0\} \longrightarrow \mathbb{Q}$ such that: 
\begin{itemize}
\item $\nu(\mathbb{C}^*)=0$,
\item $\nu(f_1+f_2)\geq \mathrm{min}\{\nu(f_1),\nu(f_2)\}$
when $f_1$, $f_2$ and $f_1+f_2$ are in $k^*$,
\item $\nu(f_1f_2)=\nu(f_1)+\nu(f_2)$ for all $f_1,f_2\in  k^*$.
\end{itemize}
\end{defn}

Let us now choose $B$ a Borel subgroup of $G$.
Define $\mathcal{M}_B(O) \subset \mathfrak{X}(B)$ as the set of characters $\chi$ 
such that there exists a function $f\in k^*$ with $b\cdot f= \chi(b)f$.
It is a subgroup of $\mathfrak{X}(B)$, and hence a finitely generated free abelian group. 

Define $\mathcal{N}_B(O)=\mathrm{Hom}(\mathcal{M}_B(O),\mathbb{Z})$.
To any valuation $\nu$ of $k$ we can associate an element $\rho_{\nu}$
of $\mathcal{N}_B(O) \otimes \mathbb{Q}$, defined by $\rho_{\nu}(\chi)=\nu(f)$
where $f\in k^*$ is such that $b\cdot f = \chi(b)f$ for all $b\in B$.
This is well defined because $B$ has an open and dense orbit, so two such 
functions are non-zero scalar multiples of each other.

It is a fundamental result in the theory that the map 
$\nu \longmapsto \rho_{\nu}$ is injective on the set of $G$-invariant 
valuations, 
and we denote by $\mathcal{V}_B(O)$ the image of the set of $G$-invariant valuations 
of $k$ under this map. 
This is a convex cone in $\mathcal{N}_B(O) \otimes \mathbb{Q}$ called the 
\emph{valuation cone} of $O$ (with respect to $B$).
Although the set of $G$-invariant valuations of $k$ does not depend on 
the choice of Borel subgroup, we will use in the following its image 
$\mathcal{V}_B(O)$, which does depend on the choice of $B$.

As an example, let us record the following characterization of horospherical 
varieties. Other examples will be described in Section~\ref{sec_sym}.
\begin{prop}{\cite[Corollaire 5.4]{BP87}} 
\label{prop_char_horo}
A spherical homogeneous space $O$ is a horospherical homogeneous space
if and only if its valuation cone $\mathcal{V}_B(O)$ is the full space
$\mathcal{N}_B(O)\otimes \mathbb{Q}$.
\end{prop}

Denote the set of $B$-stable prime divisors in $O$ by $\mathcal{D}_B(O)$. 
An element of $\mathcal{D}_B(O)$ is called a \emph{color} of $O$.
A color $D\in \mathcal{D}_B(O)$ defines a valuation on $G/H$ 
and thus an element $\rho(D)$ in $\mathcal{N}_B(O)_{\mathbb{Q}}$. 
However, the map 
$\rho: \mathcal{D}_B(O) \longrightarrow \mathcal{N}_B(O) \otimes \mathbb{Q}$ is not 
injective in general. 

We will in general drop the mention of $B$ and $O$ in the notations, as no 
confusion should be possible. 

\subsubsection{Colored fans}

Let $G/H$ be a spherical homogeneous space, and choose $B$ a Borel subgroup of $G$.
Let $\mathcal{V} \subset \mathcal{N}\otimes \mathbb{Q}$ be the valuation cone of $G/H$
with respect to $B$ and let $\mathcal{D}$ be its set of colors.

\begin{defn}\mbox{}
\begin{itemize}
\item A \emph{colored cone} is a pair $(\mathcal{C},\mathcal{R})$, where 
$\mathcal{R}\subset \mathcal{D}$, $0\notin \rho(\mathcal{R})$, and 
$\mathcal{C}\subset \mathcal{N} \otimes \mathbb{Q}$ is a strictly convex cone 
generated by $\rho(\mathcal{R})$ and finitely many elements of $\mathcal{V}$
such that the intersection of the relative interior of $\mathcal{C}$ with 
$\mathcal{V}$ is not empty.
\item Given two colored cones $(\mathcal{C},\mathcal{R})$ and 
$(\mathcal{C}_0,\mathcal{R}_0)$, we say that $(\mathcal{C}_0,\mathcal{R}_0)$ is a 
\emph{face} of $(\mathcal{C},\mathcal{R})$ if $\mathcal{C}_0$ is a face of 
$\mathcal{C}$ and $\mathcal{R}_0=\mathcal{R}\cap \rho^{-1}(\mathcal{C}_0)$.
\item A \emph{colored fan} is a non-empty finite set $\mathcal{F}$ of colored cones 
such that the face of any colored cone in $\mathcal{F}$ is still in $\mathcal{F}$, 
and any $v\in \mathcal{V}$ is in the relative interior of at most one cone.
\end{itemize}
\end{defn}

\begin{thm}{\cite[Theorem 3.3]{Kno91}}
\label{thm_class}
There is a bijection $(X,x) \longmapsto \mathcal{F}_X$ between embeddings of $G/H$ up to 
$G$-equivariant isomorphism and colored fans. 
There is a bijection $Y\mapsto (\mathcal{C}_Y, \mathcal{R}_Y)$ 
between the orbits of $G$ in $X$, and the colored cones in $\mathcal{F}_X$. 
An orbit $Y$ is in the closure of another orbit $Z$ in $X$ if and only if 
the colored cone $(\mathcal{C}_Z, \mathcal{R}_Z)$ is a face of 
$(\mathcal{C}_Y, \mathcal{R}_Y)$.
\end{thm}

The \emph{support} of the colored fan $\mathcal{F}_X$ is defined as 
$|\mathcal{F}_X|=\bigcup \{\mathcal{C} ; (\mathcal{C},\mathcal{R})\in \mathcal{F}_X\}$.

\begin{prop}{\cite[Theorem 4.2]{Kno91}}
A spherical variety $X$ is complete if and only if the support $|\mathcal{F}_X|$
of its colored fan contains $\mathcal{V}$.
\end{prop}

Given $X$ a spherical embedding of $G/H$, we denote by $\mathcal{D}_X$ 
its set of colors, which is the union of all sets $\mathcal{R} \subset \mathcal{D}$ for 
$(\mathcal{C},\mathcal{R})\in \mathcal{F}_X$.

\begin{defn}
A spherical variety $X$ is called \emph{toroidal} if $\mathcal{D}_X$ is empty.
\end{defn}

\begin{exa}{\cite[Exemple 1.10]{Pas08}}
A horospherical variety is toroidal if and only if the fibration structure of the horospherical 
homogeneous space given in Proposition~\ref{prop_norm_fibr} extends to the 
embedding. In other words, toroidal horospherical varieties are precisely 
the homogeneous fibrations over generalized flag manifolds, with fibers toric varieties.

In the case of the horospherical homogeneous space 
$\mathbb{C}^2\setminus \{0\} \simeq \mathrm{SL}_2(\mathbb{C})/U$,
there are two complete embeddings: $\mathbb{P}^2$ and the blow up of $\mathbb{P}^2$ 
at one point. The latter is toroidal, and the fibration structure is obvious, while 
the former is not toroidal.
\end{exa}

\subsubsection{Equivariant automorphisms}
\label{sec_equi_aut}

The classification of spherical embeddings up to $G$-equivariant automorphisms, 
together with the deep uniqueness Theorem of Losev \cite{Los09},  
shows that the neutral component of the group of $G$-equivariant 
automorphisms of a spherical variety $X$ is isomorphic to the 
neutral component of the group of $G$-equivariant 
automorphisms of its open $G$-orbit, through the restriction. 
Indeed, the construction in \cite{Kno91} of the colored fan $\mathcal{F}_X$ 
of a spherical embedding $(X,x)$ of $G/H$ does not depend on the choice of 
base point $x$ such that its isotropy group is $H$.

Choose $x$ a base point in the open $G$-orbit in $X$, and let $\sigma$ be an 
equivariant automorphism of the open $G$-orbit.
Then the stabilizer $H$ of $x$ in $G$ is also the stabilizer of its image $\sigma(x)$
by $\sigma$, since $\sigma$ commutes with $G$. 
The colored data of the pointed homogeneous spaces $(G/H,x)$ and $(G/H,\sigma(x))$ 
are thus related by an automorphism, which may differ from the identity 
only by exchanging some colors, by \cite{Los09}.
This exchange of colors is impossible if the equivariant automorphism is 
in the connected component of the identity.   
In this situation, $(X,x)$ and $(X,\sigma(x))$ 
are two embeddings of $G/H$ with the same fan $\mathcal{F}_X$, so they are 
$G$-equivariantly isomorphic by Theorem~\ref{thm_class}, which means that the 
equivariant automorphism of $G/H$ sending $x$ to $\sigma(x)$ extends to $X$. 

If we fix a base point (or rather the stabilizer $H$ of a base point), then we get 
an explicit description of these equivariant automorphisms. 
Indeed, the group of $G$-equivariant automorphisms of $G/H$ is isomorphic to the quotient 
$N_G(H)/H$, whose action on $G/H$ is induced by the action of $N_G(H)$ by multiplication
on the right by the inverse: $p\cdot gH = gHp^{-1} = gp^{-1}H$ 
(see for example \cite[Proposition 1.2]{Tim11}). 

The reader may find a more precise description of the full equivariant 
automorphism group in \cite{Gan18}.

\subsubsection{Morphisms between spherical varieties}

Let $H<H'$ be two spherical subgroups of $G$. Denote by 
$\phi : G/H \longrightarrow G/H'$ the corresponding $G$-equivariant surjective 
map of homogeneous spaces.
It induces 
a surjective homomorphism 
\[
\phi_*: \mathcal{N}(G/H)\otimes \mathbb{Q} \longrightarrow \mathcal{N}(G/H')\otimes \mathbb{Q}.
\]
Let $\mathcal{D}_{\phi}\subset \mathcal{D}(G/H)$ be the set of all $D\in \mathcal{D}(G/H)$
such that $\phi(D)=G/H'$.

\begin{thm}{\cite[Theorem 4.1]{Kno91}}
\label{thm_morphisms}
Let $(X,x)$, respectively $(X',x')$, be a spherical embedding of $G/H$, respectively $G/H'$,
then $\phi:G/H\longrightarrow G/H'$ extends to a $G$-equivariant morphism 
$\phi:X\longrightarrow X'$ sending $x$ to $x'$ if and only if for every colored cone 
$(\mathcal{C},\mathcal{R})\in \mathcal{F}_{X}$, there exists a colored cone 
$(\mathcal{C}',\mathcal{R}')\in \mathcal{F}_{X'}$ such that 
$\phi_*(\mathcal{C})\subset \mathcal{C}'$ and 
$\phi_*(\mathcal{R}\setminus \mathcal{D}_{\phi})\subset \mathcal{R}'$.
\end{thm}

\subsection{Line bundles on spherical varieties}

Let us recall some results from \cite{Bri89} about line bundles on spherical 
varieties. 

\subsubsection{Cartier divisors}

Let $X$ be a complete spherical variety, embedding of some spherical homogeneous space $O$.
Let $\mathcal{I}^G_X$ denote the finite set of $G$-stable prime divisors of $X$.
Any divisor $Y\in \mathcal{I}^G_X$ corresponds to a ray 
$(\mathcal{C},\emptyset) \in \mathcal{F}_X$, 
and we denote by $u_Y$ the indivisible generator of this ray in $\mathcal{N}$.
The set $\mathcal{D}$ of colors of $O$ is in bijection with the set of irreducible $B$-stable
but not $G$-stable divisors in $X$, by associating to a color of $O$ its closure in $X$.

Any $B$-stable Weil divisor $d$ on $X$ writes
\[
d=\sum_{Y\in \mathcal{I}_X^G} n_Y Y + \sum_{D\in \mathcal{D}} n_D \overline{D}
\]
for some integers $n_Y$, $n_D$.
In fact, any Weil divisor is linearly equivalent to a $B$-invariant divisor 
and Brion proved the following criterion to characterize Cartier divisors.

\begin{prop}{\cite[Proposition 3.1]{Bri89}}
A $B$-stable Weil divisor in $X$ is Cartier 
if and only if there exists an integral piecewise linear function $l_d$ 
on the fan $\mathcal{F}_X$ 
such that 
\[
d=\sum_{Y\in \mathcal{I}_X^G} l_d(u_Y) Y + \sum_{D\in \mathcal{D}_X} l_d(\rho(D)) \overline{D} 
+\sum_{D\in \mathcal{D}\setminus \mathcal{D}_X}n_D \overline{D}
\]
for some integers $n_D$.
\end{prop}

\subsubsection{Ample Cartier divisors and polytopes}
\label{sec_pol_car}

Let $\mathcal{F}_X^{\mathrm{max}}$ denote the set of cones 
$\mathcal{C} \subset \mathcal{N}\otimes \mathbb{Q}$ of maximal dimension, such that 
there exists $\mathcal{R}\subset \mathcal{D}$ with 
$(\mathcal{C},\mathcal{R})\in\mathcal{F}_X$.
If $d$ is a Cartier divisor and $\sigma \in \mathcal{F}_X^{\mathrm{max}}$,
let $m_{\sigma}$ be the element of $\mathcal{M}$ such that 
$l_d(x)=m_{\sigma}(x)$ for $x\in \sigma$. Since we assumed $X$ complete, $l_d$ is 
uniquely determined by the $m_{\sigma}$.

\begin{prop}{\cite[Th\'eor\`eme 3.3]{Bri89}}
\label{prop_ample}
Assume $X$ is complete and 
\[ 
d=\sum_{Y\in \mathcal{I}_X^G} l_d(v_Y) Y + \sum_{D\in \mathcal{D}_X} l_d(\rho(D)) \overline{D} 
+\sum_{D\in \mathcal{D}\setminus \mathcal{D}_X}n_D \overline{D}
\]  
is a Cartier divisor on $X$. 
It is ample if and only if the following conditions are satisfied:
\begin{itemize}
\item the function $l_d$ is convex,
\item $m_{\sigma_1} \neq m_{\sigma_2}$ if 
$\sigma_1\neq \sigma_2\in \mathcal{F}_X^{\mathrm{max}}$, 
\item $n_D > m_{\sigma}(\rho(D))$ for all $D\in \mathcal{D} \setminus \mathcal{D}_{X}$ 
and $\sigma \in \mathcal{F}_X^{\mathrm{max}}$.
\end{itemize}
\end{prop}

To a Cartier divisor $d$, we associate a polytope 
$\Delta_d\subset \mathcal{M}\otimes \mathbb{R}$ 
defined as the set of $m\in \mathcal{M}\otimes \mathbb{R}$ such that 
$m\in -m_{\sigma} + \sigma^{\vee}$ for all $\sigma \in \mathcal{F}_X^{\mathrm{max}}$, 
and $m(\rho(D))+n_D\geq 0$ for all $D\in \mathcal{D}\setminus \mathcal{D}_X$.

The support function 
$v_{\Delta} : \mathcal{N}\otimes \mathbb{R} \longrightarrow \mathbb{R}$ of a 
polytope $\Delta \subset  \mathcal{M}\otimes \mathbb{R}$ is defined by 
\[
v_{\Delta}(x)=\mathrm{sup} \{ m(x); m\in \Delta \}.
\] 
If $d$ is ample then $l_d(x)=v_{\Delta_d}(-x)$ for $x\in |\mathcal{F}_X|$. 

\subsubsection{Linearized line bundles and moment polytopes}

Let $L$ be a $G$-linearized ample line bundle on a spherical variety $X$.
Let $B$ be a Borel subgroup of $G$ and $T$ a maximal torus of $B$. 
Denote by $V_{\lambda}$ an irreducible representation of $G$ of  
highest weight $\lambda \in \mathfrak{X}(T)$ with respect to $B$.
Since $X$ is spherical, for all $r\in \mathbb{N}$, there exists a 
finite set $\Delta_r\subset \mathfrak{X}(T)$ such that 
\[ 
H^0(X,L^r)= \bigoplus_{\lambda\in \Delta_r} V_{\lambda}.              
\]

\begin{defn}   
The \emph{moment polytope} $\Delta_L$ of $L$ \emph{with respect to} $B$ 
is defined as the closure of 
$\bigcup_{r\in \mathbb{N}^*} \Delta_r/r$ in 
$\mathfrak{X}(T) \otimes \mathbb{R}$.
\end{defn} 

Even though it is not clear from the definition, $\Delta_L$ is 
a polytope. More precisely, we may recall the explicit relation 
between $\Delta_L$ and the polytope associated to a Cartier divisor 
whose associated line bundle is $L$. Fix a linearization of $L$, 
and choose a global $B$-semi-invariant section $s$ of $L$,
so that the zero divisor $d$ of $s$ is an ample $B$-invariant Cartier divisor. 
Let $\mu_s$ be the character of $B$ defined by $s$, that is, such that 
$b\cdot s(b^{-1} \cdot x) = \mu_s(b) s(x)$ for all $x\in X$.

\begin{prop}{(\cite[Proposition 3.3]{Bri89}, see also \cite[Section 5.3]{Bria})}          
\label{prop_mom_pol}
The moment polytope $\Delta_L$ of $L$ and the polytope $\Delta_d$ associated to $d$ 
are related by $\Delta_L = \mu_s+\Delta_d$.  
\end{prop}

\subsubsection{Anticanonical line bundle}
\label{sec_GH}

Let us now recall some results from \cite{GH15}. 
In this article, Gagliardi and Hofscheier study the anticanonical line bundle
on a spherical variety, in particular on $\mathbb{Q}$-Fano spherical varieties. 
It is based on the work of Brion \cite{Bri97}, and the analogue for $\mathbb{Q}$-Fano 
horospherical varieties by Pasquier \cite{Pas08}.
We consider the anticanonical divisor on a $\mathbb{Q}$-Fano spherical variety $X$.
It is clear that the discussion of Cartier divisors and moment polytopes above 
extends to $\mathbb{Q}$-Cartier divisors and linearized $\mathbb{Q}$-line bundles.
Let $K_X^{-1}$ denote the (naturally linearized) $\mathbb{Q}$-line bundle on $X$.

Let $P$ be the stabilizer of the open orbit of $B$ in $X$.
There exists a $B$-semi-invariant section of $K_X^{-1}$ with weight $2\rho_P$ 
and divisor 
\[
d=\sum_{Y\in \mathcal{I}_X^G} Y + \sum_{D\in \mathcal{D}} n_D \overline{D}
\]
where the $n_D$ are explicitly obtained in terms of $2\rho_P$ and the 
types of the roots (see \cite{GH15} for a precise description of these 
coefficients). 

The moment polytope $\Delta^+$ is then $2\rho_P+\Delta_d$ 
by Proposition~\ref{prop_mom_pol}, and furthermore 
the dual polytope $\Delta_d^*$ of $\Delta_d$ is a 
$\mathbb{Q}$-$G/H$-reflexive polytope in the sense of \cite{GH15}, 
which can be obtained as the convex hull:
\[
\Delta_d^*=\mathrm{conv}(\{\rho_D/n_D, D\in\mathcal{D} \} \cup \{u_Y,Y\in \mathcal{I}_X^G \}).
\]

The $\mathbb{Q}$-Fano variety $X$ can further be recovered from its 
$\mathbb{Q}$-$G/H$-reflexive polytope $\Delta_d^*$ by the following procedure, 
detailed in \cite{GH15}.
The colored fan of $X$ is obtained from $\Delta_d$ as the union of 
the colored cones $(\mathrm{Cone(F)},\rho^{-1}(F))$ for all faces $F$ of $\Delta_d^*$
such that the intersection of the relative interior of $\mathrm{Cone}(F)$ with 
the valuation cone $\mathcal{V}$ is not empty.

\subsection{Equivariant degenerations of spherical spaces}
\label{sec_elem_emb}

\subsubsection{Adapted parabolic and Levi subgroups}

Let $X$ be a spherical variety under the action of $G$ and $B$
a Borel subgroup of $G$. The stabilizer in $G$ of the open orbit 
of $B$ is a parabolic subgroup of $G$ containing $B$ called the 
\emph{adapted parabolic}.

\begin{defn}
Let $H$ be a spherical subgroup of $G$.
An \emph{elementary embedding} of $G/H$ is a spherical embedding $(E,x)$ 
of $G/H$ such that the boundary $E_0=E\setminus (G/H)$ is a single 
codimension one $G$-orbit (necessarily closed).
\end{defn}

The colored fan of an elementary embedding is a single ray in the 
valuation cone, with no colors \cite[2.2]{BP87}. 

Choose $B$ a Borel subgroup of $G$ such that $BH$ is open in $G$.
The adapted parabolic $P$ is also the stabilizer in $G$ of $BH$. 

\begin{defn}
A Levi subgroup $L$ of $P$
is called 
\emph{adapted to} $H$ 
if the following conditions hold:
\begin{itemize}
\item $P\cap H=L\cap H$,
\item $L\cap H$ contains the derived subgroup $[L,L]$,
\item for any elementary embedding $(E,x)$ of $G/H$ with closed orbit $E_0$,
the closure $\overline{C\cdot x}$ of the orbit of $x$ under the action of the 
connected center $C$ of $L$ meets that orbit of $B$ which is open in $E_0$.
\end{itemize}
\end{defn}

The choice of an adapted Levi subgroup $L$ together with a maximal torus 
$T\subset L$ allows us to identify $\mathcal{M}$ 
with an explicit subgroup of $\mathfrak{X}(T)$
 \cite[2.9]{BP87}.
More precisely, the group $\mathcal{M}$ is identified with $\mathfrak{X}(T/T\cap H)$,
the group $\mathcal{N}$ is identified with the group 
$\mathfrak{Y}(T/T\cap H)$ of one parameter subgroups of $T/T\cap H$,  
and so $\mathcal{V}$ is identified with a cone in the $\mathbb{Q}$-vector space 
$\mathfrak{Y}(T/T\cap H) \otimes \mathbb{Q}$.
Denote by $\pi$ the quotient map 
$\mathfrak{Y}(T) \otimes \mathbb{Q} \longrightarrow 
\mathfrak{Y}(T/T\cap H) \otimes \mathbb{Q}$.

\begin{defn}
Let $(E,x)$ be an elementary embedding of $G/H$, and $\mathcal{C}_E$ be the ray 
in $\mathcal{V}$ associated to $E$. We say that a one parameter subgroup 
$\lambda \in \mathfrak{Y}(T)$ is \emph{adapted to $E$} if it projects to an 
element of $\mathcal{C}_E \cap \mathfrak{Y}(T/T\cap H)$ under the 
quotient map $\pi$.
\end{defn}

\begin{prop}{\cite[2.10]{BP87}}
Let $(E,x)$ be an elementary embedding of $G/H$, and $\lambda \in \mathfrak{Y}(T)$
adapted to $E$, then $\lim_{z\rightarrow 0} \lambda(z)\cdot x$ exists and is a 
point in the open orbit of $B$ in $E_0$.
\end{prop}

\subsubsection{Choice of an adapted Levi subgroup}

We will need to use some properties of the adapted Levi subgroups, essentially
proved in \cite{BLV86,BP87}. However, since they are sometimes proved for an 
adapted Levi subgroup of a particular form, we need to show that we can choose
a good Levi subgroup, up to changing the base point. Let us first prove the 
following elementary fact.

\begin{prop}
\label{prop_conj}
Let $L$ be a Levi subgroup adapted to $H$, and $u\in P^u$ the unipotent 
radical of $P$, then $uLu^{-1}$ 
is a Levi subgroup of $P$ adapted to $uHu^{-1}$.
\end{prop}

Remark that all Levi subgroups of $P$ are conjugate under an element of 
$P^u$. Remark also that if $H$ is the stabilizer of $x$, then 
$uHu^{-1}$ is the stabilizer of the point $u\cdot x$, which is still in 
the open $B$-orbit by definition of $P$. 

\begin{proof}
We can first see that
\[
P\cap uHu^{-1} = u(P\cap H)u^{-1}=u(L\cap H)u^{-1}=uLu^{-1}\cap uHu^{-1},
\]
and
\[
[uLu^{-1},uLu^{-1}]=u[L,L]u^{-1}\subset uHu^{-1}.
\] 

Assume now that $(E,x)$ is an elementary embedding of $G/uHu^{-1}$ with closed
orbit $E_0$. Then $(E,u^{-1}\cdot x)$ is an elementary embedding of $G/H$ with closed
orbit $E_0$. Since $L$ is adapted to $H$, if $C$ is the connected center of $L$,
$\overline{C\cdot u^{-1}\cdot x}$ meets the open orbit of $B$ in $E_0$.
The connected center of $uLu^{-1}$ is $uCu^{-1}$, and we get 
$\overline{uCu^{-1} \cdot x}=u\cdot \overline{C\cdot u^{-1}\cdot x}$.
This meets the open orbit of $B$ in $E_0$ since $u\in P^u\subset U$, so 
$uLu^{-1}$ is adapted to $uHu^{-1}$.
\end{proof}

\begin{prop}
\label{prop_levi_choice}
Let $T$ be a maximal torus of $B$.
Then up to changing the base point in the open orbit of $B$ and thus its stabilizer $H$, 
we can choose an adapted Levi subgroup $L$ containing the torus $T$ and such that 
$N_G(H)=H(C\cap N_G(H))$.
\end{prop}

\begin{proof}
First choose any base point $\tilde{x}$ in the open orbit of $B$, 
and let $\tilde{H}$ be its stabilizer.
Let $\tilde{L}=\mathrm{Stab}_G(df_e)$ where $f\in \mathbb{C}[G]$ is a regular function
on $G$ which vanishes everywhere on $G\setminus BH$ and 
$df_e$ is the differential of $f$ at the 
neutral element $e$, considered as an element of the coadjoint representation.
Let also $\tilde{C}$ denote the connected center of $\tilde{L}$.
Then $\tilde{L}$ is adapted to $\tilde{H}$ \cite[Section 3]{BLV86}. 

Furthermore, it is shown in \cite[Section 5]{BP87} that $B\tilde{H}=BN_G(\tilde{H})$,  
so $N_G(\tilde{H})$ is spherical
and $P$ and $\tilde{L}$ are also adapted to $N_G(\tilde{H})$. In particular, we have 
$P\cap N_G(\tilde{H}) = \tilde{L}\cap N_G(\tilde{H})$.

The Levi subgroup $\tilde{L}$ might not contain the maximal torus $T$, but there is a 
conjugate $L$ under an element $u$ of $P^u$ that does contain $T$. 
By Proposition~\ref{prop_conj}, $L=u\tilde{L}u^{-1}$ is adapted to the 
subgroup $H=u\tilde{H}u^{-1}$.

Furthermore, we have 
$P\cap N_G(H) = L \cap N_G(H)$.
From this we deduce that the inclusion of $CH$ in $CN_G(H)$ is an equality.
Indeed, $P^u\cap N_G(H) \subset P\cap N_G(H) = L \cap N_G(H)$, and $L\cap P^u=\{e\}$, 
so the following surjective maps are isomorphisms. 
\[
P^u\times CN_G(H)\longrightarrow PN_G(H)=BH=PH \longleftarrow P^u \times CH.
\] 
These isomorphisms imply that $CH=CN_G(H)$, and 
this equality implies the last conclusion: $N_G(H)=H(C\cap N_G(H))$.
\end{proof}

\subsubsection{Equivariant degenerations}

We fix now a spherical homogeneous space $O$ under the action of $G$, 
a Borel subgroup $B$ of $G$ and a maximal torus $T$ in $B$.
Using Proposition~\ref{prop_levi_choice}, we choose a base point $x\in O$ 
with isotropy group $H$ so that there exists a Levi subgroup $L$ of $P$ 
adapted to $H$ and containing the fixed maximal torus $T$ of $B$. 
Note that $G/H\times \mathbb{C}^*$ is a 
spherical homogeneous space under the action of $G\times \mathbb{C}^*$.

We use the work of Brion and Pauer on isotropy subgroups of elementary embeddings 
\cite[section 3]{BP87} to obtain information about the equivariant 
degenerations of a spherical homogeneous space.
In this article, equivariant degenerations of spherical homogeneous 
spaces are defined as follows.

\begin{defn}
An \emph{equivariant degeneration} of $G/H$ is an elementary embedding 
$(E,\tilde{x})$ of $G/H\times \mathbb{C}^*$, equipped with 
a surjective $G\times \mathbb{C}^*$-equivariant morphism 
$p:E\longrightarrow \mathbb{C}$, 
where $G$ acts trivially on $\mathbb{C}$, and $\mathbb{C}^*$
acts by multiplication on $\mathbb{C}$, 
with $p^{-1}(0)=E_0$ the closed orbit of $E$. 
\end{defn}

Let us remark that $P\times \mathbb{C}^*$ is an adapted parabolic for 
$(G\times \mathbb{C}^*)/(H\times \{1\})$, and $L\times \mathbb{C}^*$ 
is an adapted Levi subgroup with connected center $C\times \mathbb{C}^*$. 
We identify $\mathfrak{Y}(T\times \mathbb{C}^*)$ with $\mathfrak{Y}(T)\oplus \mathbb{Z}$.
Furthermore, the valuation cone of $(G\times \mathbb{C}^*)/(H\times \{1\})$ can be 
identified with $\mathcal{V}\times \mathbb{Q}\subset 
(\mathcal{N}\otimes \mathbb{Q})\oplus \mathbb{Q}$.

\begin{prop}
\label{prop_act}
Let $(E,\tilde{x})$ be an equivariant degeneration of $G/H$.
Let $(\lambda,m)\in \mathfrak{Y}(T)\oplus \mathbb{Z}$ be a one parameter subgroup 
adapted to $E$.
Then $m>0$, the action of $G$ on $E_0$ is transitive, and if $H_0$ denotes the isotropy 
subgroup in $G$ of $\tilde{x}_0= \lim_{z\rightarrow 0} (\lambda(z),z^m)\cdot \tilde{x}$, 
then the action of $e^{\tau}\in \mathbb{C}^*$ on $G/H_0$ is given by 
multiplication on the right by $\lambda(e^{-\tau/m})$.
\end{prop}

\begin{proof}
Let $p : E\longrightarrow \mathbb{C}$ denote the $G\times \mathbb{C}^*$-equivariant 
morphism associated to the degeneration. 
Then since $\tilde{x}_0 \in E_0 = p^{-1}(0)$, we have $m>0$.
Denote by $\tilde{H}_0$ the isotropy subgroup of $\tilde{x}_0$ in $G\times \mathbb{C}^*$. 
Obviously, $(\lambda(z),z^m) \in \tilde{H}_0$ for all $z\in \mathbb{C}^*$.

Let us first show that $G$ acts transitively on $(G\times \mathbb{C}^*)/\tilde{H}_0$.
Consider $(g_1,z_1)$ and $(g_2,z_2)$ in $G\times \mathbb{C}^*$. Let $s\in \mathbb{C}^*$
be such that $s^m=z_2/z_1$. Then 
\[
(g_1,z_1)=(g_1,z_2/s^m)=(g_1\lambda(s),z_2)(\lambda(1/s),1/s^m),
\] 
so since $(\lambda(1/s),1/s^m)\in \tilde{H}_0$, 
\[
(g_1,z_1)\tilde{H}_0=(g_1\lambda(s)g_2^{-1},1)(g_2,z_2)\tilde{H}_0,
\] 
which shows the transitivity of the action of $G$. 
In particular if $H_0$ is the isotropy group of $\tilde{x}_0$ in $G$, we can identify $E_0$ with $G/H_0$.

The action of $\mathbb{C}^*$ is obtained similarly. Given $z=e^{\tau}\in \mathbb{C}^*$, 
we have 
\begin{align*}
(e,z)(g,1)\cdot \tilde{x}_0 & = (e,z)(g,1)\tilde{H}_0 \\
          & = (g,e^{\tau})\tilde{H}_0 \\
          & = (g\lambda(e^{-\tau/m}),1)(\lambda(e^{\tau/m}),e^{\tau})\tilde{H}_0 \\
          & = (g\lambda(e^{-\tau/m}),1)\tilde{H}_0 \\
          & = (g\lambda(e^{-\tau/m}),1)\cdot \tilde{x}_0.     
\end{align*} 
This finishes the proof of the proposition.
\end{proof}

Conversely, any elementary embedding whose ray is generated by 
some $(\lambda,m)\in \mathfrak{Y}(T/T\cap H)\otimes \mathbb{Q} \times \mathbb{Q}$ 
with $m>0$ 
provides an equivariant degeneration of $G/H$ thanks to 
Theorem~\ref{thm_morphisms}.

\begin{prop}
\label{prop_still_adapted}
Keeping the same notations as in Proposition~\ref{prop_act}, we have:
$BH_0$ is open in $G$ (in particular $H_0$ is spherical), 
$P=\mathrm{Stab}_G(BH_0)$, and $L$ is adapted to $H_0$.
\end{prop}

\begin{proof}
By \cite[Théorème 3.6]{BP87}, $(B\times \mathbb{C}^*)\tilde{H}_0$ is open in $G\times \mathbb{C}^*$,
$P\times \mathbb{C}^* = \mathrm{Stab}_{G\times \mathbb{C}^*}((B\times \mathbb{C}^*)\tilde{H}_0)$, 
and $L\times \mathbb{C}^*$ is adapted to $\tilde{H}_0$, 
where $\tilde{H}_0$ denotes the isotropy group of $\tilde{x}_0$ in $G\times \mathbb{C}^*$ as in the proof 
of Proposition~\ref{prop_act}.

First remark that for any $b \in B$ and $z\in \mathbb{C}^*$, we have 
$(b,z^m)\cdot \tilde{x}_0 = (b\lambda(1/z),1) \cdot \tilde{x}_0$.
Since $\lambda$ is a one parameter subgroup of $T \subset B$,
we obtain that $B\times \mathbb{C}^* \cdot \tilde{x}_0= B\times \{1\} \cdot \tilde{x}_0$, 
thus the orbit of $\tilde{x}_0 = eH_0$ in $G/H_0$ is open, which implies that 
$BH_0$ is open in $G$.

Now let us show that $\mathrm{Stab}_G(BH_0) = P$.  
We have proved above that $B\cdot \tilde{x}_0 = (B\times \mathbb{C}^*)\cdot \tilde{x}_0$, 
and furthermore we have $\{e\} \times \mathbb{C}^*\subset 
\mathrm{Stab}_{G\times \mathbb{C}^*}((B\times \mathbb{C}^*)\cdot \tilde{x}_0)$, so it follows that 
\begin{align*}
P\times \mathbb{C}^* & = \mathrm{Stab}_{G\times \mathbb{C}^*}((B\times \mathbb{C}^*)\cdot \tilde{x}_0) \\ 
	 & = \mathrm{Stab}_{G}((B\times \mathbb{C}^*)\cdot \tilde{x}_0) \times \mathbb{C}^* \\
	 & = \mathrm{Stab}_{G}(B\cdot \tilde{x}_0) \times \mathbb{C}^* 
\end{align*}
Hence $P=\mathrm{Stab}_G(B\cdot \tilde{x}_0)=\mathrm{Stab}_G(BH_0)$.

Finally we have to show that $L$ is adapted to $H_0$.

First let us describe the subgroup $\tilde{H}_0$ more explicitly in terms of $H_0$, 
$\lambda$ and $m$. By the description of the action of $G\times \mathbb{C}^*$ on 
$G/H_0$ we easily check that 
\[
\tilde{H}_0 = 
\bigcup_{\tau \in \mathbb{C}} H_0\lambda(e^{\tau/m})\times \{e^{\tau}\}.
\]
Since $L\times \mathbb{C}^*$ is adapted to $\tilde{H}_0$ we obtain 
\begin{align*}
(P\cap H_0) \times \{1\} & = ((P\times \mathbb{C}^*) \cap \tilde{H}_0) \cap (G\times \{1\}) \\  
& = ((L\times \mathbb{C}^*) \cap \tilde{H}_0) \cap (G\times \{1\}) \\  
& = (L\cap H_0) \times \{1\}
\end{align*}
hence $P\cap H_0 = L\cap H_0$.
Similarly, $[L,L]\subset H_0$ follows from:
\[
[L,L]\times\{1\} = [L\times \mathbb{C}^*,L\times \mathbb{C}^*] 
\subset  \tilde{H}_0\cap (G\times \{1\}) = H_0\times \{1\}.
\]

Consider $(Z,z)$ an elementary $G$-embedding of $G/H_0=(G\times \mathbb{C}^*) / \tilde{H}_0$, 
with closed orbit $Z_0$. 
Since the action of $\mathbb{C}^*$ commutes with the action of $G$ on $G/H_0$, 
this is also an elementary embedding for the action of $G\times \mathbb{C}^*$. 
Since $L\times \mathbb{C}^*$ is adapted to $(G\times \mathbb{C}^*) / \tilde{H}_0$, 
there exists a one parameter subgroup $t\longmapsto (\mu(t),t^k)$ of $T\times \mathbb{C}^*$
such that $z_0 := \lim_{t\rightarrow 0} (\mu(t),t^k)\cdot z$ is in the open 
$B\times \mathbb{C}^*$-orbit in $Z_0$. This $z_0$ obviously lies in 
the closure of $C\cdot z$ in $Z$, since 
$(\mu(t),t^k)\cdot z = (\mu(t)\lambda(s^k),1) \cdot z$ if $s^m=t$.
As previously, $(b,t) \in B \times \mathbb{C}^*$ acts on $z_0$ as 
$(b,t)\cdot z_0 = (b,t)(\mu(1/s),1/t)\cdot z_0 = (b\mu(1/s),1)\cdot z_0$
where $s^k=t$, so $B\cdot z_0= (B\times \mathbb{C}^*) \cdot z_0$ is open in 
$Z_0$. We have thus shown that $L$ is adapted to $H_0$.
\end{proof}

\subsubsection{Elementary embeddings and equivariant automorphisms}

Let $(E,\tilde{x})$ be an equivariant degeneration of $G/H$,
and $(\lambda,m) \in \mathfrak{Y}(T) \times \mathbb{Z}$ a one parameter subgroup
adapted to $E$. Let $\tilde{x}_0 = \lim_{z\rightarrow 0} (\lambda(z),z^m)\cdot \tilde{x}$
and denote by $H_0$ the isotropy subgroup of $\tilde{x}_0$ in $G$.

\begin{prop}
\label{prop_incl_norm}
We have $T\cap H \subset T\cap H_0$ and $T\cap N_G(H) \subset T\cap N_G(H_0)$.
\end{prop}

\begin{proof}
If $t\in T\cap H$ then 
\begin{align*}
t\cdot \tilde{x}_0 & = \lim_{z\rightarrow 0} (t,1) \cdot (\lambda(z),z^m)\cdot \tilde{x} \\
	& = \lim_{z\rightarrow 0} (\lambda(z),z^m)\cdot (t,1) \cdot \tilde{x} \\
	& = \tilde{x}_0.
\end{align*}

Recall from Section~\ref{sec_equi_aut} that the action of 
\[
N_G(H)/H = \mathrm{Aut}_G(G/H) \subset 
\mathrm{Aut}_{G\times \mathbb{C}^*}(G/H \times \mathbb{C}^*)
\] 
extends to the elementary embedding $E$. It is \emph{a priori} no longer 
explicit on $G/H_0$. We denote by 
$y \curvearrowleft nH$ the action of $nH \in N_G(H)/H$
on $y\in E$. 
We have $(n,1)\cdot \tilde{x} \curvearrowleft nH = \tilde{x}$.

Assume that $t \in T\cap N_G(H)$.
We then have 
\begin{align*}
(t,1)\cdot \tilde{x}_0 & = \lim_{z\rightarrow 0} (t,1)\cdot (\lambda(z),z^m)\cdot \tilde{x} \\
	& = \lim_{z\rightarrow 0} (\lambda(z),z^m)\cdot (t,1)\cdot \tilde{x} \\
	& = \lim_{z\rightarrow 0} (\lambda(z),z^m) \tilde{x} \curvearrowleft t^{-1}H \\
	& = \tilde{x}_0 \curvearrowleft t^{-1}H.
\end{align*} 
Since the action of $N_G(H)/H$ commutes with the action of $G$, 
the isotropy group of $\tilde{x}_0 \curvearrowleft t^{-1}H$ in $G$ is the same 
as the isotropy group of $\tilde{x}_0$, which is $H_0$. 
On the other hand, the isotropy subgroup of $(t,1)\cdot \tilde{x}_0$ in $G$ is 
$tH_0t^{-1}$, so we obtain that $H_0=tH_0t^{-1}$. In other words, $t\in N_G(H_0)$.
\end{proof}

Let us also highlight the relation between the linear part of 
the valuation cone, equivariant automorphisms, and 
equivariant degenerations of $G/H$ whose closed $G\times \mathbb{C}^*$-orbit 
$G/H_0$ is 
isomorphic to $G/H$.
We assume here that $H$ and $L$ are as in 
Proposition~\ref{prop_levi_choice}. 

Assume $(\lambda,m) \in 
\mathfrak{Y}(T\cap N_G(H))\cap \pi^{-1}(\mathcal{V}) \oplus \mathbb{N}^*$.
In this case, the equivariant degeneration of $G/H$ associated to the ray 
generated by $(\lambda,m)$ may be described explicitly 
\cite[Section 2.8]{BP87}. 
This is the quotient of $G/H\times \mathbb{C}^* \times \mathbb{C}$ 
by the action of $\mathbb{C}^*$ given by 
$t\cdot (gH,z, \theta) = (g\lambda(1/t)H,z/t^m,t\theta)$.
The open $G\times \mathbb{C}^*$-orbit is the image of the 
$G\times \mathbb{C}^* \times \mathbb{C}^*$-orbit of $(eH,1,1)$, isomorphic 
to $G/H\times \mathbb{C}^*$, and the closed $G\times \mathbb{C}^*$-orbit is the image 
of the $G\times \mathbb{C}^* \times \mathbb{C}^*$-orbit of $(eH,1,0)$, 
whose stabilizer in $G\times \mathbb{C}^*$ is 
$(\lambda(z),z^m)(H\times \{1\})$, and stabilizer in $G$ 
is $H$. 
In particular, the closed orbit is in this case isomorphic to $G/H$.

Remark that the construction above may be carried out whenever 
$\lambda \in \mathfrak{Y}(T\cap N_G(H))$ and $m\in \mathbb{N}^*$, 
which shows that $\mathfrak{Y}(T\cap N_G(H)) \otimes \mathbb{Q}$ 
is a $\mathbb{Q}$-vector space contained in $\pi^{-1}(\mathcal{V})$.
In fact, it is the maximal such vector space, also called the linear 
part of $\pi^{-1}(\mathcal{V})$, by \cite[Proposition 5.3]{BP87}. 
Let us give a statement summarizing this paragraph for future reference.

\begin{prop}
\label{prop_prod_deg}
Let $\lambda \in \mathfrak{Y}(T)\cap \pi^{-1}(\mathcal{V})$, 
and $m\in \mathbb{N}^*$. 
Then the following are equivalent:
\begin{itemize}
\item $-\lambda \in \mathfrak{Y}(T)\cap \pi^{-1}(\mathcal{V})$,
which means that $\lambda$ is in the linear part of $\pi^{-1}(\mathcal{V})$,
\item $\lambda \in \mathfrak{Y}(T\cap N_G(H))$,
\item the equivariant degeneration associated to the ray generated by 
$(\lambda,m)$ has closed orbit isomorphic to $G/H$. 
\end{itemize}
\end{prop}

Recall that $G/H$ is horospherical if and only if the valuation cone $\mathcal{V}$
is the full vector space $\mathcal{N}\otimes \mathbb{Q}$. We then have in particular:

\begin{cor}
If $G/H$ is horospherical then all equivariant degenerations of $G/H$ 
have closed orbit isomorphic to $G/H$.
\end{cor}

\subsubsection{Horospherical degenerations}

Let $Q$ denote the parabolic subgroup of $G$ opposite to $P$ with respect to the 
Levi subgroup $L$ containing the maximal torus $T$, and $Q^u$ its unipotent radical.

\begin{prop}
\label{prop_iso_horo}
We still use the notations from Proposition~\ref{prop_act}.
Assume that $\pi(\lambda)$ is in the interior of the valuation cone, then 
$H_0$ is horospherical for $G$, and $H_0= Q^u(L\cap H_0)$.
\end{prop}

\begin{proof}
This is essentially \cite[Proposition 3.10]{BP87} 
except we consider only the action of $G$ and not the full action of 
$G\times \mathbb{C}^*$. But it follows easily from this case. Indeed, 
if $\tilde{H}_0$ contains a maximal unipotent subgroup $U\times \{1\}$ 
of $G\times \mathbb{C}^*$, then $H_0$ obviously contains $U$, so it 
is horospherical. Furthermore, we have 
$\tilde{H}_0 = (Q^u\times \{1\})(L\times \mathbb{C}^*\cap \tilde{H}_0)$
and taking the intersection with $G\times \{1\}$ yields 
$H_0 = Q^u(L\cap H_0)$.
\end{proof}

\subsection{Equivariant test configurations for spherical varieties}

We will in this section apply the general theory of spherical varieties
to construct special equivariant test configurations for $\mathbb{Q}$-Fano 
spherical varieties. Combined with the description of the action of $\mathbb{C}^*$
on equivariant degenerations obtained in Proposition~\ref{prop_act}, we will 
have enough information about spherical test configurations for the proof of 
our main result. 

\subsubsection{Equivariant test configurations}

Let us first recall the definition of test configurations.

\begin{defn}
Let $(X,L)$ be a polarized normal projective variety. 

A \emph{test configuration} for $(X,L)$  
is a normal variety $\mathcal{X}$ with an action of $\mathbb{C}^*$, 
equipped with a $\mathbb{C}^*$-linearized line bundle $\mathcal{L}$ 
and a $\mathbb{C}^*$-equivariant flat morphism 
$\pi : \mathcal{X} \longrightarrow \mathbb{C}$ such that $\mathcal{L}$ is 
$\pi$-ample, the fiber 
$(\pi^{-1}(1),\mathcal{L}|_{\pi^{-1}(1)})$ is isomorphic to $(X,L^r)$ for 
some fixed integer $r>0$.

The \emph{scheme theoretic} fiber of $\pi$ over $0\in \mathbb{C}$ is called the 
\emph{central fiber} 
of $\mathcal{X}$ and denoted by $X_0$. The fiber $X_t$ over any 
$0\neq t\in \mathbb{C}$ is isomorphic to $X$ thanks to the action of 
$\mathbb{C}^*$.

A test configuration for $(X,L)$ is \emph{special} if the central fiber 
$X_0$ is a normal variety, in particular reduced and irreducible.

If a reductive group $G$ acts on $(X,L)$, we say that a test configuration
is $G$-\emph{equivariant} if $(\mathcal{X},\mathcal{L})$ admits an action of $G$ which 
commutes with the $\mathbb{C}^*$ action and such that the isomorphism between
$(X,L^r)$ and $(\pi^{-1}(1),\mathcal{L}|_{\pi^{-1}(1)})$ is $G$-equivariant.
\end{defn}

The definition extends \emph{verbatim} to normal varieties equipped with an 
ample $\mathbb{Q}$-line bundle. In the rest of the paper, we always consider 
$\mathbb{Q}$-Fano varieties, and the ample $\mathbb{Q}$-line bundle $L$ will 
always be the anticanonical $\mathbb{Q}$-line bundle $K_X^{-1}$, so we will 
omit it in the notations. Furthermore, if the test configuration is special, 
the induced $\mathbb{Q}$-line bundle $L_0$ on the central fiber $X_0$ is the anticanonical line bundle $K_{X_0}^{-1}$, so 
the central fiber $X_0$ is also a $\mathbb{Q}$-Fano variety (see for example
\cite[Lemma 2.2]{Ber16}).

Let $(X,L)$ be a polarized spherical variety under $G$. Then if 
$(\mathcal{X},\mathcal{L})$ is a $G$-equivariant test configuration, 
the normal variety $\mathcal{X}$ is a spherical variety under the action 
of $G\times \mathbb{C}^*$. More precisely, if $(X,x)$ is a spherical embedding 
of $G/H$, then $(\mathcal{X},\tilde{x})$ is a spherical embedding of 
$G\times \mathbb{C}^*/H\times \{1\}$, where $\tilde{x} \in \pi^{-1}(1)$ 
corresponds to $x$ in the isomorphism between $\pi^{-1}(1)$ and $X$.

The central fiber $X_0$ equipped with the reduced induced structure 
is a normal $G\times \mathbb{C}^*$-stable subvariety of the 
$G\times \mathbb{C}^*$-spherical variety $\mathcal{X}$, 
hence it is also a $G\times \mathbb{C}^*$-spherical variety. 
In particular, if we consider only the union of the open 
$G\times \mathbb{C}^*$-orbits in $\mathcal{X}$ and in $\pi^{-1}(0)$, 
we obtain an elementary embedding of $(G\times \mathbb{C}^*)/(H\times \{1\})$, 
which can be studied using Section~\ref{sec_elem_emb}. 
Whenever the test configuration is special, the scheme theoretic central 
fiber itself is a spherical variety.

Finally remark that since the action of 
$\mathrm{Aut}^0_{G\times \mathbb{C}^*}(G/H\times \mathbb{C}^*)
\supset \mathrm{Aut}^0_{G}(G/H)\times \{1\}$
extends to any $G/H\times \mathbb{C}^*$-embedding 
by Section~\ref{sec_equi_aut}, any $G$-equivariant 
test configuration for a spherical embedding $X$ of $G/H$
is also $G\times \mathrm{Aut}^0_{G}(X)$-equivariant.

\subsubsection{Construction}

\begin{thm}
\label{thm_exist_test}
Let $X$ be a $\mathbb{Q}$-Fano spherical embedding of $G/H$.
Let $m\in \mathbb{N}^*$, and let 
$\lambda \in \mathfrak{Y}(T)\cap \pi^{-1}(\mathcal{V})$.
Then there exists a $G$-equivariant test configuration 
$(\mathcal{X}, \mathcal{L})$ for $X$, 
with irreducible central fiber $X_0$, such that $X_0$ equipped with the 
reduced induced structure is a spherical embedding of $G/H_0$ 
and the action of $e^{\tau}\in \mathbb{C}^*$ on $G/H_0$ is given by 
$e^{\tau}\cdot gH_0 = g\lambda(e^{-\tau/m})H_0$. 
Furthermore, there exists $k\in \mathbb{N}^*$ such that the $G$-equivariant 
test configuration constructed for $(m, k\lambda)$ is special.
\end{thm}

\begin{proof}
The construction depends only on the line generated by $(\lambda,m)$, so 
we may assume that $(\lambda, m)$ is primitive.
The construction uses results of \cite{GH15} as recalled in Section~\ref{sec_GH}.
Let $d$ be the $\mathbb{Q}$-Cartier divisor representing $K_X^{-1}$ constructed there, 
and $n_D$ the coefficients of the colors.
Let $\Delta_d^*$ be the $\mathbb{Q}$-$G/H$-reflexive polytope associated 
to $X$, considered in $(\mathcal{N}\oplus \mathbb{Z})\otimes \mathbb{R}$.

We build from this polytope the colored fan $\mathcal{F}_{\mathcal{X}}$ 
of $\mathcal{X}$. 
Remark that the colors of $G/H$ are in bijection with the colors of 
$G/H\times \mathbb{C}^*$ by sending $D\in \mathcal{D}$ to $D\times \mathbb{C}^*$.
For every face $F$ of $\Delta_d^*$, consider the three colored cones
\begin{align*}
&(\mathrm{Cone}(F), \rho^{-1}(\mathrm{Cone}(F)), \\
&(\mathrm{Cone}(F\cup \{(0,-1)\}), \rho^{-1}(\mathrm{Cone}(F\cup \{(0,-1)\})) \\
&(\mathrm{Cone}(F\cup \{(\lambda,m)\}), \rho^{-1}(\mathrm{Cone}(F\cup \{(\lambda,m)\})).
\end{align*}
Among these cones, keep only the cones such that 
the intersection of the relative interior of their support with 
$\mathcal{V} \times \mathbb{Q}$ is non empty.
Then $\mathcal{F}_{\mathcal{X}}$ is the set of these cones.

It is clear that the corresponding spherical embedding of $G/H\times \mathbb{C}^*$ 
is complete. Furthermore, by the description of equivariant morphisms between 
spherical varieties, it admits a $G\times \mathbb{C}^*$-equivariant inclusion 
of $X\times \mathbb{C}^*$, and a $G \times \mathbb{C}^*$-equivariant surjective 
morphism $p$ to the spherical embedding $\mathbb{P}^1$ of the $G\times \mathbb{C}^*$ 
homogeneous space $G\times \mathbb{C}^*/G\times \{1\}\simeq \mathbb{C}^*$, such 
that $p^{-1}(\mathbb{C}^*)=X\times \mathbb{C}^*$.

We then build the line bundle. We chose to build $\mathcal{X}$ complete, 
and will choose a line bundle $\mathcal{L}$ that is ample.
Denote by $X_{\infty} = p^{-1}(\infty)$, respectively $X_0=p^{-1}(0)$, the fibers 
of $p$, corresponding as irreducible $G\times \mathbb{C}^*$-invariant divisors of $\mathcal{X}$
to the rays generated by $(0,-1)$, respectively $(\lambda,m)$. 
Remark that these are spherical varieties, and that they are the closures in 
$\mathcal{X}$ of equivariant degenerations of $G/H$. 
Consider the $B\times \mathbb{C}^*$-invariant Weil divisor
\[
\delta = a (\sum_{Y\in W_G} \overline{Y\times \mathbb{C}^*} 
+ \sum_{D\in \mathcal{D}} n_D \overline{D\times \mathbb{C}^*})
+ b (X_{\infty}+X_0) 
\]
where $a,b  \in \mathbb{N}$.

Let $v$ denote the support function of the dilated polytope $a\Delta_d = \Delta_{ad}$
and consider the function $l_{\delta}$ defined for 
$(x,n)\in \mathcal{N}\otimes \mathbb{Q} \oplus \mathbb{Q}$ by
\begin{itemize}
\item $l_{\delta}(x,n)= v(-x+ n/m\lambda) + bn/m$ if $n\geq 0$ and 
\item $l_{\delta}(x,n)= v(x) - bn$ if $n\leq 0$.
\end{itemize}
This is clearly a piecewise linear function on the support of the fan 
$\mathcal{F}_{\mathcal{X}}$. It furthermore satisfies the 
three conditions of Proposition~\ref{prop_ample}, provided $b> 0$ is large enough. 
It is however not integral in general.
By definition, it satisfies $l_{\delta}(0,-1)=b$ and 
$l_{\delta}(\lambda,m)=b$.
By the relation between $l_d$ and the support function of the polytope 
$\Delta_d$ recalled in Section~\ref{sec_pol_car}, 
we thus have 
\begin{align*}
\delta  = & \sum_{Y\in \mathcal{I}_X^G} l_{\delta}(u_Y,0) \overline{Y\times \mathbb{C}^*} 
+ \sum_{D\in \mathcal{D}_X} l_{\delta}(\rho(D),0) \overline{D\times \mathbb{C}^*} \\
& + \sum_{D\in \mathcal{D}\setminus \mathcal{D}_X} an_D \overline{D\times \mathbb{C}^*}
+ l_{\delta}(0,-1)X_{\infty}+ l_{\delta}(\lambda,m) X_0. 
\end{align*}

Let us now check that $l_{\delta}$ is integral piecewise linear for $a$ and $b$ divisible enough.
Denote by $m_F\in \mathcal{M}\otimes \mathbb{Q}$ the vertex of $\Delta_d$ corresponding 
to the maximal face $F$ of $\Delta_d^*$. 
Any maximal cone of $\mathcal{F}_{\mathcal{X}}$ is either the cone over 
the union of some maximal face $F$ of $\Delta_d^*$ and $(0,-1)$, in which case 
$l_{\delta}$ is given by the linear function $(-am_F,-b)$ on this 
cone, or it is the union of some maximal face $F$ of $\Delta_d^*$ and $(\lambda,m)$, in which case 
$l_{\delta}$ is given by the linear function $(-am_F,-am_F(\lambda)/m+b/m)$ on this 
cone.
Up to choosing $a$ and $b$ divisible enough, we can thus assume that 
$l_{\delta}$ is given by an element of $\mathcal{M}\oplus \mathbb{Z}$ on 
all maximal cones. 

This finally shows that $\delta$ is an ample Cartier divisor.
The corresponding ample line bundle $\mathcal{L}$ obviously restricts to 
$p_1^*K_X^{-a}$ on $X\times \mathbb{C}^*$, where $p_1$ is the first projection.
It remains to show that it may be equipped with a 
$G\times \mathbb{C}^*$-linearization such that its  
restriction to $X\times \mathbb{C}^*$ respects the natural linearization
of $p_1^*K_X^{-a}$.
Since $\mathcal{X}$ is normal, as a spherical variety, a tensor power $\mathcal{L}^n$
of $\mathcal{L}$ admits a $G\times \mathbb{C}^*$-linearization \cite{KKV89}. 
Furthermore, two such linearizations differ by a character of $G\times \mathbb{C}^*$. 
As a consequence, we may change the linearization of $\mathcal{L}^n$ 
to obtain the one we want.

The test configuration is obtained by considering 
$p^{-1}(\mathbb{C})$, equipped with the 
map $p$ and the restriction of $\mathcal{L}$.
If we consider the spherical embedding of $G/H\times \mathbb{C}^*$ 
given by the ray generated by $(\lambda,m)$, we obtain an equivariant 
degeneration of $G/H$, and the action of $\mathbb{C}^*$ on the closed 
orbit $G/H_0$, which is the open orbit of $G$ in the central fiber $X_0$, 
is as expected by Proposition~\ref{prop_act}.

The test configuration $\mathcal{X}$ constructed above is not special in general. 
However, the test configuration $\mathcal{X}_k$ obtained from 
$\mathcal{X}$ by the base change 
$\mathbb{C}\longrightarrow \mathbb{C}, z \mapsto z^k$ 
has reduced central fiber for $k$ divisible enough 
(see for example \cite[Proof of Proposition 7.15]{BHJ}). 
The effect on the $G\times C^*$ variety constructed above is simply 
obtained by keeping the same fan but changing the lattice 
$\mathcal{N}\oplus \mathbb{Z}$ to 
$\mathcal{N}\oplus \frac{1}{k}\mathbb{Z}$.
Hence the result.
\end{proof}

Taking any $\lambda$ projecting to the interior of the valuation cone 
yields, thanks to Proposition~\ref{prop_iso_horo}, 

\begin{cor}
\label{cor_horo_deg}
Any $\mathbb{Q}$-Fano spherical variety  admits a special test configuration 
with horospherical central fiber.
\end{cor}

\section{Modified Futaki invariant on $\mathbb{Q}$-Fano horospherical varieties}
\label{sec_fut_horo}

\subsection{Modified Futaki invariant on singular varieties}

Let $X$ be a normal $\mathbb{Q}$-Fano variety. 
Let $A$ be an ample line 
bundle on $X$ such that there exists $m\in \mathbb{N}^*$ such that the 
restriction of $A$ to the regular part 
$X_{\mathrm{reg}}$ of $X$ is $K_{X_{\mathrm{reg}}}^{-m}$.
The cohomology class $c_1(X)$ is the class $c_1(A)/m$.
We will always assume that $X$ has log terminal singularities. 
This is harmless in our setting since $\mathbb{Q}$-Fano (horo)spherical varieties 
have log terminal singularities \cite[Section 5.1]{AB04SRVII}. 
See \cite{Pas15} for a survey on singularities of spherical varieties. 

Let $\mathrm{Aut}^0(X)$ denote the connected component 
of the neutral element in the automorphism group of $X$. 
Since $X$ is a $\mathbb{Q}$-Fano variety, this is a linear algebraic group.
Choose a maximal compact subgroup $K$ in $\mathrm{Aut}^0(X)$. 
Its complexification $G$ 
is a Levi subgroup of $\mathrm{Aut}^0(X)$. 
Choose a maximal torus $T$ in $G$ such that $K\cap T$ is the maximal compact 
torus of $T$, and denote by $\mathfrak{a}$ the subspace $i\mathfrak{k}\cap \mathfrak{t}$ 
in  the Lie algebra of $G$.
Any element $\xi \in \mathfrak{a}$ is such that $J\xi \in \mathfrak{k}$ 
generates a compact torus in $K$. 

Let $q_A$ be a \emph{smooth}, \emph{positive} $K$-invariant hermitian metric on $A$.
On a normal variety, this means that the potentials of $q_A$ with respect to 
local trivializations of $A$ are the restrictions of smooth and strictly plurisubharmonic 
functions to $X$, given a local embedding of $X$ in $\mathbb{C}^N$.
Let $\omega_A\in 2\pi c_1(A)$ denote the curvature current of $q_A$, 
and $\omega=\omega_A/m \in 2\pi c_1(X)$. The current $\omega$ is still 
$K$-invariant, and defines a Kähler form on the regular part  $X_{\mathrm{reg}}$
of $X$.

Remark that on $X_{\mathrm{reg}}$, the anticanonical line bundle $K_X^{-1}$ is a well 
defined line bundle, and the metric $q_A$ induces a smooth positive metric $q$ on 
$K^{-1}_{X_{\mathrm{reg}}}$.
The form $\omega$ is the curvature form of $q$ on $X_{\mathrm{reg}}$. 
Let $dV_q$ denote the volume form on $X_{\mathrm{reg}}$ associated to $q$, 
defined by 
\[
dV_q(x)= |s^{-1}|^2_q s \wedge \bar{s}
\]
where $s$ is any non zero element of the fiber over $x$ of 
the canonical line bundle of $X_{\mathrm{reg}}$, and $s^{-1}$ is the 
corresponding non zero element of the dual line bundle $K_{X_{\mathrm{reg}}}^{-1}$.

We call \emph{Hamiltonian function} of $\xi \in \mathfrak{a}$ with respect to $\omega$
the smooth function $\theta_{\xi}$ on $X_{\mathrm{reg}}$ defined by
$L_{\xi}\omega = i\partial \bar{\partial} \theta_{\xi}$, 
where $L_{\xi}$ is the Lie derivative with respect to $\xi$. 
It is shown to exist and to be bounded in \cite{BWN}.
This function is well defined up to a constant and we fix this constant 
by the normalization condition:
\[
\int_{X_{\mathrm{reg}}} \theta_{\xi} dV_{q} = 0.
\]

\begin{defn}
The \emph{modified Futaki invariant} of $X$ with respect to $\zeta \in \mathfrak{a}$
is defined for $\xi\in \mathfrak{a}$ by
\[
\mathrm{Fut}_{X,\zeta}(\xi) =
- \int_{X_{\mathrm{reg}}} \theta_{\xi} e^{\theta_{\zeta}}\omega^n.
\]
\end{defn}

It is proved in \cite{BWN} that it is well defined and independent of the choice of 
the metric $q_A$. 

If we want to consider the K-stability of $X$ with respect to $G$-equivariant test 
configurations (for any reductive subgroup $G$ of $\mathrm{Aut}^0(X)$),
the only $\xi$ at which we will need to evaluate the Futaki invariant generate automorphisms 
commuting with $G$. Recall that we described these $G$-equivariant automorphisms when 
$X$ is spherical under the action of $G$ in Section~\ref{sec_equi_aut}. 

Assume that $X$ is horospherical under $G$, and that $x\in X$ is in the open orbit, 
with isotropy group $H$. 
The group of equivariant automorphisms of $X$ is isomorphic to 
$P/H$, where we recall that $P=N_G(H)$ is a parabolic, 
and the action of $P/H$ on the open orbit $G\cdot x= G/H$ is by multiplication 
on the right by the inverse. 
Recall that $P/H$ is a torus. Its maximal compact subtorus is thus well 
defined and we denote by 
\[
\mathfrak{b}_1 = \mathfrak{Y}(P/H) \otimes \mathbb{R}
\] 
the subalgebra 
of the Lie algebra of $P/H$ obtained as $J$ times the Lie algebra of the maximal 
compact subtorus, where $J$ denotes the complex structure on $P/H$. 
By Proposition~\ref{prop_polar_dec}, we can identify $\mathfrak{b}_1$ 
with $\mathfrak{a}_1$. We keep a different notation to emphasize that 
an element of $\mathfrak{b}_1$ acts on the right and not on the left.

\subsection{Hermitian metrics on polarized horospherical varieties}

Let $X$ be a projective embedding of a horospherical homogeneous space $G/H$, 
equipped with a 
$G\times (P/H)$-linearized ample line bundle $L$, where 
$P$ is the normalizer of $H$ in $G$.
Denote by $\chi$ the character of the isotropy subgroup $\mathrm{diag}(P) \subset G\times (P/H)$ 
associated to $L$.

Fix $B$ a Borel subgroup of $G$ containing $T$ and opposite to $P$, 
that is the opposite Borel subgroup $B^{-1}$ is a subgroup of $P$.
Let $\Delta^+$ be the moment polytope associated to the $G$-linearized 
line bundle $L$ with respect to $B$. 
Consider the polytope $\Delta = \chi|_T - \Delta^+$, and let 
$v_{2\Delta}$ denote the support function of the dilated polytope $2\Delta$.

\begin{prop}
\label{prop_asymptotic}
Let $q$ be a $K$-invariant locally bounded metric on $L$, and let
$u:\mathfrak{a}_1 \longrightarrow \mathbb{R}$ be the function associated to its 
restriction to $G/H$. 
Assume that $q$ is smooth and positive over $G/H$.
Then $u$ is a smooth and strictly convex function, 
and the function $u-v_{2\Delta}$ is bounded on 
$\mathfrak{a}_1$. 
\end{prop}

\begin{proof}
The fact that $u$ is smooth and strictly convex is 
obvious thanks to the expression of 
the curvature of the restriction of $q$ over $G/H$ obtained in 
Theorem~\ref{thm_curv_form}.

Let us now consider the \emph{discoloration} $\tilde{X}$ of $X$. 
This is the spherical embedding of $G/H$ corresponding to the fan 
$\mathcal{F}_{\tilde{X}}$ obtained as the union of the colored cones 
$(\mathcal{C},\emptyset)$, for all colored cones 
$(\mathcal{C},\mathcal{R}) \in \mathcal{F}_X$.
It is a complete toroidal embedding of $G/H$ which admits a natural 
$G\times P/H$-equivariant surjective homomorphism $\pi:\tilde{X}\longrightarrow X$, 
thanks to the description of spherical morphisms. 
The map $\pi$ is further an isomorphism between the open orbits, 
both identified with $G/H$. 

Consider the pull back $\tilde{L}:=\pi^*L$ of $L$ on $\tilde{X}$, with 
the $G\times (P/H)$-linearization induced by the linearization of $L$.
Its restriction to $G/H$ corresponds to the same character $\chi$ of $P$.
Let $s$ be a $B\times (P/H)$-semi-equivariant section of $L$, 
let $\mu$ denote the character of $B$ associated to $s$, and let $d$
be the $B$-invariant Cartier divisor defined by $s$. 

Consider the pulled back section $\tilde{s}$ of $\tilde{L}$. It is still 
$B\times (P/H)$-semi-invariant with the same $B$-character $\mu$, and its 
divisor $\tilde{d}$ is the pull back of the Cartier divisor $d$ defined by $s$. 
The function $l_{\tilde{d}}$ associated to $\tilde{d}$ coincide with 
the function $l_{d}$ (see \cite[Proof of Lemma 5.3]{Pas15}). 
On a \emph{toroidal} horospherical manifold, the function $l_{\tilde{d}}$ coincides with 
the function associated to the restriction of the Cartier divisor $\tilde{d}$
to the toric subvariety $\tilde{Z}:=\overline{P/H}\subset \tilde{X}$ under the 
action of the torus $P/H$, whose fan \emph{as a spherical variety}
is precisely the fan underlying the colored fan $\mathcal{F}_{\tilde{X}}$
\cite[Section 3.2]{Bri89}. Note that the toric fan of $\tilde{Z}$ with the 
usual conventions for toric varieties is the opposite to 
$\mathcal{F}_{\tilde{X}}$ (see \cite[Section~2]{Pez10} for details on this subtlety).

Since $L$ is ample, the polytope $\Delta^+$ is equal to 
$\mu + \Delta_d$, where $\Delta_d$ is the polytope with support function 
$x\longmapsto l_d(-x)$, associated to the ample Cartier divisor $d$. 
On the other hand, since $l_{\tilde{d}}=l_d$, the moment polytope of 
$\tilde{L}|_{\tilde{Z}}$ is $\eta + \Delta_d$ where $\eta$ is the character 
of $T/T\cap H$ associated to the restriction of $\tilde{s}$ to $\tilde{Z}$
(that is, $t\cdot\tilde{s}(t^{-1}\cdot x)=\tilde{s}(x)$ for $x\in\tilde{Z}$, 
$t\in T/T\cap H$). 
Remark that even if $\tilde{d}$ is not ample on $\tilde{X}$, 
its restriction to $\tilde{Z}$ is. 

Consider now the metric $q$ on $L$. 
The pulled back metric $\pi^*q$ on $\tilde{L}$ is still $K$-invariant and 
locally bounded. Furthermore, the associated function $\tilde{u}$ on 
$\mathfrak{a}_1$ is still $u$ since $\pi$ is an equality on $G/H$. 
Recall that with the notations of Section~\ref{sec_horo_space},
\begin{align*}
u(x)  & = -2\ln|s_r(\exp(x)H)|_{\tilde{q}} \\
	& = \left<2\chi, x \right> -2\ln|s_l(\exp(x)H)|_{\tilde{q}}.
\end{align*}

Consider the Batyrev-Tschinkel metric $q_0$ on $\tilde{L}|_{\tilde{Z}}$ \cite[Section 3.3]{Mai00}. 
This is a continuous hermitian metric on $\tilde{L}|_{\tilde{Z}}$, invariant under 
the action of the compact torus, with potential 
$u_0 : \mathfrak{a}_1  \longrightarrow \mathbb{R}$ such that
\[
u_0(x)=-2\ln|s_l(\exp(x)H)|_{q_0}=v_{-2(\eta + \Delta_d)}(x).
\]
Then since $\tilde{q}$ is locally bounded and $q_0$ is continuous, the function 
\[
u(x)-\left<2\chi, x \right>-v_{-2(\eta + \Delta_d)}(x) = 
-2\ln \frac{|s_l(\exp(x)H)|_{\tilde{q}}}{|s_l(\exp(x)H)|_{q_0}}
\] is bounded.

It remains to determine the polytope $\eta + \Delta_d$.
Recall that $\Delta^+ = \mu + \Delta_d$, hence the polytope above is 
$\eta - \mu + \Delta^+$.
It is then straightforward to check that $\mu-\eta$ 
is equal to $(x,y)\mapsto \left<\chi,y\right>$ for 
$(x,y)\in \mathfrak{a}_1 \oplus \mathfrak{X}(T/T\cap H)\otimes \mathbb{R} = \mathfrak{a}$.
Indeed, by definition of $\eta$ and $\mu$, $\eta=\mu$ on 
$\mathfrak{a}_1$, and by definition of $\chi$ and $\mu$, 
we have, for $y\in \mathfrak{X}(T/T\cap H)\otimes \mathbb{R}$, 
\[ e^{\left<\mu,y\right>}s(x) = \exp(y)\cdot s(\exp(-y)\cdot x) = \exp(y)\cdot s(x) = e^{\left<\chi,y\right>}s(x) \]
because $\exp(y)\in H\cap B$.  
We conclude then that $u-v_{2\Delta}$ is bounded, with $\Delta = \chi|_T - \Delta^+$.
\end{proof}

\begin{rem}
The idea to use the discoloration of $X$ was suggested to me by Boris Pasquier. 
If the closure of $P/H$ in $X$ is a normal toric variety, we can also deduce the 
moment polytope of the restriction of $L$ to $\overline{P/H}$ from the proof. 
\end{rem}

\subsection{Computation of the modified Futaki invariant} 

We fix $G$ a connected reductive complex group, $B$ a Borel subgroup 
of $G$, $T$ a maximal torus of $B$, and $K$ a maximal compact subgroup of $G$
such that $K\cap T$ is a maximal compact torus of $T$.
Let $X$ be a horospherical $\mathbb{Q}$-Fano variety, and choose a base point 
in $X$ such that its isotropy group $H$ contains the unipotent radical of $B^{-1}$.
Denote by $P$ the normalizer of $H$ in $G$.

Recall from Section~\ref{sec_GH} that the moment polytope of the 
ample $\mathbb{Q}$-line bundle $K_X^{-1}$ is the polytope obtained as the 
dual of the $\mathbb{Q}$-$G/H$-reflexive polytope associated to $X$, 
translated by $-2\rho_P$.

The weight associated to the anticanonical line bundle on $G/H$ is obtained 
as the weight of the action of $\mathrm{diag}(P)$ on the 
one dimensional representation  
$\mathrm{det}(\mathfrak{g}/\mathfrak{h})$.
It follows that this character is also 
\[
-2\rho_P = \sum_{\alpha\in \Phi_{P^u}}-\alpha.
\]

Let $m\in \mathbb{N}^*$ be such that $A=K_X^{-m}$ is an ample line 
bundle. Choose $q_A$ a smooth positive $K$-invariant metric on $A$.
Denote by $q = (q_A|_{G/H})^{1/m}$ the induced metric on $K_{G/H}^{-1}$.
Let $\omega$ be the curvature form of $q$ and $u$ be the convex potential 
of $q$.
Then $u-v_{2\Delta}$ is bounded, 
where $-\Delta = \Delta^+ + 2\rho_P$ and $\Delta^+$ is the moment polytope of 
the anticanonical $\mathbb{Q}$-line bundle. 

We will compute the Futaki invariant for all $\xi \in \mathfrak{b}_1$ for $X$.

\subsubsection{Computation of Hamiltonian functions}

\begin{prop}
Let $\xi \in \mathfrak{b}_1\simeq \mathfrak{Y}(T/T\cap H) \otimes \mathbb{R}$ 
and $\theta_{\xi}$ be the Hamiltonian function
of $\xi$ with respect to $\omega$. 
Let $\tilde{\xi} \in \mathfrak{Y}(T) \otimes \mathbb{R} \simeq \mathfrak{a}$
be any lift of $\xi$.
Then $\theta_{\xi}$ is the $K$-invariant
smooth function on $X_{\mathrm{reg}}$ defined for $x\in \mathfrak{a}_1$ by  
\[
\theta_{\xi}(\exp(x)H) = - \left\{ \nabla u(x), \tilde{\xi} \right\}.
\]
\end{prop}

\begin{proof}
First remark that since the function $\theta_{\xi}$ is smooth on $X_{\mathrm{reg}}$, 
and $G/H\subset X_{\mathrm{reg}}$ is dense, it is enough to work on $G/H$. 
By $K$-invariance it will be enough to obtain $\theta_{\xi}(\exp(x)H)$ for 
$x\in \mathfrak{a}_1$.

We will work on $G$ using the pullback by $\pi : G\longrightarrow G/H$.
We defined in Section~\ref{sec_horo_space} a function $\phi$ on $G$ as the 
potential of $\pi^*q$ with respect to a global left invariant section. 
This function is a global $i\partial \bar{\partial}$-potential for 
$\pi^*\omega$ on $G$ satisfying 
$\phi(k\exp(x)h)=u(x)-2\ln(\chi(\exp(x)h))$ for $k\in K$, $x\in \mathfrak{a}_1$ and $h\in H$.

Let $\tilde{\xi}$ be any lift of $\xi$, and denote by $L_{\tilde{\xi}}$
the Lie derivative with respect to the right-$G$-invariant vector 
field defined by $\tilde{\xi}$ on $G$.
We have on the one hand
\begin{align*}
L_{\tilde{\xi}}\pi^*\omega & = L_{\tilde{\xi}} i\partial\bar{\partial} \phi \\
	& =i\partial \bar{\partial} L_{\tilde{\xi}} \phi \\
\intertext{and on the other hand, }
L_{\tilde{\xi}}\pi^*\omega & = \pi^* L_{\xi} \omega \\
	& =\pi^* i\partial \bar{\partial} \theta_{\xi} \\
	& = i\partial \bar{\partial} (\theta_{\xi}\circ \pi).
\end{align*}

Let $g=k\exp(x)h \in G$, where $k\in K$, $x\in \mathfrak{a}_1$ and $h\in H$.
Recall that $\exp(t\xi)$ acts here by multiplication on the right by the inverse.
We compute 
\begin{align*}
(\exp(t\xi)^*\phi)(g) & =\phi(g\exp(-t\xi)) \\
	& = \phi(k\exp(x-t\xi)h') \\
\intertext{where $h'\in H$ since $P$ normalizes $H$}
	& = u(x-t\xi^1)-2\ln(\chi(\exp(x-t\xi)h')) \\
	& = u(x-t\xi^1)-2\ln(\chi(\exp(x)h))+2\left<\chi,t\xi\right>,
\end{align*}
where $\xi^1$ is the $\mathfrak{a}_1$ component of 
$\xi \in \mathfrak{a} = \mathfrak{a}_0 \oplus \mathfrak{a}_1$.

Using this we obtain 
\begin{align*}
(L_{\xi} \phi)(g) & = \left.\frac{d}{dt}\right|_0 (\exp(t\xi)^*\phi)(g) \\
	& = - \left\{ \nabla u(x), \xi^1 \right\} + 2\left<\chi,\xi\right>
\end{align*}
In particular, this function is $K\times H$ invariant (where $H$ acts by multiplication
on the right), just as $\pi^*\theta_{\xi}$.

The difference between the $L_{\xi}\phi$ and $\pi^*\theta_{\xi}$ 
is thus a $K\times H$ invariant smooth function on $G$
which is pluriharmonic on $G$. The only such functions are of the form 
$\psi(k\exp(x)h)=v(x)$ where $v$ is an affine function on $\mathfrak{a}_1$, 
but both $\pi^*\theta_{\xi}$ and $- \left\{ \nabla u(x), \xi^1\right\}$ are bounded 
so the difference between the two is a constant. 

Remark that $\nabla u(x) \in \mathfrak{a}_1$ is orthogonal to $\mathfrak{a}_0$, 
so we can replace $\xi^1$ with $\xi$.
It remains to check the normalization condition. 
We use the notation 
\[
\tilde{\theta}_{\xi}(\exp(x)H)=- \left\{ \nabla u(x), \xi\right\}.
\]
Note that we can restrict to $G/H$ since $X_{\mathrm{reg}}\setminus (G/H)$ is 
of codimension at least one in $X_{\mathrm{reg}}$.
Recall that $G/H$ is a fiber bundle with fiber $P/H$ over the generalized 
flag manifold $G/P = K/(K\cap P)$. We use fiber integration and $K$-invariance 
to obtain that: 
\[
\int_{G/H} \tilde{\theta}_{\xi} dV_{q}
 = C \int_{P/H} \tilde{\theta}_{\xi} |s_r|^2_q s_r^{-1}\wedge \overline{s_r}^{-1},
\]
where $C$ is a positive constant, and $s_r$ is the right-$P/H$-invariant section 
of the restriction $K_{G/H}^{-1}|_{P/H}$. Using invariance under the action of 
the maximal compact subtorus of $P/H$ on itself, we obtain that this is 
up to a positive constant, equal to 
\[
\int_{\mathfrak{a}_1} - \left\{ \nabla u(x), \xi\right\} e^{-u(x)} dx
\]
and the last expression vanishes.
Indeed, we have 
\[
- \left\{ \nabla u(x), \xi\right\} e^{-u(x)} =
\left\{ \nabla e^{-u(x)}, \xi\right\}
\]
and the Stokes formula yields the vanishing, because by boundedness of $u-v_{2\Delta}$,
$e^{-u}\leq Ce^{-v_{2\Delta}}$
and $\Delta$ contains 0 in its 
interior, so $e^{-v_{2\Delta}}$ has exponential decay.
\end{proof}

\subsubsection{Computation of the Futaki invariant}

\begin{thm}
\label{thm_fut_horo}
Let $X$ be a $\mathbb{Q}$-Fano horospherical embedding of $G/H$, 
let $\Delta^+$ be the moment polytope of $X$, and let $\zeta, \xi \in \mathfrak{b}_1$.
Let $\tilde{\zeta}$ and $\tilde{\xi}$ be lifts in $\mathfrak{a}$ of $\zeta$ and $\xi$.
Then 
\[
\mathrm{Fut}_{X,\zeta}(\xi) = 
C \left<-2\rho_P-\mathrm{bar}_{DH,\tilde{\zeta}}(\Delta^+), \tilde{\xi} \right>
\]
where 
$\mathrm{bar}_{DH,\tilde{\zeta}}$ is the barycenter with respect to the measure 
\[
e^{\left< 2p+4\rho_P, \tilde{\zeta} \right>}
\prod_{\alpha\in (-\Phi_{P^u})} \kappa(\alpha, p) dp
\]
with $dp$ the Lebesgue measure, 
and $C$ is a positive constant independent of $\xi$ and depending on $\zeta$
only via the volume of $\Delta^+$ with respect to the measure above. 
\end{thm}

\begin{proof}
We have to compute 
\[
\mathrm{Fut}_{X,\zeta}(\xi) =
- \int_{X_{\mathrm{reg}}} \theta_{\xi} e^{\theta_{\zeta}}\omega^n.
\]

Let us first use Theorem~\ref{thm_curv_form}
to express the volume form $\omega^n$ on $P/H$.
The theorem gives that, for $x\in \mathfrak{a}_1$,
\[
\omega_{\exp(x)H} =   
\sum_{1\leq j_1,j_2 \leq r} \frac{1}{4} \frac{\partial^2 u}{\partial x_{j_1}\partial x_{j_2}}(x)
\gamma_{j_1}\wedge \bar{\gamma}_{j_2}  
	 + 
\sum_{\alpha \in \Phi_{P^u}} \left< \alpha, \nabla u(x)/2-t_{\chi} \right>
\gamma_{\alpha} \wedge \bar{\gamma}_{\alpha}.
\]
We then obtain, recalling that $\chi=-2\rho_P$,
\[
\omega^n_{\exp(x)H} = 
	\frac{\mathrm{MA}_{\mathbb{R}}(u)(x) }{4^r2^{\mathrm{Card}(\Phi_{P^u})}}
	\prod_{\alpha\in \Phi_{P^u}} \left<\alpha, \nabla u(x)+4t_{\rho_P}\right> 
	\Omega
\]
where
\[
\Omega := \bigwedge_{1\leq j\leq r} \gamma_j \wedge \bar{\gamma}_{j} 
	\bigwedge_{\alpha\in \Phi_{P^u}} \gamma_{\alpha}\wedge \bar{\gamma}_{\alpha}.
\]

We now use the same techniques as in the previous proof, that is, $K$-invariance,
fiber integration on the fiber bundle $G/H\longrightarrow G/P$,
and integration of a compact torus-invariant function on a torus, to compute: 
\begin{align*}
\mathrm{Fut}_{X,\zeta}(\xi) & = - \int_{G/H} \theta_{\xi} e^{\theta_{\zeta}}\omega^n \\
	& =  C \int_{\mathfrak{a}_1} 
		\left\{ \nabla u(x), \tilde{\xi} \right\} e^{-\left\{ \nabla u(x), \tilde{\zeta} \right\}}
		\prod_{\alpha\in \Phi_{P^u}} \left<\alpha, \nabla u(x)+4t_{\rho_P} \right>  
		\mathrm{MA}_{\mathbb{R}}(u)(x) dx \\
\intertext{where $C>0$ is a constant independent of $\xi$ and $\zeta$.
Since $u$ is smooth and strictly convex, we can use the change of variables 
$p= -(d_xu+4\rho_P)/2$, where $d_xu \in \mathfrak{a}^*$ is the derivative of $u$ at $x$. 
Remark that this equivalently means $t_p = -(\nabla u(x)+4t_{\rho_P})/2$, 
and that the domain of integration after this change of variables becomes $\Delta^+$
by Proposition~\ref{prop_asymptotic}. Hence we obtain that }
\mathrm{Fut}_{X,\zeta}(\xi) & = C \int_{\Delta^+} 
		\left< -p-2\rho_P, \tilde{\xi} \right> 
		e^{\left< 2p+4\rho_P, \tilde{\zeta} \right>}       
		\prod_{\alpha\in \Phi_{P^u}} \left\{ \alpha, -p \right\} dp \\
		& = C \int_{\Delta^+} 
		\left< -p-2\rho_P, \tilde{\xi} \right> 
		e^{\left< 2p+4\rho_P, \tilde{\zeta} \right>}       
		\prod_{\alpha\in (-\Phi_{P^u})} \kappa(\alpha, p) dp \\
\intertext{since $\alpha$ is in the semisimple part,}
	& = C' \left<-2\rho_P-\mathrm{bar}_{DH,\zeta}(\Delta^+), \xi\right>.
\end{align*}
The final constant $C'>0$ depends on $\tilde{\zeta}$ only at the last step, where we 
multiplied $C$ by the volume of $\Delta^+$ 
with respect to the measure 
\[
e^{\left< 2p+4\rho_P, \tilde{\zeta} \right>}       
		\prod_{\alpha\in -\Phi_{P^u}} \kappa(\alpha, p) dp.
\]
\end{proof}

\section{Modified K-stability criterion for $\mathbb{Q}$-Fano spherical varieties}
\label{sec_criterion}

\subsection{Statement of the criterion}
We will prove in this section our main theorem.
The proof combines the results from Section~\ref{sec_test_config} and 
Section~\ref{sec_fut_horo}, together with an argument allowing us to 
compute the Futaki invariant of a $\mathbb{Q}$-Fano variety on an 
equivariant degeneration.

\subsubsection{Definition of modified K-stability}

Let $X$ be a $\mathbb{Q}$-Fano variety.
Given $\zeta \in \mathrm{Lie}(\mathrm{Aut}(X))$, 
let 
\[
T_{\zeta}=\{ \exp(z\zeta) ; z\in \mathbb{C} \}.
\]
It defines a subgroup of $\mathrm{Aut}^0(X)$ which may very well
not be closed. 
It is equipped with a privileged map 
$\mathbb{C} \rightarrow T_{\zeta}, z\mapsto \exp(z\zeta)$, 
and $\zeta$ may be recovered as the derivative at $1$ of the restriction 
of this map to $\mathbb{R}$.

\begin{defn}
Let $\zeta \in \mathrm{Lie}(\mathrm{Aut}(X))$ be a semisimple element 
such that $J\zeta$ generates a compact real subgroup. 
Let $(\mathcal{X},\mathcal{L})$ be a $\overline{T_{\zeta}}$-equivariant special 
test configuration for $X$ with central fiber $X_0$. 
Denote by $\zeta_0$ the element of $\mathrm{Lie}(\mathrm{Aut}(X_0))$ 
obtained as derivative at $1$ of the map obtained by composition:
\[\mathbb{R} \subset \mathbb{C} \rightarrow T_{\zeta} \rightarrow \mathrm{Aut}(X_0).\]
Denote 
by $\xi$ the element of $\mathrm{Lie}(\mathrm{Aut}(X_0))$ generating the action of 
$\mathbb{C}^*$ on $X_0$ induced by the test configuration. 
The \emph{modified Donaldson-Futaki invariant} of the test configuration is 
\[
\mathrm{DF}_{\zeta}(\mathcal{X},\mathcal{L})=\mathrm{Fut}_{X_0,\zeta_0}(\xi).
\]
\end{defn}

We assume now that $X$ is equipped with an action of a reductive group $G$, 
and fix $\zeta \in \mathrm{Lie}(\mathrm{Aut}(X))$ a semisimple element such 
that $J\zeta$ generates a compact real subgroup. We further assume that the 
action of $G$ commutes with $\overline{T_{\zeta}}$. 

\begin{defn}
The $\mathbb{Q}$-Fano variety $X$ is \emph{modified K-semistable (with respect to 
$G$-equivariant special test configurations)} for the vector field $\zeta$ if for any 
$G\times \overline{T_{\zeta}}$-equivariant special test configuration $(\mathcal{X},\mathcal{L})$ 
for $X$, $\mathrm{DF}_{\zeta}(\mathcal{X},\mathcal{L}) \geq 0$. 
It is \emph{K-stable (with respect to $G$-equivariant special test configurations)}
for the vector field $\zeta$
if furthermore $\mathrm{DF}_{\zeta}(\mathcal{X},\mathcal{L}) = 0$ only if the central 
fiber $X_0$ of $\mathcal{X}$ is isomorphic to $X$.
\end{defn}

\subsubsection{Statement}

Our main result is a criterion for modified K-stability in terms of combinatorial 
data describing a $\mathbb{Q}$-Fano spherical variety.

Let $X$ be a $\mathbb{Q}$-Fano variety, spherical under the action of a reductive group $G$. 
Let $B$ be a Borel subgroup of $G$, $T$ a maximal torus of $B$.

Let $\Delta^+ \subset \mathfrak{X}(T)\otimes \mathbb{R}$ be the moment polytope of $X$ 
with respect to $B$.
Let $\mathcal{M}$ be the subgroup of $\mathfrak{X}(T)$ as defined in Section~\ref{sec_intro_spher}, 
let 
$\mathcal{N} = \mathrm{Hom}(\mathcal{M}, \mathbb{Z})$ be the dual, and let 
$\mathcal{V} \subset \mathcal{N}\otimes \mathbb{R}$ be the valuation cone 
of $X$.
Denote by 
$\pi:\mathfrak{Y}(T)\otimes \mathbb{R} \longrightarrow \mathcal{N}\otimes \mathbb{R}$ 
the quotient map induced by the inclusion of $\mathcal{M}$ in $\mathfrak{X}(T)$.
Let us denote by $\Xi\subset \mathfrak{X}(T) \otimes \mathbb{R}$ the dual cone of 
$\overline{\pi^{-1}(-\mathcal{V})}$, and by $\mathrm{Relint}(\Xi)$ its relative 
interior.

Recall that we denote by $\Phi$ the root system of $(G,T)$ and $\Phi^+$ the positive roots
defined by $B$. 
Let $\Phi_{Q^u}$ denote the roots in $\Phi^+$ which are not orthogonal to $\Delta^+$, 
and let $2\rho_Q$ be the sum of the elements of $\Phi_{Q^u}$.

Let $\zeta$ be an element of the \emph{linear part} of 
$\overline{\mathcal{V}} \subset \mathcal{N}\otimes \mathbb{R}$.
Choose $\tilde{\zeta}\in \pi^{-1}(\zeta)$ any lift of $\zeta$ in 
$\mathfrak{Y}(T)\otimes \mathbb{R}$.
Recall that we denote by $\mathrm{bar}_{DH,\tilde{\zeta}}(\Delta^+)$ 
the barycenter of the polytope $\Delta^+$ with respect to the measure 
with density 
$p\mapsto e^{2\left<p-2\rho_Q,\tilde{\zeta}\right>}
\prod_{\alpha\in \Phi_{Q^u}} \kappa (\alpha,p)$ 
with respect to 
the Lebesgue measure $dp$ on $\mathfrak{X}(T) \otimes \mathbb{R}$.

\begin{thm}
\label{main_thm}
The variety $X$ is modified K-semistable (with respect to special 
$G$-equivariant test configurations) for $\zeta$  
if and only if 
\[
\mathrm{bar}_{DH,\tilde{\zeta}}(\Delta^+)\in 2\rho_Q + \Xi.
\]
It is modified K-stable (with respect to special 
$G$-equivariant test configurations) for $\zeta$ if and only if 
\[
\mathrm{bar}_{DH,\zeta}(\Delta^+)\in 2\rho_Q + \mathrm{Relint}(\Xi).
\]
\end{thm}

\subsection{Computing the Futaki invariant on a degeneration}

\subsubsection{Algebraic definition of the modified Futaki invariant}

Let us recall the algebraic definition of the modified Futaki invariant, given 
by Berman and Witt Nystrom \cite{BWN}.

Let $X$ be a $\mathbb{Q}$-Fano variety, and $A$ an ample line bundle on $X$ 
so that $A|_{X_{\mathrm{reg}}}= K_{X_{\mathrm{reg}}}^{-m}$. 
Consider $\xi, \zeta$ two semisimple elements in the Lie algebra of 
$\mathrm{Aut}^0(X)$ such that 
$J\xi$ and $J\zeta$ generate commuting compact real subgroups. 

The complex torus $\overline{T_{\xi}}\times \overline{T_{\zeta}}$ acts on 
$H^0(X,A^k)$ for all $k\in \mathbb{N}$. This action is diagonalizable.
Choose a basis $(e_j)_{j=1}^{N_k}$ of $H^0(X,A^k)$, such that 
$\overline{T_{\xi}}\times \overline{T_{\zeta}}$ acts on $\mathbb{C}e_j$ 
through the character $\chi_j$.
Define $s_j= \left< \chi_j, (\xi,0)\right>$ and 
$r_j= \left< \chi_j, (0,\zeta)\right>$, where $\left<,\right>$ denotes 
the duality pairing 
$(\mathfrak{X}(\overline{T_{\xi}}\times \overline{T_{\zeta}})\otimes \mathbb{R}) 
\times 
(\mathfrak{Y}(\overline{T_{\xi}}\times \overline{T_{\zeta}})\otimes \mathbb{R})
\rightarrow \mathbb{R}$.

The \emph{quantized 
modified Futaki invariant at level $k$} is defined by:
\[
\mathrm{Fut}_{X,\zeta}^{(k)}(\xi)= - \sum_{j=1}^{N_k} \exp(\frac{r_j}{mk})s_j.
\] 
The algebraic definition of the modified Futaki invariant is obtained 
thanks to the following result.

\begin{prop}\cite[Proposition 4.7]{BWN}
\label{quantfut}
The modified Futaki invariant of $X$ is obtained as the following limit:
\[
\mathrm{Fut}_{X,\zeta}(\xi) = 
\lim_{k\rightarrow \infty} \frac{1}{mkN_k}\mathrm{Fut}_{X,\zeta}^{(k)}(\xi).
\]
\end{prop}

\subsubsection{Equivariant degenerations and space of sections of linearized line bundles}

The fact that we can compute the modified Futaki invariant of a $\mathbb{Q}$-Fano 
variety on an equivariant degeneration will be a consequence of the following result.
Our proof of this statement consists of a little twist in the method used by Donaldson 
to compute the Futaki invariant on toric test configurations in \cite{Don02}, also 
used by Li and Xu to prove an intersection formula for the Donaldson-Futaki invariant \cite{LX14}.

\begin{prop}
\label{contrep}
Let $\mathcal{X}$ be a projective $G$-variety, where $G$ is a reductive group, 
and let $\mathcal{L}$ be a $G$-linearized line bundle on $\mathcal{X}$.
Let $p : \mathcal{X} \longrightarrow \mathbb{P}^1$ be a surjective 
$G$-invariant morphism, and assume that $\mathcal{L}$ is relatively ample 
with respect to $p$. For $t\in \mathbb{P}^1$, let $X_t=p^{-1}(t)$ denote 
the scheme theoretic central fiber and 
$L_t= \mathcal{L}|_{X_t}$. Let $t_1,t_2 \in \mathbb{P}^1$. 
Then for large enough $k\in \mathbb{N}$, the representation of $G$
given by $H^0(X_{t_1},L_{t_1}^k)$ is isomorphic, as a representation of $G$, to $H^0(X_{t_2},L_{t_2}^k)$.
\end{prop}

\begin{proof}
Let us first show that we can assume that $\mathcal{L}$ is ample on $\mathcal{X}$.
By \cite[II.7.10]{Har77}, there exists an $N\in \mathbb{N}$ such that 
$\mathcal{L} \otimes p^*(O_{\mathbb{P}^1}(N))$ is ample on $\mathcal{X}$. 
Furthermore, we have $H^0(X_t,L_t)\simeq H^0(X_t,(\mathcal{L} \otimes p^*(O_{\mathbb{P}^1}(N)))|_{X_t})$ equivariantly for the action of $G$.

From now on we assume that $\mathcal{L}$ is ample.

Consider on $\mathbb{P}^1$ the divisor given by a point $t\in \mathbb{P}^1$.
Choose a global section $s_t$ of $\mathcal{O}(1)$ with zero divisor $\{t\}$.
Consider the pull back $p^*\mathcal{O}(1)$ of this line bundle, and the pulled back 
global section $p^*s_t$ of $p^*\mathcal{O}(1)$. The zero divisor of $p^*s_t$ is the 
divisor $X_t$, 
and since $p$ is $G$-invariant, the section $p^*s_t$ is $G$-invariant. 

The divisor $X_t$ provides, for all $k\in \mathbb{N}$, an exact sequence: 
\[
1 \longrightarrow \mathcal{L}^k\otimes \mathcal{O}(-X_t)
	\longrightarrow \mathcal{L}^k
	\longrightarrow  \mathcal{L}^k|_{X_t}
	\longrightarrow 1
\]

If we fix $t$, then for large $k$, we have by Serre vanishing Theorem, 
\[
H^1(X,\mathcal{L}^k\otimes \mathcal{O}(-X_t))=0,
\]  
so the short exact sequence of line bundles gives the following exact sequence 
in cohomology:
\[
0 \longrightarrow H^0(\mathcal{X},\mathcal{L}^k\otimes \mathcal{O}(-X_t))
	\longrightarrow H^0(\mathcal{X},\mathcal{L}^k)
	\longrightarrow H^0(X_t,L_t^k)
	\longrightarrow 0
\]
where the first map is given by multiplication by the global section $p^*s_t$ of 
$\mathcal{O}(X_t)$.

Remark that $H^0(\mathcal{X},\mathcal{L}^k\otimes \mathcal{O}(-X_t))=
H^0(\mathcal{X},\mathcal{L}^k\otimes p^*\mathcal{O}(-1))$ is independent of $t$.
We can thus rewrite the exact sequence as: 
\[
0\longrightarrow H^0(\mathcal{X},\mathcal{L}^k\otimes p^*\mathcal{O}(-1))
\longrightarrow H^0(\mathcal{X},\mathcal{L}^k)
\longrightarrow H^0(X_t,L_t^k)
\longrightarrow 0
\]

It is then clear from the exact sequence that the dimension of $H^0(X_t,L_t^k)$
is independent of $t$ for $k$ large enough.
Let us now consider the structure of $G$-representation. The second map is given 
by restriction to $X_t$, which is $G$-stable, so we obtain that $H^0(X_t,L_t^k)$
is the quotient representation of $H^0(\mathcal{X},\mathcal{L}^k)$ by the image of 
$H^0(\mathcal{X},\mathcal{L}^k\otimes p^*\mathcal{O}(-1))$.
But for any $t\in \mathbb{P}^1$, our choice of a section $p^*s_t$ so that $p^*s_t$ 
is $G$-invariant, ensures that the inclusion of 
$H^0(\mathcal{X},\mathcal{L}^k\otimes p^*\mathcal{O}(-1))$ in 
$H^0(\mathcal{X},\mathcal{L}^k)$ is $G$-equivariant. 
In particular, the images $H^0(\mathcal{X},\mathcal{L}^k\otimes p^*\mathcal{O}(-1))$ in 
$H^0(\mathcal{X},\mathcal{L}^k)$ for different values of $t$ are isomorphic as 
$G$-representations. Be reductivity of $G$, the quotient representations are 
thus isomorphic.

Given $t_1,t_2 \in \mathbb{P}^1$, we thus obtain that for large enough $k$, 
the structure of $G$-representation 
of $H^0(X_t,L_t^k)$ is independent of $t\in \{t_1,t_2\}$.
\end{proof}

Let $X$ be a normal $\mathbb{Q}$-Fano variety, and $(\mathcal{X},\mathcal{L})$ be 
a special test configuration for $X$. Let $X_t$ denote the fiber $\pi^{-1}(t)$.
Assume that $(\mathcal{X},\mathcal{L})$ is equipped with actions of two commuting 
complex tori $\overline{T_{\zeta}}$ and $\overline{T_{\xi}}$, generated by the holomorphic 
vector field $\zeta$, respectively $\xi$, such that the map 
$\pi : \mathcal{X} \longrightarrow \mathbb{C}$ is \emph{invariant} under both actions. 
Let $\zeta_t$, respectively $\xi_t$, denote the holomorphic vector fields on $X_t$
induced by the restriction of the action of $T_{\zeta}$ respectively $T_{\xi}$, 
identified with elements of $\mathrm{Lie}(\mathrm{Aut}^0(X_t))$. 

\begin{cor}
\label{cor_const_fut}
The function 
$\mathbb{C} \longrightarrow \mathbb{R}; t \longmapsto \mathrm{Fut}_{X_t,\zeta_t}(\xi_t)$
is constant.
\end{cor}

\begin{proof}
By assumption, $(\mathcal{X},\mathcal{L})$ is a special test configuration for $X$. 
It is a standard procedure (see for example \cite[Section 8.1]{LX14}) 
to extend this test configuration to a family   
$\pi : \mathcal{X} \longrightarrow \mathbb{P}^1$ that we denote by the same letter,
still equipped with a $\pi$-ample line bundle $\mathcal{L}$.
We can further do this procedure in a way that is invariant under 
$\overline{T_{\zeta}}\times \overline{T_{\xi}}$.

Apply Proposition~\ref{contrep} to this family. 
Given $t_1,t_2\in \mathbb{P}^1$, we obtain that the 
$\overline{T_{\zeta}}\times \overline{T_{\xi}}$-representations 
$H^0(X_{t_1},A_{t_1}^k)$ and $H^0(X_{t_2},A_{t_2}^k)$ are isomorphic for $k$ large enough.
In particular, the quantized modified Futaki invariants are equal for $k$ large 
enough:
\[
\mathrm{Fut}_{X_{t_1},\zeta_{t_1}}^{(k)}(\xi_{t_1})=
\mathrm{Fut}_{X_{t_2},\zeta_{t_2}}^{(k)}(\xi_{t_2}).
\]
By Proposition~\ref{quantfut}, this implies that the modified Futaki 
invariants of $X_{t_1}$ and $X_{t_2}$ are equal, hence the result.
\end{proof}

\subsection{Proof of Theorem~\ref{main_thm}}

Let us denote by $Q$ the stabilizer of the open orbit of $B$ in $X$.
By Proposition~\ref{prop_levi_choice}, 
we can choose a base point $x$, whose isotropy group is denoted by $H$, 
and a Levi subgroup $L$ of $Q$ adapted to $H$, containing $T$,
such that $N_G(H)=H(T\cap N_G(H))$.
In particular, the following are isomorphic 
\[
\mathrm{Aut}^0_G(X) \simeq (N_G(H)/H)^0 \simeq (T\cap N_G(H)/T\cap H)^0.
\]

The lift $\tilde{\zeta}$ of $\zeta$ in  
$\mathfrak{Y}(T) \otimes \mathbb{R}$ is thus in 
$(T\cap N_G(H))\otimes\mathbb{R}$, 
and the action of $T_{\zeta}$ on the open orbit identified with 
the coset space $G/H$ is given by:
\[ 
\exp(\tau \zeta) \cdot gH 
= g\exp(-\tau \tilde{\zeta})H.
\]

Let $\mathcal {X}'$ be a special $G$-equivariant test configuration for 
the $\mathbb{Q}$-Fano variety $X$. Recall that it is automatically equivariant 
under the action of $\mathrm{Aut}_G(X)$ which contains $T_{\zeta}$.
Its central fiber, denoted by $X_0'$, is a spherical $\mathbb{Q}$-Fano 
variety under the action of $G$ by Proposition~\ref{prop_act}. 
Let $(\lambda, m) \in \mathfrak{Y}(T) \otimes \mathbb{N}^*$ be such that 
$y=\lim_{z \rightarrow 0} (\lambda(z),z^m)\cdot(x,1)$ is in the open orbit 
of $B$ in $X_0'$, and denote the isotropy group of $y$ by $H_0'$. 

Then, by Proposition~\ref{prop_act} again, 
the action of $\mathbb{C}^*=T_{\xi}$ on $X_0'$ induced by the test 
configuration is described on $G/H_0'$ by
\[ 
e^{\tau} \cdot gH_0' 
= g\exp(-\tau \tilde{\xi})H_0',
\]
where $\tilde{\xi} = \lambda/m$.
Furthermore, by Proposition~\ref{prop_incl_norm},
the action of $T_{\zeta}$ on $X_0'$ is still given, on $G/H_0'$, by 
\[ 
\exp(\tau\zeta) \cdot gH_0' 
= g\exp(-\tau \tilde{\zeta})H_0'.
\]

Let us now choose $\mathcal{X}$ any $G$-equivariant special test configuration for the 
$\mathbb{Q}$-Fano spherical variety $X_0'$ with 
horospherical central fiber, using Corollary~\ref{cor_horo_deg}. 
Let $y_0$ be the base point obtained from $y\in X_0'$ through the choice 
of an adapted one parameter subgroup, and let $H_0$ be its isotropy group.
Using Proposition~\ref{prop_still_adapted} and Proposition~\ref{prop_incl_norm},
we see that the action of $T_{\zeta}$ and $T_{\xi}$ are still given by multiplication 
on the right by $\exp(-\tau \tilde{\zeta})$, respectively $\exp(-\tau \tilde{\xi})$,
on $G/H_0$.
Furthermore, by Proposition~\ref{prop_iso_horo}, $H_0=P^u(H_0\cap T)$, 
where $P$ is the parabolic subgroup of $G$, opposite to $Q$ with respect to $T\subset L$.  
Remark that $P$ is also the normalizer of $H_0$ in $G$.
Remark also that the roots of the Levi subgroup $L\subset Q$ are determined by the moment 
polytope $\Delta^+$, as the roots that are orthogonal 
to $\Delta^+$ (see \cite[Section 4.2]{Bri89}). Thus $\Phi_{Q^u}$ and $2\rho_Q$ as defined 
before the statement of Theorem~\ref{main_thm} coincide with the data $-\Phi_{P^u}$ 
and $-2\rho_P$ associated to $P$.

The moment polytope of $X_0$ with respect to $B$ is the same as the moment
polytope of $X$ and $X_0'$. This is a known fact for equivariant polarized 
degenerations of spherical varieties, and can be recovered using 
Proposition~\ref{contrep}.
We can now apply our computation of the Futaki invariant on $\mathbb{Q}$-Fano
horospherical varieties to obtain:
\begin{align*} 
\mathrm{DF}_{\zeta}(\mathcal{X},\mathcal{L}) & = \mathrm{Fut}_{X_0',\zeta}(\xi) \\
	& = \mathrm{Fut}_{X_0,\zeta}(\xi) \\
\intertext{by Corollary~\ref{cor_const_fut}, then by Theorem~\ref{thm_fut_horo}, this is}
	& = C \left<2\rho_Q-\mathrm{bar}_{DH,\tilde{\zeta}}(\Delta^+), \tilde{\xi} \right>  
\end{align*}
where $C$ is some positive constant.
Furthermore, we have seen that for horospherical $\mathbb{Q}$-Fano varieties, 
$\mathrm{bar}_{DH,\tilde{\zeta}}(\Delta^+)$ and the quantity above do not depend 
on the choice of lifts $\tilde{\zeta}$ and $\tilde{\xi}$ of 
$\zeta$, $\xi \in \mathfrak{Y}(T/T\cap H_0)\otimes \mathbb{R}$.
Since $T\cap H\subset T\cap H_0' \subset T\cap H_0$, we obtain that the condition
obtained does no depend on the choice of a lift $\tilde{\zeta}$ of 
$\zeta \in \mathcal{N} \otimes \mathbb{R}=\mathfrak{Y}(T/T\cap H)\otimes \mathbb{R}$.

By Theorem~\ref{thm_exist_test}, for any choice of 
$(\lambda,m)\in \mathfrak{Y}(T) \oplus \mathbb{N}^*$ such that 
$\lambda$ projects to the valuation cone $\mathcal{V}$, 
there exists $k\in \mathbb{N}^*$ and a $G$-equivariant special test 
configuration for $X$ with $\tilde{\xi}= k\lambda/m$. 
We then obtain that $X$ is K-semistable (with respect to special equivariant test 
configurations) if and only if 
\[
\left<2\rho_Q-\mathrm{bar}_{DH,\tilde{\zeta}}(\Delta^+), \tilde{\xi} \right>\geq 0
\]
for all $\tilde{\xi} \in \pi^{-1}(\mathcal{V})$ where 
$\pi : \mathfrak{Y}(T)\otimes\mathbb{R} \longrightarrow \mathcal{N}\otimes \mathbb{R}$. 
This means precisely
that $\mathrm{bar}_{DH,\tilde{\zeta}}(\Delta^+)-2\rho_Q$ is in the dual 
cone $\Xi$ to $\pi^{-1}(-\mathcal{V})$.

Furthermore, $X_0'$ is isomorphic to $X$ if and only if 
if and only if $\xi$ projects to an element of the linear part of the valuation
cone. 
Indeed, if $X_0'$ is isomorphic to $X$, then the action induced by $\xi$ on $X$ 
is $G$-equivariant, so $\xi$ projects to an element of the linear part of the valuation
cone. Conversely, given such a $\xi$, the corresponding 
equivariant degeneration of $G/H$
satisfies $H=H_0'$ by Proposition~\ref{prop_prod_deg}. 
Additionally, the moment polytope of $X_0'$ 
is the same as the moment polytope of $X$, and they are both $\mathbb{Q}$-Fano. 
Hence by the results of Gagliardi and Hofscheier recalled in Section~\ref{sec_GH}, 
$X$ and $X_0'$ are equivariantly isomorphic.

It is clear that
$\left<2\rho_Q-\mathrm{bar}_{DH,\tilde{\zeta}}(\Delta^+), \xi \right> = 0$
if and only if $2\rho_Q-\mathrm{bar}_{DH,\tilde{\zeta}}(\Delta^+)$ lies in one of 
the hyperplanes defining the dual cone $(\pi^{-1}(\mathcal{V}))^{\vee}$.
It is always the case if $\xi$ projects to the linear part of $\mathcal{V}$.
On the other hand, if $-\xi \notin \pi^{-1}(\mathcal{V})$, then it means
that $2\rho_Q-\mathrm{bar}_{DH,\tilde{\zeta}}(\Delta^+)$ is on the boundary of 
the cone $(\pi^{-1}(\mathcal{V}))^{\vee}$. 
This finishes the proof of Theorem~\ref{main_thm}.

\subsection{Examples} 

Let us give some illustrations of new situations where our main 
result can be applied. We restrict to smooth varieties for simplicity. 

\subsubsection{Horospherical varieties}

Since the valuation cone is the whole space for horospherical 
varieties, the K-stability criterion becomes very simple in that case.
We thus obtain a generalization of the main result of \cite{PS10}.

\begin{cor}
Let $X$ be a smooth and Fano horospherical variety, with moment polytope 
$\Delta^+$. Then $X$ admits a Kähler-Ricci soliton. Furthermore, 
the following are equivalent:
\begin{itemize}
\item $X$ is Kähler-Einstein, 
\item $X$ is K-stable,
\item $X$ is K-semistable,
\item the Futaki invariant of $X$ vanishes,
\item $\mathrm{bar}_{DH}(\Delta^+) = 2\rho_Q.$
\end{itemize}
\end{cor}

Examples of interesting smooth, Fano, colored horospherical varieties, 
(thus not treated by \cite{PS10}), are given by Pasquier in \cite{Pas09}. 
Indeed, he classifies in this article horospherical manifolds with 
Picard number one. These are necessarily Fano, and colored when they are 
not homogeneous under a larger group. There exists in fact two infinite 
families (and three additional examples) of such manifolds 
\cite[Theorem 0.1]{Pas09}. Pasquier shows that their automorphism group 
is non reductive, which implies that they admit no Kähler-Einstein 
metrics and are not K-stable. 

Our result shows that they are not K-semistable, and admit Kähler-Ricci 
solitons. 
The first conclusion is especially interesting in the following context. 
A conjecture by Odaka and Odaka \cite[Conjecture 5.1]{OO13}, 
saying that a Fano manifold with Picard 
number one should be K-semistable, has been disproved by Fujita 
\cite{Fuj15}, who provided two counterexamples. 
The examples of Pasquier provide infinite families of counterexamples, 
showing that Fujita's examples are not exceptional.

\subsubsection{Symmetric varieties}
\label{sec_sym}

An important class of spherical varieties is given by symmetric varieties.
The combinatorial data from Section~\ref{sec_intro_spher} have a nice 
expression in the case of symmetric spaces, that we recall here.  
The reference for this point of view is \cite{Vus90}, which is 
summarized in \cite[Section 3.4.3]{Per14}.
Ruzzi \cite{Ruz12} obtained a classification of smooth and Fano symmetric 
varieties of small rank. We use this, together with \cite{GH15}, 
to obtain explicit examples of moment polytopes of smooth and Fano 
symmetric varieties.

Let $\sigma$ be a group involution of $G$, 
Let $G^{\sigma}$ be the fixed point set of $\sigma$, and $H$ a closed subgroup 
such that 
$G^{\theta} \subset H \subset N_G(G^{\theta})$.
The subgroup $H$ is called a \emph{symmetric subgroup} and $G/H$ is 
called a \emph{symmetric space}.
A $G$-equivariant normal embedding of $G/H$ is called a \emph{symmetric variety}.

A torus $S$ in $G$ is \emph{split} if $\sigma(s)=s^{-1}$ for $s\in S$.
Let $S$ be a split torus in $G$, maximal for this property, and let $T$ 
be a maximal torus of $G$ containing $S$. 
Then we have $\sigma(T)=T$, and $\sigma$ descends to an involution  
of $\mathfrak{X}(T)$, still denoted by $\sigma$. 

Recall that $\Phi$ denotes the root system of $G$ with respect to $T$. 
There exists a Borel subgroup $B$ of $G$ containing $T$ such that 
for all $\alpha \in \Phi^+$, either $\sigma(\alpha)=\alpha$ or 
$\sigma(\alpha)$ is a negative root.
The subset $BH$ is open dense in $G$, in particular $H$ is spherical. 

Let $\Phi^{\sigma}$ denote the set of $\alpha \in \Phi$ such that 
$\alpha = \sigma(\alpha)$. It is a sub root system of $\Phi$.
Let $\Psi^+ = \Phi^+ \setminus \Phi^{\sigma}$, and 
\[
2\rho_{\sigma} = \sum_{\alpha \in \Psi^+} \alpha.
\]

The set 
\[
\Phi_{\sigma}= \{\alpha - \sigma(\alpha) ; \alpha\in \Phi\setminus \Phi^{\sigma} \}
\]
is a (possibly non-reduced) root system in 
$\mathfrak{X}(T/T\cap H)\otimes \mathbb{R}$. 

We have $\mathcal{M} = \mathfrak{X}(T/T\cap H)$ 
and the valuation cone $\mathcal{V}$ with respect to $B$ is the negative 
Weyl chamber $-C_{\sigma}^+$ of the root system $\Phi_{\sigma}$ in 
$\mathfrak{Y}(T/T\cap H)\otimes \mathbb{Q}$. 

\begin{exa}
Consider the involution $\sigma$ of $G=\mathrm{SL}_n(\mathbb{C})$ 
defined by sending a matrix $g$ to the inverse of its transpose 
matrix $\sigma(g)=(g^t)^{-1}$.
By definition, the set of fixed point is 
$\mathrm{SO}_n(\mathbb{C})$. 
As a consequence, any subgroup of $G$ between $\mathrm{SO}_n(\mathbb{C})$ 
and $N_G(\mathrm{SO}_n(\mathbb{C}))$ is a symmetric subgroup.

The maximal torus $T$ in $\mathrm{SL}_n(\mathbb{C})$ formed by diagonal 
matrices is split. We choose as Borel subgroup $B$ the subgroup of upper 
triangular matrices.
Then $\Phi^{\sigma}$ is empty and $\Phi_{\sigma}$ is the 
root system formed by the roots $2\alpha$ for $\alpha\in \Phi$.
In particular, the positive Weyl chamber $C_{\sigma}^+$ of the root 
system $\Phi_{\sigma}$ is also the positive Weyl chamber of $\Phi$. 
Hence the valuation cone with respect to $B$ is the negative Weyl chamber 
of $\Phi$ and its dual is the cone generated by the negative roots $-\Phi^+$.

For $H=G^{\sigma}=\mathrm{SO}_n(\mathbb{C})$, the lattice 
$\mathcal{M} = \mathfrak{X}(T/T\cap H)$ is the lattice 
formed by the set of characters $2\chi$, where $\chi$ is a character of $T$
(or equivalently $\chi$ is an element of the weight lattice of $\Phi$).

For $H=N_G(G^{\sigma})=N_{\mathrm{SL}_n(\mathbb{C})}(\mathrm{SO}_n(\mathbb{C}))$,
the lattice $\mathcal{M} = \mathfrak{X}(T/T\cap H)$ is the lattice 
formed by the set of characters $2\chi$, where $\chi$ is an element of the 
lattice generated by roots in $\Phi$.
\end{exa}

\begin{cor}
\label{cor_sym}
Let $X$ be a smooth and Fano symmetric embedding of $G/H$, with moment polytope
$\Delta^+$. Then with the notations introduced above, $X$ is 
Kähler-Einstein if and only if the barycenter of $\Delta^+$ with respect 
to the measure 
$
\prod_{\alpha \in \Psi^+} \kappa(\alpha, p) dp
$
is in the relative interior of the translated cone 
$2\rho_{\sigma} + (C_{\sigma}^+)^{\vee}$.
\end{cor}

\begin{exa}
The classification of spherical projective varieties under the action of 
$\mathrm{SL}_2(\mathbb{C})$ is explained in \cite[Example 1.4.3]{AB06}. 
Two of them are symmetric varieties, and the other are horospherical.
Let us consider the symmetric ones.

The first one is $\mathbb{P}^1\times \mathbb{P}^1$, where 
$\mathrm{SL}_2(\mathbb{C})$ acts diagonally. There are only two orbits. 
The closed orbit is the diagonal and the open orbit is the complement 
of the diagonal. The open orbit is isomorphic to the symmetric space 
$\mathrm{SL}_2(\mathbb{C})/\mathrm{SO}_2(\mathbb{C})$.
The moment polytope of this Fano variety with respect to the Borel 
subgroup $B$ of upper triangular matrices is the following.

\begin{center}
\begin{tikzpicture}
\draw (0,-0.3) node{$0$};
\draw (2,-0.3) node{$2\rho$};

\draw [semithick, |-|] (0,0) -- (4,0);
\draw (4.3,0) node{$\Delta^+$};

\draw (8/3+0.3,-0.3) node{$8\rho/3$};
\draw (8/3,-0.0028) node{$|$};

\draw [very thick, -latex] (0,0) -- (2,0);
\end{tikzpicture}
\end{center}

The barycenter $\mathrm{bar}_{DH}(\Delta^+)$ is computed as 
\[
\mathrm{bar}_{DH}(\Delta^+) = 
\frac{\int_0^4 x^2dx}{\int_0^4xdx}\rho = \frac{8}{3} \rho.
\]
It of course satisfies the K-stability condition:
$8\rho /3 \in \mathrm{Relint}(2\rho + \mathbb{R}_+\rho)$. 
Remark that in this case, there is only one 
possible central fiber different from $\mathbb{P}^1\times \mathbb{P}^1$, 
it is the horospherical variety $S_2$ in the notations of 
\cite[Example 1.4.3]{AB06}. More explicitly, consider the rational 
ruled surface $\mathbb{F}_2 = 
\mathbb{P}(\mathcal{O}_{\mathbb{P}^1}\oplus \mathcal{O}_{\mathbb{P}^1}(2))$. 
This is an $\mathrm{SL}_2(\mathbb{C})$-horospherical variety with two closed 
orbits. One of these orbits is a curve with self-intersection $-2$, which 
may be contracted, and the image of the contraction is denoted by $S_2$.
The anticanonically polarized equivariant horospherical degeneration 
of $\mathbb{P}^1\times \mathbb{P}^1
\supset \mathrm{SL}_2(\mathbb{C})/\mathrm{SO}_2(\mathbb{C})$ is $S_2$.

The second one is $\mathbb{P}^2$, where $\mathrm{SL}_2(\mathbb{C})$ 
acts by the projectivization of its linear action on the quadratic forms 
in two variables. The open orbit is isomorphic to 
$\mathrm{SL}_2(\mathbb{C})/N_{\mathrm{SL}_2(\mathbb{C})}(\mathrm{SO}_2(\mathbb{C}))$.
Again, this manifold obviously satisfies the K-stability criterion. 
Its moment polytope, together with the barycenter 
$\mathrm{bar}_{DH}(\Delta^+) = 4 \rho$ are represented as follows.

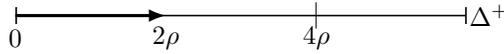
\begin{center}
\begin{tikzpicture}
\draw (0,-0.3) node{$0$};
\draw (2,-0.3) node{$2\rho$};

\draw [semithick, |-|] (0,0) -- (6,0);
\draw (6.3,0) node{$\Delta^+$};

\draw (4,-0.3) node{$4\rho$};
\draw (4,-0.0028) node{$|$};

\draw [very thick, -latex] (0,0) -- (2,0);
\end{tikzpicture}
\end{center}

The polarized equivariant horospherical degeneration of the 
symmetric variety $\mathbb{P}^2$ under this action of 
$\mathrm{SL}_2(\mathbb{C})$
is the variety $S_4$, constructed similarly as $S_2$ from 
$\mathbb{F}_4 = 
\mathbb{P}(\mathcal{O}_{\mathbb{P}^1}\oplus \mathcal{O}_{\mathbb{P}^1}(4))$
by contracting the curve with self intersection equal to $-4$.
\end{exa}

\begin{exa}
Consider the space of all conics in $\mathbb{P}^2$. 
It may be identified with $\mathbb{P}^5$ by identifying a conic 
with its equation, which has six coefficients. 
Consider the action of $\mathrm{SL}_3(\mathbb{C})$ on this space.
We see that it is a spherical variety with open orbit the orbit 
of nondegenerate conics, isomorphic to the symmetric space $
\mathrm{SL}_3(\mathbb{C})/N_{\mathrm{SL}_3(\mathbb{C})}(\mathrm{SO}_3(\mathbb{C}))
$. 

There is another smooth and Fano spherical embedding of the same 
symmetric space,
called the \emph{variety of complete conics}.
It may be constructed as the closure in $\mathbb{P}^5 \times (\mathbb{P}^5)^*$ 
of the set of couples $(C,C^*)$ where $C$ is a nondegenerate conic 
in $\mathbb{P}^2$ and $C^*$ is the dual conic, defined as the space of tangents 
to $C$. 
This construction was generalized by De Concini and Procesi \cite{DCP83}, 
to define \emph{wonderful} compactifications of symmetric spaces under 
an adjoint semisimple group, which are often Fano so provide numerous 
examples where Corollary~\ref{cor_sym} applies.

The moment polytope corresponding to the variety of complete conics 
is the following.

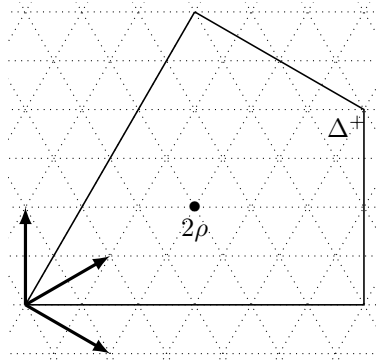
\begin{center}
\begin{tikzpicture}[scale=3/4]
\clip (-0.3,{-1.2*sqrt(3)/2}) rectangle (6.2,{6.2*sqrt(3)/2});

\newcommand*\rows{30}
    \foreach \row in {0, 1, ...,\rows} {
        \draw [dotted] ($({-\rows/2},{-sqrt(3)*\rows/6})+\row*(0.5, {0.5*sqrt(3)})$) -- ($({-\rows/2},{-sqrt(3)*\rows/6})+(\rows,0)+\row*(-0.5, {0.5*sqrt(3)})$);
        \draw [dotted] ($({-\rows/2},{-sqrt(3)*\rows/6})+\row*(1, 0)$) -- ($({-\rows/2+\rows/2},{-sqrt(3)*\rows/6+\rows/2*sqrt(3)})+\row*(0.5,{-0.5*sqrt(3)})$);
        \draw [dotted] ($({-\rows/2},{-sqrt(3)*\rows/6})+\row*(1, 0)$) -- ($({-\rows/2},{-sqrt(3)*\rows/6})+\row*(0.5,{0.5*sqrt(3)})$);
    }
    
\draw (3,{sqrt(3)}) node{$\bullet$};
\draw (3,{sqrt(3)-0.4}) node{$2\rho$};

\draw [very thick, -latex] (0,0) -- (0,{sqrt(3)});
\draw [very thick, -latex] (0,0) -- (3/2,{sqrt(3)/2});
\draw [very thick, -latex] (0,0) -- (3/2,{-sqrt(3)/2});

\draw [semithick] (6/2,{6*sqrt(3)/2}) -- (6,{2*sqrt(3)}) -- (6,0) -- (0,0) -- (6/2,{6*sqrt(3)/2});
\draw (5.7,{2*sqrt(3)-0.3}) node{$\Delta^+$};
\end{tikzpicture}
\end{center}

It is easy to check our K-stability criterion for this manifold, 
and we obtain that the variety of complete conics admits Kähler-Einstein 
metrics.
\end{exa}

\begin{exa}
A reductive group $\hat{G}=\hat{G}\times \hat{G}/ \mathrm{diag}(\hat{G})$ is a symmetric 
homogeneous space under the action of the group 
$G=\hat{G}\times \hat{G}$ for the involution $\sigma(g_1,g_2) = (g_2,g_1)$. 
If $\hat{B}\supset \hat{T}$ is a Borel subgroup of $\hat{G}$ containing 
a maximal torus $\hat{T}$, then 
$B=\hat{B}\times \hat{B}^-$ is an adapted Borel subgroup, 
$\Phi_{G}^{\sigma}$ is empty, $\mathfrak{X}(\hat{T}\times \hat{T}/\mathrm{diag}(\hat{T}))$
is the anti diagonal embedding of $\mathfrak{X}(\hat{T})$ and can thus be 
identified with $\mathfrak{X}(\hat{T})$ by projection to the first 
coordinate. Under this identification, 
$2\rho_{\sigma} = 2\rho_{B} \in \mathfrak{X}(\hat{T}\times \hat{T}/\mathrm{diag}(\hat{T}))$
is identified with $2\rho_{\hat{B}}$, and the restricted root system $\Phi_{\sigma}$
is identified with $\hat{\Phi}$ the root system of $\hat{G}$ with respect 
to $\hat{T}$.
We may identify, still under the projection to the first coordinate, 
the moment polytope $\Delta^+$ with the polytope in 
$\mathfrak{X}(T)\otimes \mathbb{R}$
as defined in \cite{DelTh,DelKE}, 
and the barycenter becomes 
\[
\mathrm{bar}_{DH}(\Delta^+)=
\int_{\Delta^+} p \prod_{\alpha \in \hat{\Phi}^+}\kappa (\alpha,p)^2 dp 
/\int_{\Delta^+} \prod_{\alpha \in \hat{\Phi}^+}\kappa (\alpha,p)^2 dp, 
\]
as used in \cite{DelTh,DelKE}. 

We then recover our previous result: a smooth and Fano group compactification $X$
with moment polytope $\Delta^+ \subset \mathfrak{X}(T)\otimes \mathbb{R}$ 
is Kähler-Einstein if and only if 
$\mathrm{bar}_{DH}(\Delta^+) \in 2\rho + \Xi$
where $\Xi$ is the relative interior of the cone generated by $\Phi^+$. 
The similar statement for Kähler-Ricci solitons is new, even though it 
may be obtained using the techniques of \cite{DelTh,DelKE}.

Let us take the opportunity to mention other interesting examples of group 
compactifications which were not mentioned in \cite{DelTh,DelKE}, 
where the only examples were toroidal compactifications of simple groups.  
First consider the smooth and Fano compactifications of the reductive
but not semisimple group $\mathrm{GL}_2(\mathbb{C})$. 
Using \cite{Ruz12}, we obtain the following 
list of eight moment polytopes (I thank Yan Li for pointing out that I forgot 
two polytopes in an earlier version of this article).

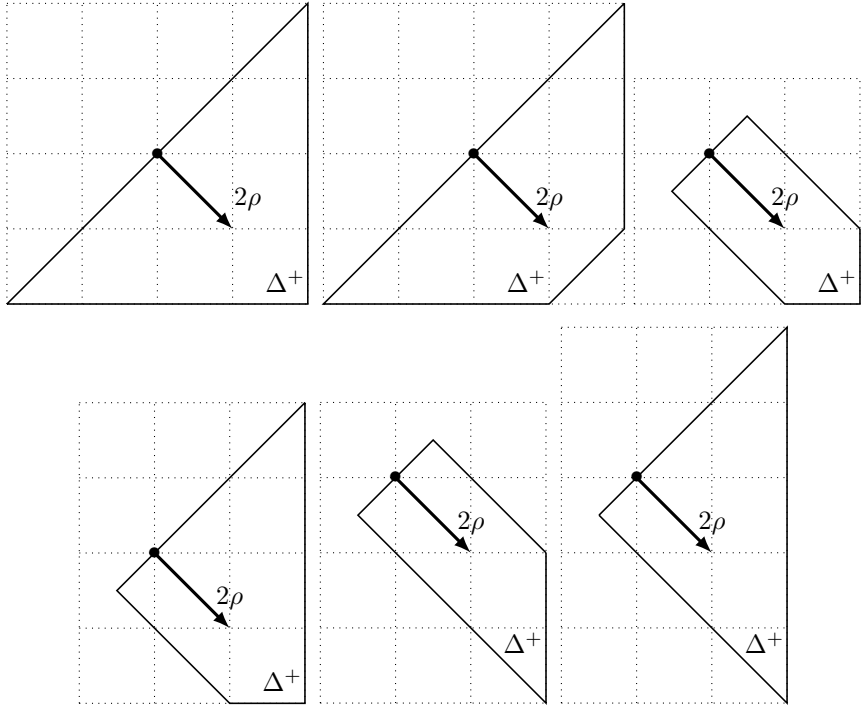
\begin{center}
\begin{tikzpicture}
\draw [dotted] (-2,-2) grid[xstep=1,ystep=1] (2,2);

\draw (0,0) node{$\bullet$};
\draw (1.2,-0.6) node{$2\rho$};

\draw [semithick] (-2,-2) -- (2,2) -- (2,-2) -- (-2,-2);
\draw (1.7,-1.7) node{$\Delta^+$};

\draw [very thick, -latex] (0,0) -- (1,-1);
\end{tikzpicture}
\begin{tikzpicture}
\draw [dotted] (-2,-2) grid[xstep=1,ystep=1] (2,2);

\draw (0,0) node{$\bullet$};
\draw (1,-0.6) node{$2\rho$};

\draw [semithick] (2,2) -- (2,-1) -- (1,-2) -- (-2,-2) -- (2,2);
\draw (0.7,-1.7) node{$\Delta^+$};

\draw [very thick, -latex] (0,0) -- (1,-1);
\end{tikzpicture}
\begin{tikzpicture}
\draw [dotted] (-1,-2) grid[xstep=1,ystep=1] (2,1);

\draw (0,0) node{$\bullet$};
\draw (1,-0.6) node{$2\rho$};

\draw [semithick] (2,-1) -- (2,-2) -- (1,-2) -- (-1/2,-1/2) -- (1/2,1/2) -- (2,-1);
\draw (1.7,-1.7) node{$\Delta^+$};

\draw [very thick, -latex] (0,0) -- (1,-1);
\end{tikzpicture}
\end{center}

\begin{center}
\begin{tikzpicture}
\draw [dotted] (-1,-2) grid[xstep=1,ystep=1] (2,2);

\draw (0,0) node{$\bullet$};
\draw (1,-0.6) node{$2\rho$};

\draw [semithick] (2,2) -- (2,-2) -- (1,-2) -- (-1/2,-1/2) -- (2,2);
\draw (1.7,-1.7) node{$\Delta^+$};

\draw [very thick, -latex] (0,0) -- (1,-1);
\end{tikzpicture}
\begin{tikzpicture}
\draw [dotted] (-1,-3) grid[xstep=1,ystep=1] (2,1);

\draw (0,0) node{$\bullet$};
\draw (1,-0.6) node{$2\rho$};

\draw [semithick] (2,-1) -- (2,-3) -- (-1/2,-1/2) -- (1/2,1/2) -- (2,-1);
\draw (1.7,-2.2) node{$\Delta^+$};

\draw [very thick, -latex] (0,0) -- (1,-1);
\end{tikzpicture}
\begin{tikzpicture}
\draw [dotted] (-1,-3) grid[xstep=1,ystep=1] (2,1);

\draw (0,0) node{$\bullet$};
\draw (1,-0.6) node{$2\rho$};

\draw [semithick] (-1/2,-1/2) -- (1/2,1/2) -- (1,0) -- (2,-2) -- (2,-3) -- (-1/2,-1/2);
\draw (1.7,-2.2) node{$\Delta^+$};
\draw [very thick, -latex] (0,0) -- (1,-1);
\end{tikzpicture}
\end{center}

\begin{center}
\begin{tikzpicture}
\draw [dotted] (-1,-4) grid[xstep=1,ystep=1] (3,1);

\draw (0,0) node{$\bullet$};
\draw (1,-0.6) node{$2\rho$};

\draw [semithick] (-1/2,-1/2) -- (1/2,1/2) -- (1,0) -- (3,-4) -- (-1/2,-1/2);
\draw (1.7,-2.2) node{$\Delta^+$};
\draw [very thick, -latex] (0,0) -- (1,-1);
\end{tikzpicture}
\begin{tikzpicture}
\draw [dotted] (-1,-3) grid[xstep=1,ystep=1] (2,2);

\draw (0,0) node{$\bullet$};
\draw (1,-0.6) node{$2\rho$};

\draw [semithick] (-1/2,-1/2) -- (2,2) -- (2,-3) -- (-1/2,-1/2);
\draw (1.7,-2.2) node{$\Delta^+$};
\draw [very thick, -latex] (0,0) -- (1,-1);
\end{tikzpicture}
\end{center}

It is obvious that the last five do not admit Kähler-Einstein metrics. 
Computations show that the first three admit Kähler-Einstein metrics. 
We may also approximate, numerically, the unique 
holomorphic vector field such 
that the modified Futaki invariant vanishes (see \cite{TZ02}), 
then check numerically that the 
modified K-stability condition is satisfied.  

As an example, the precise coordinates of the barycenter 
$\mathrm{bar}_{DH}(\Delta^+)$ 
in the case of the third polytope are $(2343/1750,-2343/1750)$.
Remark that only the third and the fifth polytopes correspond to toroidal 
compactifications of $\mathrm{GL}_2(\mathbb{C})$.

Considering colored compactifications also allows to obtain smaller 
dimensional examples of group compactifications with vanishing Futaki 
invariant which are not K-semistable. Namely, the three smooth and Fano 
compactifications of the (semisimple but not simple) group 
$\mathrm{SO}_4(\mathbb{C})$, of dimension six, 
are colored, with the following moment polytopes:

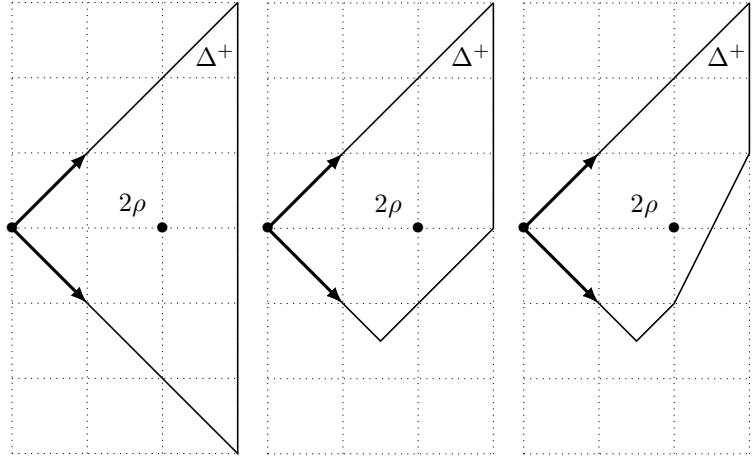
\begin{center}
\begin{tikzpicture}
\draw [dotted] (0,-3) grid[xstep=1,ystep=1] (3,3);

\draw (0,0) node{$\bullet$};
\draw (2,0) node{$\bullet$};
\draw (1.6,0.3) node{$2\rho$};

\draw [semithick] (3,3) -- (0,0) -- (3,-3) -- (3,3);
\draw (2.7,2.3) node{$\Delta^+$};

\draw [very thick, -latex] (0,0) -- (1,-1);
\draw [very thick, -latex] (0,0) -- (1,1);
\end{tikzpicture}
\begin{tikzpicture}
\draw [dotted] (0,-3) grid[xstep=1,ystep=1] (3,3);

\draw (0,0) node{$\bullet$};
\draw (2,0) node{$\bullet$};
\draw (1.6,0.3) node{$2\rho$};

\draw [semithick] (3,3) -- (3,0) -- (3/2,-3/2) -- (0,0) -- (3,3);
\draw (2.7,2.3) node{$\Delta^+$};

\draw [very thick, -latex] (0,0) -- (1,-1);
\draw [very thick, -latex] (0,0) -- (1,1);
\end{tikzpicture}
\begin{tikzpicture}
\draw [dotted] (0,-3) grid[xstep=1,ystep=1] (3,3);

\draw (0,0) node{$\bullet$};
\draw (2,0) node{$\bullet$};
\draw (1.6,0.3) node{$2\rho$};

\draw [semithick] (3,3) -- (3,1) -- (2,-1) -- (3/2,-3/2) --(0,0) -- (3,3);
\draw (2.7,2.3) node{$\Delta^+$};

\draw [very thick, -latex] (0,0) -- (1,-1);
\draw [very thick, -latex] (0,0) -- (1,1);
\end{tikzpicture}
\end{center}

While the first example satisfies the K-stability criterion, both the 
two others do not satisfy the K-semistability criterion, while having 
a vanishing Futaki invariant since their groups of equivariant 
automorphisms are finite. Hence they do not admit any Kähler-Ricci solitons.
It would be interesting to know their full 
automorphism group and in particular if it is reductive. 
\end{exa}

\begin{exa}
There are examples of smooth and Fano symmetric varieties with Picard 
number one, which are not homogeneous under a larger group \cite{Ruz10,Ruz11}. 
Unlike the similar horospherical examples of Pasquier, these may very 
well be K-stable. For example, the smooth and Fano spherical embedding 
of $\mathrm{SL}_3(\mathbb{C})/\mathrm{SO}_3(\mathbb{C})$ with Picard number 
one, the smooth and Fano biequivariant compactifiation of the group $G_2$ 
with Picard number one, and the smooth and Fano biequivariant compactification 
of the group $\mathrm{SL}_3(\mathbb{C})$ admit Kähler-Einstein metrics.
Their respective moment polytopes are as follows. 

\begin{center}
\begin{tikzpicture}[scale=3/4]
\clip (-0.3,{-1.2*sqrt(3)/2}) rectangle (6.1,{6.2*sqrt(3)/2});

\newcommand*\rows{24}
    \foreach \row in {0, 1, ...,\rows} {
        \draw [dotted] ($({-\rows/2},{-sqrt(3)*\rows/6})+\row*(0.5, {0.5*sqrt(3)})$) -- ($({-\rows/2},{-sqrt(3)*\rows/6})+(\rows,0)+\row*(-0.5, {0.5*sqrt(3)})$);
        \draw [dotted] ($({-\rows/2},{-sqrt(3)*\rows/6})+\row*(1, 0)$) -- ($({-\rows/2+\rows/2},{-sqrt(3)*\rows/6+\rows/2*sqrt(3)})+\row*(0.5,{-0.5*sqrt(3)})$);
        \draw [dotted] ($({-\rows/2},{-sqrt(3)*\rows/6})+\row*(1, 0)$) -- ($({-\rows/2},{-sqrt(3)*\rows/6})+\row*(0.5,{0.5*sqrt(3)})$);
    }
    
\draw (3,{sqrt(3)}) node{$\bullet$};
\draw (3,{sqrt(3)-0.4}) node{$2\rho$};

\draw [very thick, -latex] (0,0) -- (0,{sqrt(3)});
\draw [very thick, -latex] (0,0) -- (3/2,{sqrt(3)/2});
\draw [very thick, -latex] (0,0) -- (3/2,{-sqrt(3)/2});

\draw [semithick] (6/2,{6*sqrt(3)/2}) -- (6,0) -- (0,0) -- (6/2,{6*sqrt(3)/2});
\end{tikzpicture}
\begin{tikzpicture}[scale=3/4]
\clip (-2,-0.3) rectangle (4,6.5);

\newcommand*\rows{33}
    \foreach \row in {0, 1, ...,\rows} {
        \draw [dotted] ($({-\rows/2},{-sqrt(3)*\rows/6})+\row*(0.5, {0.5*sqrt(3)})$) -- ($({-\rows/2},{-sqrt(3)*\rows/6})+(\rows,0)+\row*(-0.5, {0.5*sqrt(3)})$);
        \draw [dotted] ($({-\rows/2},{-sqrt(3)*\rows/6})+\row*(1, 0)$) -- ($({-\rows/2+\rows/2},{-sqrt(3)*\rows/6+\rows/2*sqrt(3)})+\row*(0.5,{-0.5*sqrt(3)})$);
        \draw [dotted] ($({-\rows/2},{-sqrt(3)*\rows/6})+\row*(1, 0)$) -- ($({-\rows/2},{-sqrt(3)*\rows/6})+\row*(0.5,{0.5*sqrt(3)})$);
    }
    
\draw (1,{3*sqrt(3)}) node{$\bullet$};
\draw (1,{3*sqrt(3)-0.4}) node{$2\rho$};

\draw [very thick, -latex] (0,0) -- (1,0);
\draw [very thick, -latex] (0,0) -- (1/2,{sqrt(3)/2});
\draw [very thick, -latex] (0,0) -- (-1/2,{sqrt(3)/2});
\draw [very thick, -latex] (0,0) -- (0,{sqrt(3)});
\draw [very thick, -latex] (0,0) -- (3/2,{sqrt(3)/2});
\draw [very thick, -latex] (0,0) -- (-3/2,{sqrt(3)/2});

\draw [semithick] (7/2,{7*sqrt(3)/2}) -- (0,0) -- (0,{7*sqrt(3)/2}) 
-- (7/2,{7*sqrt(3)/2});
\end{tikzpicture}
\end{center}

\begin{center}
\begin{tikzpicture}[scale=3/4]
\clip (-0.3,{-1.2*sqrt(3)/2}) rectangle (5.1,{5.2*sqrt(3)/2});

\newcommand*\rows{24}
    \foreach \row in {0, 1, ...,\rows} {
        \draw [dotted] ($({-\rows/2},{-sqrt(3)*\rows/6})+\row*(0.5, {0.5*sqrt(3)})$) -- ($({-\rows/2},{-sqrt(3)*\rows/6})+(\rows,0)+\row*(-0.5, {0.5*sqrt(3)})$);
        \draw [dotted] ($({-\rows/2},{-sqrt(3)*\rows/6})+\row*(1, 0)$) -- ($({-\rows/2+\rows/2},{-sqrt(3)*\rows/6+\rows/2*sqrt(3)})+\row*(0.5,{-0.5*sqrt(3)})$);
        \draw [dotted] ($({-\rows/2},{-sqrt(3)*\rows/6})+\row*(1, 0)$) -- ($({-\rows/2},{-sqrt(3)*\rows/6})+\row*(0.5,{0.5*sqrt(3)})$);
    }
    
\draw (3,{sqrt(3)}) node{$\bullet$};
\draw (3,{sqrt(3)-0.4}) node{$2\rho$};

\draw [very thick, -latex] (0,0) -- (0,{sqrt(3)});
\draw [very thick, -latex] (0,0) -- (3/2,{sqrt(3)/2});
\draw [very thick, -latex] (0,0) -- (3/2,{-sqrt(3)/2});

\draw [semithick] (5/2,{5*sqrt(3)/2}) -- (5,0) -- (0,0) -- (5/2,{5*sqrt(3)/2});
\end{tikzpicture}
\end{center}
\end{exa}

\bibliographystyle{alpha}
\bibliography{KSSV}

\end{document}